\documentclass{amsart}

\usepackage{graphicx,caption,subcaption}

\usepackage{amssymb,amsmath,amsfonts,amsthm,amscd,amstext,mathtools}
\usepackage{epstopdf}
\usepackage{microtype}
\usepackage{mathrsfs}
\usepackage{stmaryrd}	
\usepackage{enumitem}	
\setlist[1]{leftmargin=2\parindent,itemsep=2pt}
\usepackage{verbatim} 
\usepackage{url}
\usepackage{hyperref}
\usepackage{xcolor}
\usepackage{tikz}

\newcommand{\ind}{{\sf 1}}

\newcommand{\bP}{\mathbf{P}}

\newcommand{\bZ}{\mathbf{Z}}
\newcommand{\bE}{\mathbf{E}}
\newcommand{\bbP}{\mathbb{P}}

\newcommand{\bbE}{\mathbb{E}}

\newcommand{\bbR}{\mathbb{R}}
\newcommand{\R}{\mathbb{R}}
\newcommand{\bbN}{\mathbb{N}}
\newcommand{\N}{\mathbb{N}}
\newcommand{\bbZ}{\mathbb{Z}}

\newcommand{\Var}{\mathbb{V}\mathrm{ar}}

\newcommand{\bT}{\mathbf{T}}

\newcommand{\cA}{{\ensuremath{\mathcal A}} }
\newcommand{\cF}{{\ensuremath{\mathcal F}} }
\newcommand{\cP}{{\ensuremath{\mathcal P}} }
\newcommand{\cE}{{\ensuremath{\mathcal E}} }

\newcommand{\cC}{{\ensuremath{\mathcal C}} }

\newcommand{\cL}{{\ensuremath{\mathcal L}} }
\newcommand{\cJ}{{\ensuremath{\mathcal J}} }

\newcommand{\cD}{{\ensuremath{\mathcal D}} }
\newcommand{\cW}{{\ensuremath{\mathcal W}} }

\newcommand{\cI}{{\ensuremath{\mathcal I}} }

\newcommand{\cG}{{\ensuremath{\mathcal G}} }
\newcommand{\cM}{{\ensuremath{\mathcal M}} }
\newcommand{\cK}{{\ensuremath{\mathcal K}} }
\renewcommand{\epsilon}{\varepsilon}
\renewcommand{\phi}{\varphi}

\newcommand{\ga}{\alpha}
\newcommand{\gb}{\beta}
\newcommand{\gd}{\delta}
\newcommand{\gep}{\varepsilon}  
\newcommand{\gp}{\varphi}

\newcommand{\gz}{\zeta}

\newcommand{\gGa}{\Gamma}

\newcommand{\go}{\omega}

\newcommand{\gO}{\Omega}
\newcommand{\gl}{\lambda}

\newcommand{\gs}{\sigma}

\newcommand{\gh}{\eta}

\newcommand{\tf}{\mathtt{F}}

\newcommand{\btau}{\boldsymbol{\tau}}

\newcommand{\bsigma}{\boldsymbol{\sigma}}

\newcommand{\bnu}{\boldsymbol{\nu}}
\newcommand{\blamb}{\boldsymbol{\lambda}}

\newcommand{\biota}{\boldsymbol{\iota}}

\newcommand{\bi}{\textbf{\textit{i}}}
\newcommand{\bj}{\textbf{\textit{j}}}
\newcommand{\bn}{\textbf{\textit{n}}}

\newcommand{\bt}{\textbf{\textit{t}}}
\newcommand{\ba}{\textbf{\textit{a}}}
\newcommand{\bA}{\textbf{\textit{A}}}
\newcommand{\bb}{\textbf{\textit{b}}}

\newcommand{\bJ}{\textbf{\textit{J}}}
\newcommand{\bX}{\textbf{\textit{X}}}

\newcounter{cst}[section]		
\newcounter{svf}[section]		
\newcommand{\cntc}{{\stepcounter{cst}\arabic{cst}}}		

\newtheorem{theorem}{Theorem}[section]
\newtheorem{remark}[theorem]{Remark}
\newtheorem{definition}[theorem]{Definition}
\newtheorem*{definition*}{Definition}

\newtheorem*{notation*}{Notation}

\newtheorem{proposition}[theorem]{Proposition}
\newtheorem{corollary}[theorem]{Corollary}
\newtheorem{lemma}[theorem]{Lemma}
\newtheorem{claim}[theorem]{Claim}
\newtheorem{example}[theorem]{Example}

\numberwithin{equation}{section}			

\newcommand{\dd}{\mathrm{d}}		

\newcommand{\sumtwo}[2]{\sum_{\substack{#1 \\ #2}}}
\newcommand{\sumthree}[3]{\sum_{\substack{#1 \\ #2 \\ #3}}}

\newcommand{\limtwo}[2]{\lim_{\substack{#1 \\ #2}}}
\renewcommand{\preceq}{\preccurlyeq}		
\renewcommand{\succeq}{\succcurlyeq}		
\newcommand{\bone}{{\boldsymbol{1}}}
\newcommand{\bzero}{{\boldsymbol{0}}}

\renewcommand{\hat}{\widehat}
\renewcommand{\tilde}{\widetilde}
\newcommand{\ol}{\overline}

\newcommand{\free}{\ensuremath{\mathrm{free}}}
\newcommand{\quen}{\ensuremath{\mathrm{q}}}
\newcommand{\cond}{\ensuremath{\mathrm{cond}}}
\newcommand{\ann}{\ensuremath{\mathrm{a}}}
\newcommand{\Pfk}{\ensuremath{\mathfrak{P}}}
\newcommand{\pt}{\hspace{1pt}}
\newcommand{\bs}{\textbf{\textit{s}}}
\newcommand{\bu}{\textbf{\textit{u}}}
\newcommand{\bv}{\textbf{\textit{v}}}
\newcommand{\lrarw}{\leftrightarrow}
\newcommand{\aligne}{\leftrightarrow}
\newcommand{\nlrarw}{\nleftrightarrow}

\newcommand{\llb}{\llbracket}
\newcommand{\rrb}{\rrbracket}
\newcommand{\bI}{\textbf{\textit{I}}}
\newcommand{\cS}{{\ensuremath{\mathcal S}}}
\newcommand{\cR}{{\ensuremath{\mathcal R}}}
\newcommand{\la}{\langle}
\newcommand{\ra}{\rangle}

\renewcommand{\hat}{\widehat}
\renewcommand{\tilde}{\widetilde}

\definecolor{dgreen}{RGB}{30,140,60}

\newcommand{\Bor}{\mathrm{Bor}}
\newcommand{\bw}{{\boldsymbol w}}

\newcommand{\bbJ}{\mathbb{J}}

\title[Scaling limit of the disordered gPS model for DNA denaturation]{Scaling limit of the disordered\\
generalized Poland--Scheraga model for DNA denaturation}

\author[Q. Berger]{Quentin Berger}
\address{Sorbonne Universit\'e, LPSM, UMR 8001. Campus Pierre et Marie Curie, 4 pl.~Jussieu, case 158, 75252 Paris Cedex 5, France\\ 
and DMA, \'Ecole Normale Sup\'erieure, Universit\'e PSL, 75005
Paris, France\\
and Institut Universitaire de France}
\email{quentin.berger@sorbonne-universite.fr}

\author[A. Legrand]{Alexandre Legrand} 
\address{Institut de Math\'ematiques de Toulouse, Universit\'e Toulouse 3 Paul Sabatier,
118 Route de Narbonne, 31062 Toulouse Cedex 9, France}
\email{alegrand@math.univ-toulouse.fr}


\makeatletter
\@namedef{subjclassname@2020}{%
	2020 Mathematics Subject Classification}
\makeatother

\date{}

\begin{document}

\begin{abstract}
The Poland--Scheraga model, introduced in the 1970's, is a reference model to describe the denaturation transition of DNA. More recently, it has been generalized in order to allow for asymmetry in the strands lengths and in the formation of loops: the mathematical representation is based on a bivariate renewal process, that describes the pairs of bases that bond together.
In this paper, we consider a disordered version of the model, in which the two strands interact via a potential $\beta V(\hat \omega_i, \bar \omega_j) +h$ when the $i$-th monomer of the first strand and the $j$-th monomer of the second strand meet. Here, $h\in\mathbb R$ is a homogeneous pinning parameter, $(\hat\omega_i)_{i\geq 1}$ and $(\bar \omega_j)_{j\geq 1}$ are two sequences of i.i.d.\ random variables attached to each DNA strand, $V(\cdot,\cdot)$ is an interaction function and $\beta>0$ is the disorder intensity.
Our main result finds some condition on the underlying bivariate renewal so that, if one takes $\beta,h\downarrow 0$ at some appropriate (explicit) rate as the length of the strands go to infinity, the partition function of the model admits a non-trivial, \textit{i.e.}\ \textit{disordered}, scaling limit. This is known as an \emph{intermediate disorder} regime and is linked to the question of disorder relevance for the denaturation transition.
Interestingly, and unlike any other model of our knowledge, the rate at which one has to take $\beta\downarrow 0$ depends on the interaction function $V(\cdot,\cdot)$ and on the distribution of $(\hat\omega_i)_{i\geq 1}$, $(\bar \omega_j)_{j\geq 1}$. On the other hand, the intermediate disorder  limit of the partition function, when it exists, is universal: it is expressed as a chaos expansion of iterated integrals against a Gaussian process~$\mathcal M$, which arises as the scaling limit of the field $(e^{\beta V(\hat \omega_i, \bar \omega_j)})_{i,j\geq 0}$ and exhibits strong correlations on lines and columns.
\end{abstract}

\maketitle

\begin{figure}[htbp]
\begin{center}
\begin{tabular}{cc}
\begin{tabular}{c}
\includegraphics[scale=0.4]{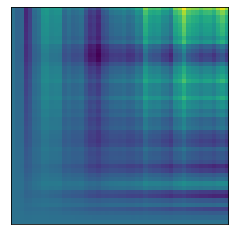}
\end{tabular}
\qquad
& 
\qquad
\begin{tabular}{c}
\includegraphics[scale=0.4]{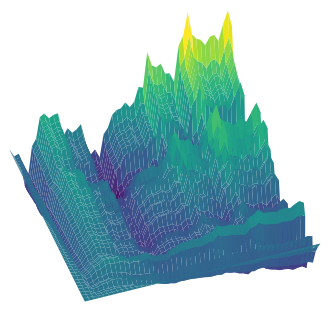}
\end{tabular}
\end{tabular}
\end{center}
\vspace{-0.7\baselineskip}
\captionsetup{width=.9\textwidth}
\caption{\footnotesize A realization of the (non-isotropic) Gaussian process $\cM$ appearing in the disordered scaling limit of the generalized Poland--Scheraga model. The field $(V(\hat \go_i, \bar \go_j))_{i,j\geq 1}$ presents correlations along rows and columns: these correlations appear in the limiting process $\cM$ and can be observed in the figure above.}
\label{fig:M}
\end{figure}

\setcounter{tocdepth}{1}
\tableofcontents


\section{Introduction}

\subsection{The Poland--Scheraga model and the question of disorder relevance}

The Poland--Scheraga  (PS) model~\cite{PS70} has
been introduced in the 1970's 
in order to describe the denaturation transition of DNA.
Since then, it has been widely studied in the bio-physics and mathematical literature,
both from a theoretical perspective, see~\cite{Fish84,Giac07,Giac10},
and an experimental one, see e.g.~\cite{BBBDDKMS99,BD98}.
The model is based on a renewal process that describes the pairs of bases that bind together and it can naturally embed the inhomogeneous character of the interactions between the bases
(see Figure~\ref{fig:PS}).
More generally, the model is known as the \emph{pinning model},
which is also used to describe the behaviour of one-dimensional interfaces (or polymers) interacting with a defect line.
The inhomogeneity of the interactions along the DNA strands is usually modelled thanks to a sequence of random variables, often dubbed as \textit{disorder},
that represent the different values of the binding potentials along the polymer; in the context of pinning models, this reduces to considering an inhomogeneous (disordered) defect line.
 
One remarkable feature of the pinning model is that its homogeneous version, 
that is when the binding potentials are all equals, is solvable.
One can show that the model exhibits a depinning (or denaturation) transition
and one can identify the critical temperature and the behaviour of the free energy when approaching the critical temperature, see~\cite[Ch.~4]{Giac07}.

\subsubsection*{Disorder relevance}

A natural question is then to know whether disorder changes the characteristics of the phase transition: in other words, can we determine if (and how) the critical temperature and the critical behaviour of the free energy is affected by the presence of inhomogeneities in the binding interactions?
This is the general question of \emph{disorder relevance} for physical systems: if an arbitrarily small amount of disorder changes the characteristics of the phase transition, then disorder is called \emph{relevant}; otherwise disorder is called \textit{irrelevant}.
In a celebrated paper, the physicist Harris~\cite{H74} proposed a general criterion,
based on the critical behaviour of the homogeneous (or pure) system---more specifically on the correlation length critical exponent $\nu$---,
to predict whether an i.i.d.\ disorder is relevant or not:
for a $d$-dimensional physical system, 
disorder should be irrelevant if $\nu >2/d$ and relevant if $\nu<2/d$; the case $\nu=2/d$,
called marginal, requires more investigation.

The pinning model has seen an intense activity over the past decades,
both in theoretical physics (see e.g.\ 
\cite{CH97,DHV92,DR14,FLNO86,KL12,TC01}
to cite a few) and in rigorous mathematical physics (see e.g.\ 
\cite{Alex08,AZ09,BCPSZ14,BL18,CdH13,DGLT09,
GLT10,GLT11,GT06,GT09,Lac10ECP,Ton08_AAP,Ton08_CMP}).
One reason for that activity comes from the fact that the homogeneous model is exactly solvable and displays a critical exponent $\nu$ that ranges from~$1$ to $+\infty$:
the disordered pinning model has therefore been an ideal framework to test the validity of Harris' predictions.
The Harris criterion has now been put on rigorous ground by a series of works (see \cite{Alex08,AZ09,CdH13,DGLT09,GT06,
GT09,Lac10ECP,Ton08_AAP,Ton08_CMP}), the marginal case being also completely settled (see \cite{GLT10,GLT11} and \cite{BL18}), after some contradictory predictions in the physics literature \cite{DHV92,FLNO86}.

\subsubsection*{Intermediate disorder regime}

A recent and complementary approach to the question of disorder relevance has been to consider scaling limits of the model, see \cite{CSZ16} for an overview.
In this context, disorder relevance can be understood
as the possibility of tuning down the intensity of disorder as the system size grows in such a way that disorder \emph{is still present} in the limit.
The idea is therefore to scale the different parameters of the model with the size of the system,
in such a way to obtain a non-trivial, \textit{i.e.}\ \emph{disordered}, scaling limit.
This is called the \emph{intermediate disorder}
regime, which corresponds to identifying a scaling window for the disorder intensity
in which one observes a transition from a ``weak disorder'' phase to a ``strong disorder'' phase.

This approach has first been implemented in the context of the directed polymer model in dimension~$1+1$, see~\cite{AKQ14}, and has been widened in \cite{CSZ13} to other models (including the pinning model); let us also mention~\cite{BL21_scaling,SSSX21} for other results in the same spirit.
In particular, let us stress that in~\cite{CSZ13}, the conditions
for having a non-trivial scaling limit of the model exactly matches that of Harris' condition for disorder relevance.

The intermediate disorder scaling limit
seems to have wide applications for understanding
relevant disorder systems. For instance:
it makes it possible to extract universal behaviours of 
quantities of interest such as the critical point shift
or the free energy of the model
see~\cite{CTT17,Nak19};
it is a way to construct continuum disordered systems
that arise as scaling limits of discrete models (and encapsulate their universal features),
see~\cite{AKQ14_cont,BL20_conti,BS21,CSZ16_conti}.
Let us also mention that, in the case of marginally relevant disordered systems,
understanding the intermediate disorder scaling limit is much more challenging, see~\cite{CSZ17}.
However, in the context of the directed polymer in dimension $2+1$ it provides a way to make sense of (and study) the ill-defined stochastic heat equation,  see the recent paper~\cite{CSZ21}.

\subsubsection*{Generalization of the Poland--Scheraga model}

The Poland--Scheraga model, thanks to its simplicity, plays a central role in the study of DNA denaturation. But some aspects of it are oversimplified and fail to capture important features of the model: in particular, the two DNA strands are assumed to be of equal length, and loops have to be symmetric, ruling out for instance the existence of mismatches (see Figure~\ref{fig:PS}). 
For these reasons, Garel and Orland~\cite{GO04} (see also~\cite{NG06}) introduced a generalization of the model that overcomes these two limitations: loops are allowed to be asymmetric and the two strands are allowed to be of different lengths (see Figure~\ref{fig:gPS}).

\begin{figure}[hbtp]
\vspace{-0.7\baselineskip}
	\centering
	\begin{subfigure}[t]{0.45\textwidth}
		\centering
		\includegraphics[width=1\textwidth]{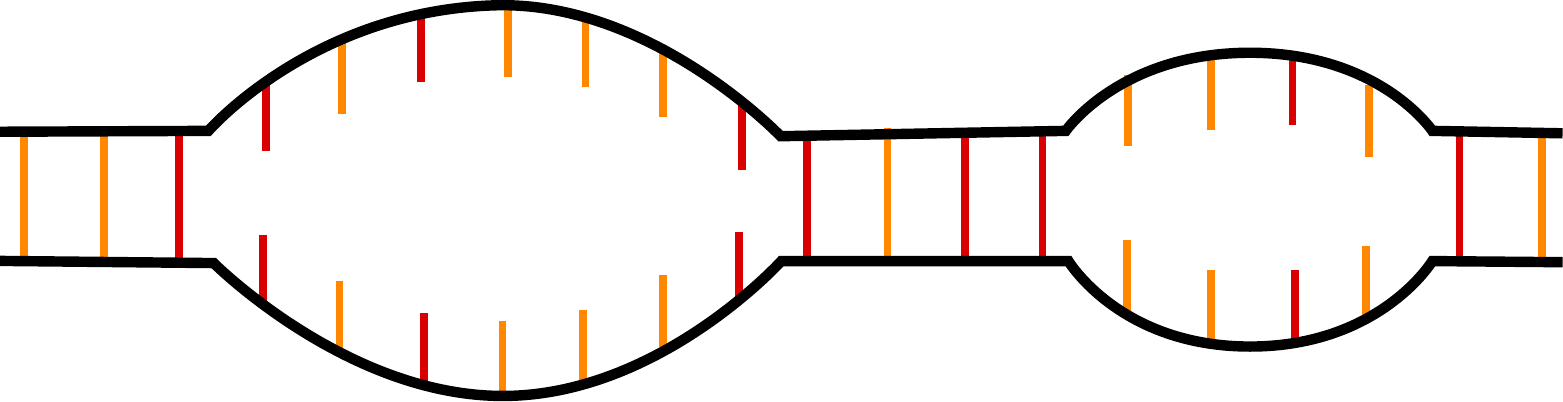}
		\caption{\footnotesize Standard PS model.}\label{fig:PS}		
	\end{subfigure}
	\qquad
	\begin{subfigure}[t]{0.45\textwidth}
		\centering
		\includegraphics[width=1\textwidth]{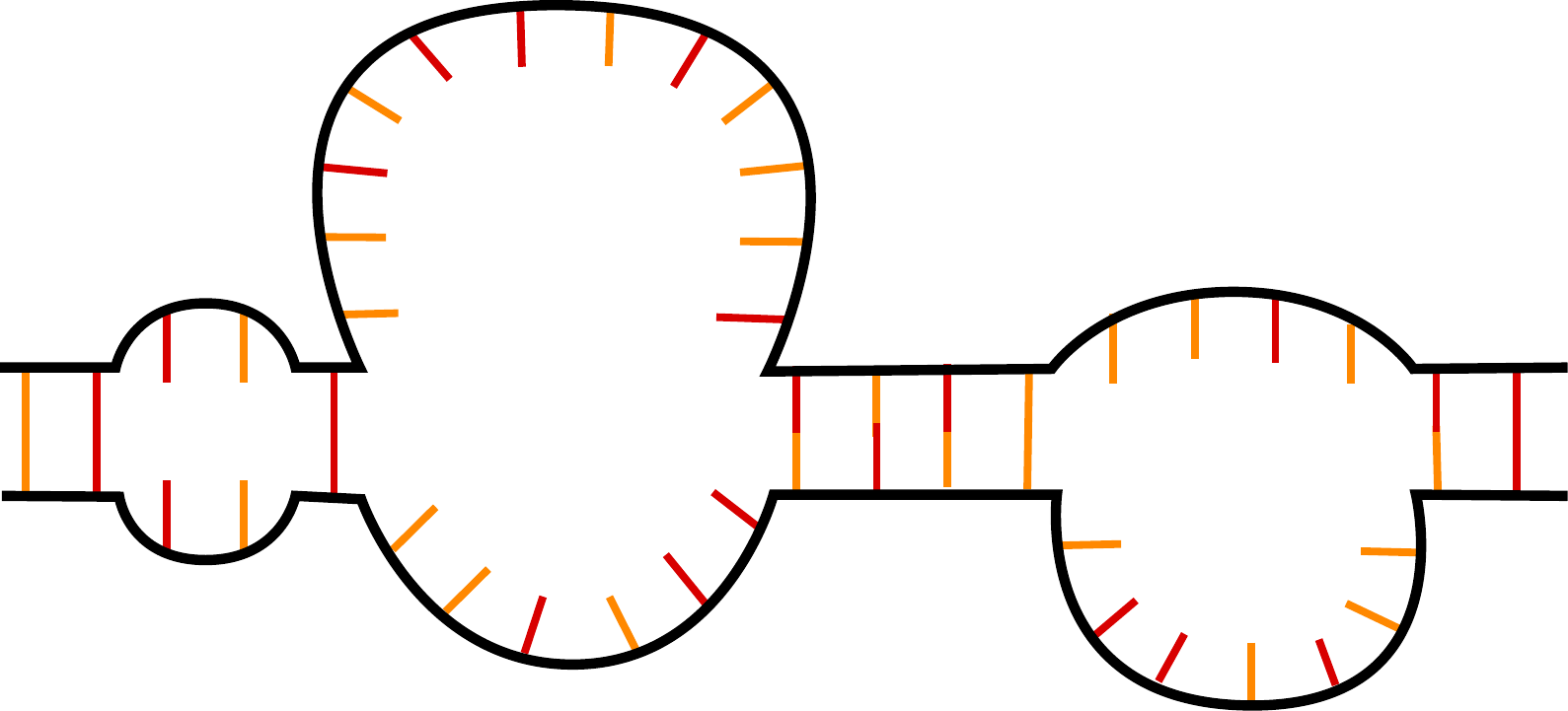}
		\caption{\footnotesize Generalized PS model.}\label{fig:gPS}
	\end{subfigure}
	\caption{\small Standard vs.\ generalized Poland-Scheraga models, with two types of monomers along the strands. The standard PS model is represented on the left: the two strands have the same length and loops are symmetric (there is no mismatch). The configuration is encoded by the sequence of lengths of the successive loops, from left to right (in the example~{\small(A)}: $1,1,8,1,1,1,5,1$). The generalized PS model 
is represented on the right: the two strands may have 
different lengths and loops are allowed to be asymmetric. (mismatches occur).
A loop is encoded by a pair $(k,\ell)$, with $k$ the length of the 'top' strand and $\ell$ the length of the 'bottom' strand in that loop:  a configuration is encoded by the sequence of pairs describing the successive loops, from left to right (in the example~{\small(B)}: $(1,1),(3,3),(13,7),(1,1),(1,1),(1,1),(5,8),(1,1)$).}\label{fig:1}
\end{figure}

The mathematical formulation of this generalized Poland--Scheraga (gPS) model has been developed by Giacomin and Khatib~\cite{GK17},
and is based on a bivariate renewal process, \textit{i.e.} a renewal process on $\N^2$,
whose increments describe the successive loops in the DNA 
(an increment $(k,\ell)$ describes a loop with length~$k$ in the first strand and length $\ell$ in the second strand, see Figure~\ref{fig:gPS}).
In \cite{GK17}, the authors consider only the homogeneous gPS model: somehow surprisingly they find that the model is also solvable, but with a much richer phenomenology than in the PS model.
In particular, in addition to the denaturation transition,
other phase transitions, called \textit{condensation} transitions, may occur;
this was first observed in~\cite{NG06}.
The critical points for the denaturation and condensation transitions can be identified (see~\cite{GK17}). Moreover, the critical behaviour of the denaturation transition
has been described in~\cite{GK17}
and the condensation transitions have been further investigated in~\cite{BGK18} (it corresponds to a \textit{big-jump} transition for the bivariate renewal process).

As far as the disordered version of the gPS model is concerned,
this has been investigated in the physics literature,
but only at a numerical level, see e.g.~\cite{GO04,TN08}.
In \cite{BGK20}, the authors consider a disordered version 
of the model, in which a pairing between the $i$-th monomer
of the first strand and the $j$-th monomer of the second strand
is associated with a disorder variable $\go_{i,j}$, where $(\go_{i,j})_{i,j\geq 0}$ are i.i.d.\ random variables. 
In that case, Harris' predictions for disorder relevance on the denaturation phase transition
have been confirmed in~\cite{BGK20}: if $\nu$ is the critical exponent for the free energy in the pure model, disorder is irrelevant as soon as $\nu>1$ (here, the dimension of the disorder field is $d=2$).

However, this choice of an i.i.d.\ disorder is not necessarily adapted to the modelling of DNA. A more faithful choice would be to consider two sequences $(\hat \go_i)_{i\geq 0}, (\bar \go_j)_{j\geq 0}$ of (i.i.d.) random variables which represent the nucleotides sequences along the two DNA strands; and to define the disorder field with some base pairing function, i.e.\ $\go_{i,j}\coloneqq V(\hat\go_i,\bar\go_j)$, which describes the chemical bonding interactions between the nucleotides of the two strands. In~\cite{Leg21}, the author considers the case where $\go_{i,j} \coloneqq \hat \go_i \bar \go_j$
as a natural toy model, for which computations are more explicit.
A striking finding of~\cite{Leg21} is that, in that case, disorder relevance
depends on the distribution of $\hat \go_i,\bar \go_j$:
for ``most'' distributions, disorder is irrelevant if $\nu>2$ and relevant if $\nu<2$ (the disorder is fundamentally one-dimensional);
on the other hand, there are distributions, namely $\frac12 (\delta_{-x} +\delta_x)$ for some $x>0$, such that
disorder is irrelevant as soon as $\nu>1$ (the disorder is essentially two-dimensional).

\subsubsection*{Intermediate disorder scaling limit for the gPS model}

One of the goal of the present paper is to complement 
the existing results on the influence of disorder on the denaturation transition for the gPS model, when the disorder comes from a base pairing interaction between inhomogeneous strands.
For this purpose, we investigate the intermediate disorder scaling limit of the model, in the spirit of~\cite{CSZ13, CSZ16}.
We extend the results of~\cite{Leg21} in several ways:
\begin{itemize}
\item We consider a more general disorder variable $\go_{i,j}$ associated to the pairing
of the $i$-th monomer
of the first strand and the $j$-th monomer of the second strand:
we take $\go_{i,j} = V(\hat \go_i,\bar \go_j)$
where $(\hat \go_i)_{i\geq 0}$, $(\bar \go_j)_{j\geq 0}$ are sequences of i.i.d.\ random variables and $V(\cdot,\cdot)$ is any (symmetric) interaction function.
\item We identify the correct intermediate disorder scaling and we prove the convergence of the partition function towards a non-trivial limit under that scaling.
Remarkably, the scaling depends finely on the
distribution of $\go_{i,j}$, \textit{i.e.}\ on the function $V(\cdot,\cdot)$ and on the distribution of $\hat \go_i,\bar\go_j$.
\item The identification of the intermediate disorder scaling
allows us to give a sufficient condition for disorder relevance (we determine whether the effective dimension of the disorder is one or two). It also enables us to obtain sharp bounds on the critical point shift, improving some results of~\cite{Leg21}.
\end{itemize}

One of the main novelties of the present paper
is that, to the best of our knowledge, it is the first instance where 
the intermediate disorder scaling depends on the distribution of disorder, therefore displaying some non-universality feature.
On the other hand, the limit of the partition function 
under the intermediate disorder scaling is universal, in the sense that it does not depend on the distribution of the disorder $\go_{i,j}$
or on the fine details of the underlying bivariate renewal.

One major difficulty of the present work is that the disorder field $(\go_{i,j})_{i,j\geq 0}$ presents some long-range correlations (along lines and columns): as a result, the limit that we obtain is based on a correlated Gaussian field $\cM$ (see~Figure~\ref{fig:M} for an illustration) that exhibits the same type of correlations.
Note that in the PS model, the question of the influence of long-range correlated disorder on the denaturation transition has been investigated, for instance in~\cite{B13,Ber14,BL12,BP15,CCP19,Poi13}.
However, to the best of our knowledge, intermediate disorder regimes have so far been considered only in the case of i.i.d.\ disorder fields (or at least time-independent for models in dimension $1+d$), with the exception of~\cite{SSSX21}.
One can therefore view our result as a new attempt to investigate the influence of a correlated disorder on physical systems.

\subsection*{Some notation}
Throughout the article, we write elements of $\N^2$, $\R^2$ with bold characters, and elements of $\N$, $\R$ with plain characters (in particular we note $\bzero:=(0,0)$ and $\bone:=(1,1)$); moreover for $\bt\in\R^2$, $\bt^{(a)}$ will denote its projection on its $a$-th coordinate, $a\in\{1,2\}$. When there is no risk of confusion, we may also write more simply $\bt=(t_1,t_2)$.
We also define orders on $\bbR^2$:
for $\bs,\bt\in\R^d$, write 
\[
\bs \prec \bt \quad \text{ if } \ s_1 < t_1,\, s_2 < t_2\,,
\qquad \text{ and } \qquad  
\bs \preceq \bt \quad \text{ if } \ s_1 \leq t_1, \, s_2 \leq  t_2 \,.
\]
For $ \bzero \preceq \bs \preceq \bt$, let $[\bs,\bt]$ denote the rectangle $[s_1,t_1]\times [s_2,t_2]$ (and similarly $[\bs,\bt) := [s_1,t_1)\times [s_2,t_2)$, etc.) and $\llb  \bs,\bt\rrb:= [\bs,\bt] \cap \bbZ^2$. 
For $s,t\in\R$, we write $s\wedge t=\min(s,t)$ and $s\vee t=\max(s,t)$; and for $\bs,\bt\in\R^2$, 
\[
\bs\wedge\bt\;:=\;(s_1\wedge t_1,s_2\wedge t_2) \qquad
\text{and}\qquad 
\bs\vee\bt\;:=\;(s_1\vee t_1,s_2\vee t_2)\;.
\]
For $\bs \in \bbR^2$, let $\lfloor \bs \rfloor := (\lfloor s_1 \rfloor , \lfloor s_2\rfloor)$. Finally, we will say that $\bs,\bt\in\R^2$ are \emph{aligned} if they are on the same line or column,
that is if $s_1=t_1$ or $s_2=t_2$, and we then write $\bs \lrarw \bt$; otherwise we write $\bs\nlrarw\bt$.

\subsection{The generalized Poland-Scheraga model: definition and first properties}
\label{sec:defgPS}

Let $\btau = (\btau_k)_{k\geq 0}$ be a bivariate renewal process, with $\btau_0=\bzero$ and inter-arrival distribution
\begin{equation}\label{def:tau}
\bP\big( \btau_1= (\ell_1,\ell_2) \big) := K(\ell_1+\ell_2) = \frac{L(\ell_1+\ell_2)}{(\ell_1+\ell_2)^{2+\ga}},
\qquad \forall \, {\boldsymbol\ell}=(\ell_1,\ell_2) \in \N^2 \,,
\end{equation}
with $\bP(|\btau_1|<+\infty)=1$, $\ga>0$, and where $L(\cdot)$ is a slowly varying function (see~\cite{BGT87}). 
This specific choice of distribution is motivated by the natural assumption that the energy of a loop depends only on its total length (see~\cite[Sections~1.2,~1.3]{GK17} for more details). With a slight abuse of notation, we also interpret $\btau$ as a set $\{\btau_1, \btau_2,\ldots\}$ (we will always omit $\btau_0$).

Let $\hat \go = (\hat \go_{i})_{i\geq 1}$ and $\bar \go = (\bar \go_{i})_{i\ge1}$ be two independent sequences of i.i.d.\ random variables, with the same law. 
For $\bi\in \bbN^2$, we denote $\go_{\bi}= \go_{i_1,i_2}:= V(\hat \go_{i_1} , \bar \go_{i_2})$, where $V(\cdot, \cdot)$ is a \emph{symmetric} function describing the interactions between the monomers; we naturally assume that $V(\cdot,\cdot)$ is not constant.
Let us stress that $\go :=(\go_{\bi})_{\bi\in \bbN^2}$ is a strongly correlated field. 
Throughout the paper, the laws of $\hat\go$, $\bar\go$ and $\go$ will be denoted by~$\bbP$ (we use the same symbol, and it will always be clear from context which one is considered). Moreover we assume that there is some $\gb_0>0$ such that
\[
\lambda(\gb):= \log \bbE[ e^{\gb \go_{\bone}}]<+\infty\qquad  \text{ for } \gb\in[0,\gb_0) \,.
\]

\begin{example}
A first, natural example, is to take $V$ in a product form,
that is $ V(x,y) = f(x)f(y)$ for some function $f$: this is the choice made in~\cite{Leg21}, with $f(x)=x$.
Another natural example would be to take $V(x,y) = g(x+y)$, for some function $g$. Finally, in order to model DNA, one could define $V$ on the set $\{A,T,C,G\}^2$ of possible nucleotide pairings.
\end{example}

\begin{remark}
\label{rem:singleomega}
Recalling that the gPS model was introduced to model DNA denaturation, and that DNA strands are complementary, it would also be natural to consider a field $(\go_{\bi})_{\bi\in \bbN^2}$
defined as a function of a unique sequence $(\tilde \go_i)_{i\geq 1}$ of i.i.d.\ random variables, 
by $\go_{\bi} = V( \tilde \go_{i_1} , \tilde \go_{i_2})$, with $V$ a symmetric function.
Our approach would actually provide results very similar to the one we obtain in the setting described above.
We comment in Section~\ref{sec:uniquesequence} below what is expected when constructing $\go$ with a unique sequence $\tilde\go$,
but we do not develop that case any further since it becomes more technical and should not bring much different results.
\end{remark}

For a fixed realization of $\go$ (\emph{quenched} disorder),  we define, for $\gb\geq 0$ (the disorder strength) and $h\in \bbR$ (the pinning potential), the following \emph{polymer measures}: for any $\bn\in \bbN^2$, representing the respective lengths of the strands, let
\begin{equation}
\label{def:Pnh}
\frac{\dd \bP_{\bn,h}^{\gb,\go,\quen}}{\dd \bP} (\tau) := \frac{1}{Z_{\bn, h}^{\gb, \go,\quen}}  \exp\Big( \sum_{\bi \in \llbracket \bone,\bn \rrbracket  } \big( \gb \go_{\bi} -\lambda(\gb)+h\big) \ind_{\{ \bi \in \btau\}} \Big)  \ind_{\{ \bn \in \btau \}} ,
\end{equation}
where
\begin{equation}\label{def:Z}
Z_{\bn, h}^{\gb, \go,\quen}:= \bE\Big[   \exp\Big( \sum_{\bi \in \llbracket  \bone ,\bn \rrbracket } \big( \gb \go_{\bi} -\lambda(\gb)+h\big) \ind_{\{ \bi \in \btau\}} \Big)   \ind_{\{ \bn \in \btau \}} \Big]
\end{equation}
is the partition function of the system.  This corresponds to giving a reward (or penalty if it is negative) $\gb \go_{\bi}+h$ if $\bi=(i_1,i_2)\in\btau$, that is if monomer $i_1$ of the first strand is paired with monomer $i_2$ of the second strand. The term $-\lambda(\gb)$ is only present for renormalization purposes, and even though $Z_{\bn, h}^{\gb, \go,\quen}$ depends on the realization of $\go$, we will drop it in the notation for conciseness. 

Let us mention that it is also natural consider a \emph{conditioned} or \emph{free} version of the model, either by replacing $\ind_{\{\bn\in \btau\}}$ with a conditioning or simply by removing it: the partition functions are then
\begin{align}
\label{def:Zcond}
Z_{\bn, h}^{\gb,\cond} & := \bE\Big[  \exp\Big( \sum_{\bi \in  \llbracket  \bone,\bn \rrbracket } \big( \gb \go_{\bi} -\lambda(\gb)+h\big) \ind_{\{ \bi \in \btau\}} \Big)  \, \Big | \,  \bn\in \btau \Big] \, ,\\
\text{ and } \qquad 
Z_{\bn, h}^{\gb,\free}& := \bE\Big[  \exp\Big( \sum_{\bi \in  \llbracket  \bone,\bn \rrbracket  } \big( \gb \go_{\bi} -\lambda(\gb)+h\big) \ind_{\{ \bi \in \btau\}} \Big)  \Big] \, .
\label{def:Zfree}
\end{align}

\subsubsection*{Free energy and denaturation phase transition}

In~\cite{BGK20,Leg21}, it is shown that for $\gamma>0$ (the asymptotic strand lengths ratio),
the following limit, called \emph{free energy}, exists a.s.\ and in $L^1(\bbP)$:
\begin{equation}
\label{def:freeenergy}
\tf_{\gamma}(\gb,h) = \limtwo{n_1,n_2\to\infty}{ n_1/n_2 \to \gamma} \frac{1}{n_1} \log Z_{(n_1,n_2), h}^{\gb,\quen}   
 = \limtwo{n_1,n_2\to\infty}{ n_1/n_2 \to \gamma}  \frac{1}{n_1} \bbE\log Z_{(n_1,n_2), h}^{\gb,\quen}   \,.
\end{equation}
Also, the limit is unchanged if one replaces the partition function by its conditioned 
or its free counterparts.
The function $(\gb,h)\mapsto \tf_{\gamma}(\gb,h+\lambda(\gb))$ is non-negative, convex
and non-decreasing in each coordinate.
Additionally, the free energy encodes localization properties of the model: indeed, one can exploit the convexity of $\tf_{\gamma}$ to show that, if $\partial_h \tf_{\gamma}(\gb,h)$ exists, which is for all but at most countably many $(\gb,h)$, then 
\[
\partial_h \tf_{\gamma}(\gb,h) = \limtwo{n_1,n_2\to\infty}{ n_1/n_2 \to \gamma} \bE_{\bn,h}^{\gb,\quen}\Big[ \frac{1}{n_1} \sum_{\bi\in \llbracket \bone,\bn \rrbracket } \ind_{\{\bi\in \btau\}}\Big]\,, \qquad \bbP\text{-a.s.}
\]
In other words, $\partial_h \tf_{\gamma}(\gb,h)$ is the asymptotic fraction of contacts between the two strands.
This leads to the definition of the critical point:
\begin{equation}
\label{def:hcritic}
h_c^{\quen}(\gb) := \inf \{ h  \colon \tf_{\gamma}(\gb,h) > 0\} \,,
\end{equation}
which marks the transition between a \textit{delocalized} phase ($h<h_c^{\quen}(\gb)$, zero density of contacts) 
and a \textit{localized} phase ($h<h_c^{\quen}(\gb)$, positive density of contacts).
Let us stress that $h_c^{\quen}(\gb)$ does not depend on $\gamma>0$, since we have the following bounds
$\tf_{\gamma}(\gb,h) \leq \tf_{\gamma'}(\gb,h) \leq \frac{\gamma'}{\gamma}\tf_{\gamma}(\gb,h)$ for $0<\gamma\leq \gamma'$, see~\cite[Prop.~2.1]{BGK20}.

\subsubsection*{Homogeneous gPS model and Harris' predictions for disorder relevance} 

As mentioned above, the homogeneous version of the model, \textit{i.e.}\ when $\gb=0$, is solvable, see~\cite{GK17}. More precisely, under the assumption~\eqref{def:tau}, we have $h_c=h_c(0)=0$ and we can identify the critical behaviour
\begin{equation}
\label{def:homcritic}
\tf_{\gamma}(0,h) \sim c_{\alpha,\gamma} \psi(1/h) h^{\nu} \,,\qquad \text{as } h\downarrow 0\,, \quad \text{ with }  \nu = \frac1\ga \vee 1 \,,
\end{equation}
for some slowly varying function $\psi$ (that depends on $\alpha$ and $L(\cdot)$) and some constant $c_{\alpha,\gamma}$ (it is the only quantity on the r.h.s.\ of~\eqref{def:homcritic} that depends on $\gamma$). This determines the critical behaviour of the homogeneous denaturation transition, identified by the critical point $h_c(0)=0$.

Simply by applying Jensen's inequality, we get that 
\[
\tf_{\gamma}(\gb,h) =  \limtwo{n_1,n_2\to\infty}{ n_1/n_2 \to \gamma}  \frac{1}{n_1} \bbE\log Z_{(n_1,n_2), h}^{\gb,\quen}  
\leq \limtwo{n_1,n_2\to\infty}{ n_1/n_2 \to \gamma}  \frac{1}{n_1} \log \bbE[Z_{(n_1,n_2), h}^{\gb,\quen}]  = \tf_{\gamma}(0,h) \,,
\]
where we have used that $\bbE[Z_{\bn, h}^{\gb,\quen}] = Z_{\bn,h}^{\gb=0,\quen}$, see~\cite[Eq.~(1.11)]{Leg21}.
Hence, we get that $h_c^{\quen}(\gb)\geq 0$ for any~$\gb \geq 0$.

In view of~\eqref{def:homcritic}, Harris' criterion for disorder relevance becomes:
if $d$ is the ``dimension of disorder'', then
disorder should be irrelevant if $\alpha < \frac d2$ and relevant if $\alpha > \frac d2$.
It would be natural to assume that in our setting where $\go_{\bi} =  V( \hat\go_{i_1} , \bar\go_{i_2})$, disorder is one-dimensional (since two strands of length $n$ involve~$2n$ independent random variables). Yet,~\cite{Leg21} studies the question of the influence of disorder in the case $V(x,y)=xy$, and shows the following:
(i) if the law of $\hat\go_1$ is not of the form $\frac12 (\delta_{-x}+\delta_x)$ for some $x>0$; then disorder is ``one-dimensional'': it is relevant if $\alpha >\frac12$ and irrelevant if $\alpha<\frac12$;
(ii) if the law of $\hat\go_1$ is given by $\frac12 (\delta_{-x}+\delta_x)$ for some $x>0$, then disorder is ``two-dimensional'': it is relevant if $\alpha >1$ and irrelevant if $\alpha<1$.

\section{Main results: intermediate disorder for the gPS model}

Our aim is to complete those results on disorder (ir)-relevance for the gPS model, by taking inspiration from \cite{AKQ14, CSZ13}. In those papers the authors proved for some disordered systems (notably the disordered pinning model in \cite{CSZ13}) that, by choosing a disorder intensity $\gb_n$ decaying to 0 as $n\to\infty$, it was possible to exhibit an \emph{intermediate disorder} regime, laying in-between the homogeneous ($\gb=0$) and disordered ($\gb>0$ constant) ones. In~\cite{CSZ16}, it is argued that the fact that such a scaling gives rise to a non-trivial, random limit is a new notion of disorder relevance, and that it should coincide with the usual meaning introduced by Harris \cite{H74}.

Our main result consists in proving an intermediate disorder scaling limit of the disordered gPS model defined in Section~\ref{sec:defgPS}:
we focus on the scaling limit of the partition function in the case $\alpha \in (0,1)$, since then the bivariate renewal admits a non-trivial scaling limit, see Proposition~\ref{thm:renouv} below. We then derive some consequences of this scaling limit in terms of disorder relevance, more precisely regarding the critical point shift.
One of the main difficulties we have to overcome is the fact that the disorder field $(\go_{\bi})_{\bi \in \bbN^2}$
has long-range (in fact, infinite-range) correlations along lines and columns.

\subsection{Heuristics of the chaos expansion}
As in \cite{AKQ14, CSZ13}, we look for scaling limits of the partition functions by computing polynomial expansions, starting with the free version \eqref{def:Zfree} for simplicity. Let us define 
\[
\gz_{\bi} = \gz_{\bi}(\gb) := e^{\gb \go_{\bi} -\lambda(\gb)}-1 \,,
\]
so that $\exp\big( (\gb \go_{\bi} -\lambda(\gb) +h)\ind_{\{\bi \in \btau\}} \big) = 1+(e^h\zeta_{\bi}  +e^h-1 )\ind_{\{\bi \in \btau\}}$. Then, for $\bt\in(\R_+^*)^2$ and $n\in\N$, expanding the product
in 
$Z_{n\bt,h}^{\gb, \free} = \bE\big[ \prod_{\bi \in \llbracket\bone,n\bt\rrbracket} \big( 1+(e^h\zeta_{\bi}  +e^h-1 )\ind_{\{\bi \in \btau\}}\big) \big]$
and using the renewal structure, we have
\begin{equation}
\label{expansion}
\begin{split}
Z_{n\bt,h}^{\gb, \free} 
&= 1 + \sum_{k=1}^{(nt_1)\wedge (n t_2)} \sum_{ \bzero=\bi_0 \prec \bi_1 \prec \ldots \prec \bi_k \preceq  n\bt } \prod_{l=1}^k \Big( (e^h\zeta_{\bi_l}  +e^h-1)\,u(\bi_l-\bi_{l-1})\Big) \,,
\end{split}
\end{equation}
where we denoted $u(\bi):= \bP(\bi\in \btau)$ the \emph{renewal mass function}. 
In order to understand the correct scaling for the parameters $h$ and $\gb$, let us
focus on the convergence of the term $k=1$. As $h \to 0$, it is
equal to (up to smaller order terms in $h$) 
\begin{equation}
\label{chaosk=1}
\sum_{\bi \in \llbracket \bone , n\bt  \rrbracket  } \zeta_{\bi} \pt  u(\bi)   + h \sum_{\bi \in \llbracket \bone , n\bt  \rrbracket  } \pt u(\bi)  \, .
\end{equation}

\subsubsection{The homogeneous term}
Looking at the homogeneous term in~\eqref{chaosk=1} (\textit{i.e.}\ the second one), we need to estimate the renewal mass function $u(\bi)$. When $\ga \in (0,1)$, this is provided by \cite{Will68}.
\begin{proposition}[\cite{Will68}, main result]\label{thm:renouv}
Assume $\ga\in (0,1)$ in \eqref{def:tau}. Then for $\bs\in(\R_+^*)^2$, we have
\begin{equation}
\label{def:phi}
\lim_{n\to+\infty} n^{2-\ga} L(n) \bP \big(  \lfloor n \bs \rfloor \in \btau \big) \;=\; \gp(\bs)\,,
\end{equation}
for some continuous function $\phi:(\R_+^*)^2\to\R_+$. Writing $\bs$ in the polar form $\bs= r e^{i\theta}$, we get that $\gp(\bs) = r^{\ga-2} a(\theta)$, for some continuous function $a:[0,\pi/2] \to \bbR_+$, which is equal to $0$ at $\theta=0$ and $\theta=\pi/2$. 
\end{proposition}
This theorem and a Riemann sum approximation
imply that $\sum_{\bi \in \llbracket \bone, n \bt\rrbracket } u(\bi) \sim c_{\bt}  L(n)^{-1} n^{\alpha} $ as $n\to\infty$,
with $c_{\bt}:=\int_{[\bzero,\bt]} \gp(\bs) \dd \bs <+\infty$.
Hence, in order to make the second term converge in \eqref{chaosk=1}, we have to take $h= h_n$ proportional to $L(n) n^{-\ga}$.

\begin{remark}
\label{rem:limitingset}
One could show that for any $\alpha>0$,
the random set $\frac1n \btau = \{\frac{\btau_i}{n}\}_{i\geq 0}\subseteq (\bbR_+)^2$ converges in distribution towards 
a random closed set $\cS_{\alpha}\subseteq (\bbR_+)^2$
(for the Fell--Matheron topology, we refer to~\cite[App.~A]{CSZ16_conti} for an overview of such convergence for univariate renewals). 
When $\alpha\in (0,1)$, Proposition~\ref{thm:renouv}
shows that~$\cS_{\alpha}$ is random (and $\gp$ characterizes its finite-dimensional distributions).
On the other hand, 
when $\alpha \geq 1$, $\cS_{\alpha}$ is easily seen to be simply the diagonal $\Delta=\{(x,x), x\in \bbR_+\}$: 
this justifies to focus on the case $\alpha\in (0,1)$.
\end{remark}

In fact, let us state right away the scaling limit of the homogeneous constrained  partition function, \textit{i.e.}\ when $\gb_n \equiv 0$, in the scaling window $h_n \asymp L(n) n^{-\ga}$.
Similar results hold for the free and conditioned partition function (see Remark~\ref{rem:freecond} in the non-homogeneous case).

\begin{proposition}
\label{prop:scalinghom}
Assume that $\lim_{n\to\infty} n^{\alpha} L(n)^{-1} h_n  = \hat h \in \mathbb R$.
Then, for any $\alpha \in (0,1)$ and any $\bt \succ 0$, writing $Z_{\bn,h}:=Z_{\bn,h}^{\gb=0, \quen}$, we have
\begin{equation}
\label{def:Zhom}
\lim_{n\to\infty} n^{2-\alpha} L(n)  Z_{\lfloor n \bt\rfloor, h_n}  = \bZ_{\bt, \hat h}  := \gp(\bt)+\sum_{k=1}^{+\infty}  \hat h^k 
\idotsint \limits_{ \bzero \prec \bs_1 \prec \cdots \prec \bs_k \prec \bt }  \prod_{i=1}^{k+1} \gp( \bs_i- \bs_{i-1})   \dd \bs_1 \cdots \dd \bs_k \,, 
\end{equation}
In~\eqref{def:Zhom} we use the convention $\bs_0:=\bzero$ and $\bs_{k+1}:=\bt$ for the $k$-th term of the sum.
\end{proposition}

\noindent
In view of Remark~\ref{rem:limitingset}, in the same way as the function $\gp$ characterizes the finite-dimensional distribution of the set $\mathcal S_{\alpha}$,
one can interpret the quantity $\bZ_{\bt,\hat h}$ as characterizing the finite-dimensional distribution of
a (weakly) pinned set $\mathcal S_{\alpha}$, much as in~\cite{Soh09} for the usual pinning model. In particular, we have $\bZ_{\bt,\hat h=0} = \gp(\bt)$.

\subsubsection{The disordered term}
Considering the disorder term in \eqref{chaosk=1} (\textit{i.e.}\ the first one), notice that $\bbE[\gz_{\bi}] =0$ and $\Var(\gz_i) \sim \Var(\go_{\bone}) \gb^2$ as $\gb\downarrow 0$ (see Lemma~\ref{lem:correl} below).
If the variables $(\gz_\bi)_{\bi\in\N^2}$ were independent, then, properly rescaled,
the sum would converge to an integral of $\phi$ against a white noise, as is the case for an i.i.d.\ disorder, see \cite{AKQ14,CSZ13}.
However, in our case, the field  $(\gz_\bi)_{\bi\in\N^2}$ displays strong correlations on each line and column of $\N^2$. 
Therefore, the first result we prove is that the partial sums of $(\gz_\bi)_{\bi\in\N^2}$,
properly normalized by $\gb^r n^{3/2}$ for some $r\in \bbN$, converge towards a Gaussian random field~$\cM$ which encapsulates the correlation structure of~$\gz$, see Theorem~\ref{thm:cvgcM} below.
A (non-trivial) consequence of Theorem~\ref{thm:cvgcM} (and Proposition~\ref{thm:renouv}) is the following convergence in distribution and in~$L^2(\bbP)$,
as $n\to\infty$ and $\gb\downarrow 0$,
\begin{equation}\label{eq:convergencek=1}
  \frac{1}{\sigma_r n^{3/2}\gb^r } \sum_{\bi \in \llbracket\bone, n\bt\rrbracket} \zeta_{\bi} \; \frac{u(\bi) L(n)}{n^{\ga-2} }  \;\longrightarrow\; \int_{[\bzero,\bt]} \gp(\bs) \dd\cM(\bs) \,,
\end{equation}
where $r\in \bbN$ and $\sigma_r$ are constants that depend on the distribution of $\go_{\bi}$ (see Lemma~\ref{lem:correl} below).
The definition of the integral with respect to $\cM$, together with the fact that the integral on the right-hand side of~\eqref{eq:convergencek=1} is well-defined, is part of the statement, and is discussed in detail in Section~\ref{sec:integration} below.

\subsubsection{Scaling window}
All together, the analysis of the first term~\eqref{chaosk=1} in the chaos expansion~\eqref{expansion}
suggests that one should take the following scaling for the parameters $\gb_n,h_n$:
\begin{equation}
\label{def:scalings}
\lim_{n\to\infty} \frac{h_n}{ L(n) n^{-\ga}} := \hat h \in \bbR \qquad \text{ and } \qquad \lim_{n\to\infty} \frac{\gb_n}{(n^{\frac12-\ga} L(n))^{\frac1r}} :=\hat\gb  \in [0,+\infty)  \,.
\end{equation}
Note in particular that in order to be able to have $\lim_{n\to\infty} \gb_n =0$ (which is required to obtain an intermediate disorder regime), we require $\alpha >\frac12$.

\subsection{Convergence of the field $(\zeta_{\bi})_{\bi\in \bbN^2}$}

Recall that $(\hat\go_{i_1})_{i_1\geq1}$, $(\bar\go_{i_2})_{i_2\geq1}$ are i.i.d.\ sequences with the same law, and that $\go_{\bi} = V(\hat\go_{i_1},\bar \go_{i_2})$ for some symmetric function $V(\cdot,\cdot)$.
Let us denote $\Pfk$ the set of disorder distributions $\bbP$ 
such that $\lambda(\gb) := \log \bbE[ e^{\gb \go_{\bone}} ]<+\infty$ for $\gb\in [0,\gb_0)$.
Recall also that  $\zeta_{\bi} = \zeta_{\bi} (\gb_n):= e^{\gb_n \go_{\bi} - \lambda(\gb_n)} -1$
and that it has mean $0$.
Before we prove the convergence of the field $(\gz_{\bi})_{\bi\in \bbN^2}$, let us state a central lemma
that provides the asymptotic behaviour of the two-point correlations $\bbE[\gz_{\bi} \gz_{\bj}]$:
a key fact is that the correlations on lines and columns actually depend on the interaction function $V(\cdot,\cdot)$
and on the distribution $\bbP$.

\begin{lemma}\label{lem:correl}
Let $\bi, \bj \in \bbN^2$. If $\bi\nlrarw\bj$ then 
$\bbE\left[\zeta_{\bi}\zeta_{\bj}\right] =0$.
Additionally,
as $\gb_n \to 0$,
\begin{equation}\label{eq:correl:zeta}
\bbE\left[\zeta_{\bi}\zeta_{\bj}\right] = 
\begin{cases}
\sigma^2 \gb_n^2  + o(\gb_n^2)\quad & \text{if } \bi=\bj ,
\\
\sigma_r^2\gb_n^{2r} + o(\gb_n^{2r})\quad &\text{if }\bi \lrarw \bj,\, \bi \neq \bj \text{ and }\bbP\in\Pfk_r  \,,
\end{cases}
\end{equation}
with $\sigma^2 := \mathrm{Var}(\go_{\bone})$, $\sigma_r^2 = \frac{1}{(r!)^2}  \Var\big( \bbE[ \go_{\bone}^r \,|\, \hat \go_{1} ] \big)$ and
\begin{equation}\label{eq:partPfk}
\Pfk_r = \Big\{  \bbP \colon\   \min\big\{ k\geq 1\,,\, \Var\big( \bbE[ \go_{\bone}^k \,|\, \hat \go_{1} ] \big) >0 \big\} = r \Big\} \,.
\end{equation}
If $\sigma_k^2=0$ for all $k\in \bbN$, \textit{i.e.}\ if
$\bbP\in \Pfk_{\infty}$, then $\bbE\left[\zeta_{\bi}\zeta_{\bj}\right]=0$ for all $\bi\neq \bj$ (in fact, $\bbE\left[\zeta_{\bi}\mid \hat \go_{i_1}\right]=0$ a.s.).
\end{lemma}

Therefore, $\Pfk$ is partitioned into sets $\Pfk_r$ for $r\in \bbN \cup\{+\infty\}$ and the decay of the correlations
on lines and columns depend on which $\Pfk_r$ contains the distribution $\bbP$.
Let us stress that in general, the $\Pfk_r$
(and notably~$\Pfk_{\infty}$) 
are non-empty: let us give a few examples.

\begin{example}\label{example1}
In the case $V(x,y)=xy$
we find that (see~\cite{Leg21}):
$\Pfk_1$ is the set of distributions such that $\bbE[\hat \go_1] \neq 0$ (since $\bbE[ \go_{\bone} \,|\, \hat \go_{1} ]=\hat\go_1 \bbE[\bar \go_1]=\hat\go_1 \bbE[\hat \go_1]$ a.s.); $\Pfk_2$ is the set of distributions such that $\bbE[\hat \go_1] = 0$ and $\Var(\hat \go_1^2) >0$;
$\Pfk_r$ is empty for any $3\leq r <\infty$; and
$\Pfk_{\infty}$ contains the remaining distributions,
\textit{i.e.}\ when $\hat\go_1$ has law $\frac12 (\delta_{-x}+\delta_x)$ for some $x>0$.
\end{example}


\begin{example}\label{example3}
In Appendix~\ref{app:example},
we tailor an example to obtain an instance where $\Pfk_4 \neq \emptyset$ and even $\Pfk_8 \neq \emptyset$.
The example is based on an interaction function of the form $V(x,y) =xf(y) +yf(x)$,
for some well chosen distribution $\bbP$ and function $f$.
It is reasonable to expect that such an example could be adapted to
construct cases where $\Pfk_r$ is non-empty for some arbitrarily large values of $r$.
\end{example}

\begin{remark}\label{rem:general:r}
In spite of Example~\ref{example3}, we believe that most ``natural'' cases verify $r\in\{1,2,+\infty\}$, the others emerging only from specifically tailored distributions. 
Let us stress however that the model does not appear simpler to study when restricted to $r\in \{1,2\}$ (there is no significant simplification), so in the remainder of the paper we treat the case of a generic $r$.
\end{remark}

Let us now state the convergence of the field $(\zeta_{\bi})_{i\in \bbN^2}$.
For $\bs = (s_1,s_2)\in(\R_+)^2$ and $n\in\N$, let us define
\begin{equation}
M_n(\bs)\;:=\; \sum_{\bi \in \llbracket \bone, n\bs \rrbracket}  \zeta_{\bi}(\gb_n)\,,
\end{equation}
with the convention $M_n(\bs)=0$ if $s_1<1/n$ or $s_2<1/n$.

\begin{theorem}
\label{thm:cvgcM}
Let $\bt\in(\R_+)^2$ and recall that $\zeta_{\bi}= \zeta_{\bi}(\gb_n) = e^{\gb_n\go_{\bi}-\gl(\gb_n)}-1$, $\bi\in\N^2$.
Assume that $\bbP \in \Pfk_r$ for some $r\in \bbN$.
If $\lim_{n\to \infty} \gb_n=0$ and $\lim_{n\to\infty} n\gb_n^{2r} = +\infty$,
then
\begin{equation}
\bigg( \frac{1}{\sigma_r n^{3/2} \gb_n^{r}} M_n(\bs)\bigg)_{\bs\in [\bzero,\bt]} \;\xrightarrow{(d)}\; \big( \cM(\bs) \big)_{\bs\in[\bzero,\bt]}\,,
\end{equation}
where $\cM$ is a Gaussian field on $(\R_+)^2$ with zero-mean and covariance matrix given by
\begin{equation}\label{def:KcovcM}
K(\bu,\bv):= \big(u_1\wedge v_1\big) \big(u_2\wedge v_2\big) \big( u_1\vee v_1 + u_2\vee v_2 \big)\,, \quad \bu=(u_1,u_2),\bv=(v_1,v_2)\in(\R_+)^2\,.
\end{equation}
The convergence holds for the topology of the uniform convergence on $[\bzero,\bt]$.
\end{theorem}

Let us also mention that when $\bbP \in \Pfk_{\infty}$, then using Lemma~\ref{lem:correl} and the central limit theorem, 
we can show that 
\[
\bigg( \frac{1}{\sigma n \gb_n} M_n(\bs)\bigg)_{\bs\in [\bzero,\bt]} \;\xrightarrow{(d)}\; \big( \cW(\bs) \big)_{\bs\in[\bzero,\bt]}\,,
\]
where $\cW$ is a Gaussian field with covariance $(u_1\wedge v_1) (u_2\wedge v_2)$,
\textit{i.e.}\ $\cW$ is a Brownian sheet.
In other words, the rescaled field $ (\frac{1}{\sigma n\gb_n} \zeta_{\bi})_{\bi\in \mathbb N^2}$ converges to a Gaussian two-dimensional white noise. We do not prove this statement since it is not needed below.

\subsection{Intermediate disorder: statement of the main result}
We now have the tools to state our main result.
The only missing piece is that our statement involves iterated integrals against the field $\cM$.
We refer to Section~\ref{sec:integration} and Appendix~\ref{app:intstoch} below for a construction of the integrals  against the field $\cM$, and in particular for the proof that the chaos expansion series in~\eqref{eq:conjchaosexpansion} is well-defined in $L^2(\bbP)$.

\begin{theorem}\label{conj:scalinggPS}
Let $\btau$ satisfy \eqref{def:tau} with $\ga \in (\frac12, 1)$. Let $\bbP\in\Pfk_r$ for some $r\in \bbN$, and let $(\gb_n)_{n\geq 1}$, $(h_n)_{n\geq 1}$ satisfy the scaling relation~\eqref{def:scalings}.
Then for any $ \bt \succ \bzero$, we have the convergence in distribution
\begin{equation}\label{eq:convZZcont}
n^{2-\alpha} L(n) Z_{\lfloor n \bt\rfloor, h_n}^{ \gb_n,\go, \quen} \quad \xrightarrow[n\to+\infty]{(d)} \quad \mathbf{Z}_{\bt, \hat h}^{\hat \gb, \cM, \quen} \,,
\end{equation}
where the random variable $\mathbf{Z}_{\bt, \hat h}^{\hat\gb, \cM, \quen}$ is given by the chaos expansion (in~$L^2(\bbP)$)
\begin{equation}
\label{eq:conjchaosexpansion}
\mathbf{Z}_{\bt, \hat h}^{\hat\gb, \cM, \quen} := \gp(\bt) + \sum_{k=1}^{+\infty} \ \
\idotsint \limits_{ \bzero \prec \bs_1 \prec \cdots \prec \bs_k \prec \bt }  \psi_{\bt} \big( \bs_1,\ldots, \bs_k  \big) \prod_{j=1}^k \Big( \sigma_r \hat\gb^r \,\dd \cM(\bs_j) + \hat h \, \dd \bs_j \Big)\,.
\end{equation}
In \eqref{eq:conjchaosexpansion}, we have $\sigma_r^2:= \frac{1}{(r!)^2}  \Var\big( \bbE[ \go_{\bi}^r \,|\, \hat \go_{i_1} ] \big)$ and $\psi_{\bt}^{}$ is defined by 
\begin{equation}
\label{def:psit}
\psi_{\bt} ( \bs_1,\ldots, \bs_k ) \;:=\;  \ind_{\{\bzero =: \bs_0 \prec \bs_1 \prec \cdots \prec \bs_k \prec s_{k+1}:=\bt \}} \, \prod_{i=1}^{k+1} \gp( \bs_i- \bs_{i-1})   \,.
\end{equation}
\end{theorem}

\begin{remark}
\label{rem:freecond}
As far as the  conditioned $~^\cond$ and free $~^\free$ partition functions are concerned, the same result holds, without the scaling factor $n^{2-\alpha} L(n)$: one has to replace $\psi_{\bt}^{}$ respectively by $\psi_\bt^{\cond}$ and $\psi_{\bt}^{\free}$,
defined by
\begin{align}
\psi_\bt^{\cond} ( \bs_1,\ldots, \bs_k  ) &:=  \frac{\psi_{\bt}^{} ( \bs_1,\ldots, \bs_k) }{\phi(\bt)} \,, \\
\psi_\bt^{\free} ( \bs_1,\ldots, \bs_k ) & := \ind_{\{\bzero =: \bs_0 \prec \bs_1 \prec \cdots \prec \bs_k \prec \bt \}}  \, \prod_{i=1}^{k} \gp( \bs_i- \bs_{i-1})   \,.
\label{eq:defpsifree}
\end{align}
In this paper we focus on the proof for the constrained partition function, and comment in Remark~\ref{rem:freecond2} below how to deduce the statement for the other two cases.
\end{remark}

Let us conclude this section by showing that 
when $\alpha\in(0,\frac12)$ or when $\alpha\in (0,1)$ and $\bbP\in \Pfk_{\infty}$,
then one cannot obtain a disordered scaling limit by
taking $\gb_n\to 0$.
This shows that disorder is irrelevant in these case, in the sense put forward in~\cite{CSZ16}.
We state the result in the free case for future use (similar statements hold in the constrained and conditioned case).

\begin{proposition}
\label{prop:degenerate}
Assume that $\alpha\in(0,\frac12)$
or that $\alpha\in (0,1)$ and $\bbP\in \Pfk_{\infty}$.
Then, if $\lim_{n\to\infty} n^{\alpha} L(n)^{-1} h_n = \hat h \in \mathbb R$,
for any vanishing sequence $(\gb_n)_{n\geq 1}$
we have that $\lim_{n\to\infty}  Z_{\lfloor n \bt\rfloor, h_n}^{ \gb_n, \free } =\bZ_{\bt,\hat h}^{\free}$ in $L^2(\bbP)$.
\end{proposition}

\subsection{Consequence on the critical point shift for the gPS model}
As a consequence of the scaling limit obtained above,
we are able to obtain upper bounds on the critical point shift.
Indeed, the following general statement allows to relate the second moment of the partition function at the annealed critical point $h_c^{\ann}(\gb) = 0$ to the critical point shift $h_c^{\quen}(\gb)$.
It is extracted from~\cite[Prop.~3.1]{Leg21}, and its proof is inspired by the approach in~\cite{Lac10ECP}.

\begin{proposition}[Proposition~3.1 in~\cite{Leg21}]
\label{prop:secondmomentcriticalpoint}
Fix some constant $C>1$ and define
\[
n_{\gb} := \sup\big\{  n \in \bbN,\, \bbE[(Z_{n \bone ,h=0}^{\gb,\free})^2] \leq C\big\} \,.
\]
Then there is some (explicit) slowly varying function $\tilde L$ such that the critical point satisfies
\begin{equation*}
0\leq h_c^{\quen}(\gb) \leq \tilde L(n_{\gb}) n_{\gb}^{-\alpha} \,.
\end{equation*}
If $Z_{n,h=0}^{\gb,\free}$ is bounded in $L^2(\bbP)$, then $n_{\gb} = +\infty$ (provided that $C$ had been fixed large enough), so $h_c^{\quen}(\gb)=0$; moreover, there exists a slowly varying function $\hat L$ such that for all $h \in (0,1)$
we have $\tf_{\gamma}(\gb,h) \geq \hat L(1/h) h^{1/\alpha}$.
\end{proposition}

Together with Theorem~\ref{conj:scalinggPS}, this allows us to obtain an upper bound on the critical point shift.
\begin{corollary}
\label{cor:boundcriticalpoint}
Assume that $\alpha\in (\frac12,1)$ and that $\bbP\in \Pfk_r$
for some $r \in \bbN$.
Then, we have the following upper bound on the critical point:
there is some $\gb_1>0$ such that for all $\gb \in (0,\gb_1)$
we have
\[
0\leq h_c^{\quen}(\gb)  \leq L_2(1/\gb) \gb^{\frac{2\alpha r}{2\alpha-1}} \,,
\]
for some slowly varying function $L_2$.
\end{corollary}

Note that this sharpens the bound found in~\cite[Prop.~2.4]{Leg21}, which treats the case of a product interaction $V(x,y) =xy$: when $\bbP\in \mathfrak{P}_2$, it was obtained that $h_c^{\quen}(\gb) \leq   L_2(1/\gb) \gb^{\frac{2\alpha }{2\alpha-1}}$. 
We believe that the upper bound in Corollary~\ref{cor:boundcriticalpoint} is sharp in general, up to slowly varying functions: indeed,  it matches the lower bound on the critical point shift obtained in~\cite[Thm.~2.3]{Leg21} (in the case of a product interaction).
Obtaining a lower bound on the critical point shift in the case of a general interaction seems reachable but technically involved. For this, one would need to adapt the ideas developed in~\cite{Leg21}, with extra technical difficulties coming from the general interaction $V(x,y)$. We leave this problem for future work.

\begin{proof}[Proof of Corollary~\ref{cor:boundcriticalpoint}]
Let $\gb(n) := (n^{\frac{1}{2} -\alpha} L(n))^{1/r}$ and  let $\bar n (\cdot)$ be the asymptotic inverse of $\gb(\cdot)$, 
\textit{i.e.}\ such that $\gb(\bar n(u)) \sim \bar n ( \gb(1/u))^{-1} \sim u$ as $u\downarrow 0$. One can show that $\bar n(u) \sim L^{*}(u)  u^{- \frac{2r}{2\alpha-1}}$ as $u\downarrow 0$ for some (semi-explicit) slowly varying function $L^*$, see~\cite[Thm.~1.5.13]{BGT87}.

Now, by Theorem~\ref{conj:scalinggPS} and thanks to the definition of $\bar n(\cdot)$, we have that $Z_{\bar n(\gb)\bone, 0}^{\gb,\free}$ converges to $\bZ_{\bone,\hat h=0}^{\hat \gb =1,\free}$ as $\gb\downarrow 0$ in $L^2$.
Hence, letting $C:= 2 \bbE[( \bZ_{\bone,0}^{1,\free})^2]$, we get that there exists $\gb_1>0$ such that 
$\bbE[(Z_{\bar n(\gb) \bone ,h=0}^{\gb,\free})^2] \leq C$ for all $\gb\leq \gb_1$.
Put otherwise, we get that $\bar n(\gb) \leq  n_{\gb}$ for any $\gb\leq \gb_1$, with $n_{\gb}$ defined in Proposition~\ref{prop:secondmomentcriticalpoint} with the constant $C$ above.
Applying Proposition~\ref{prop:secondmomentcriticalpoint}, we therefore end up with
\[
0\leq h_c^{\quen}(\gb) \leq c \tilde L(\bar n(\gb) ) \bar n (\gb)^{-\alpha} \,, 
\]
for all $\gb\leq \gb_1$. Since  $\bar n(u) \sim L^{*}(u)  u^{- \frac{2r}{2\alpha-1}}$ as $u\downarrow 0$, this concludes the proof.
\end{proof}

Let us stress that another corollary of Proposition~\ref{prop:secondmomentcriticalpoint} comes as a consequence of the proof of Proposition~\ref{prop:degenerate}. The following result shows that when $\alpha\in(0,\frac12)$ or when $\alpha\in (0,1)$ and $\bbP\in \Pfk_{\infty}$,
then disorder is irrelevant in the sense that, for small $\gb>0$, there is no critical point shift and no modification of the homogeneous critical behaviour (recall~\eqref{def:homcritic} and the fact that $ \tf_{\gamma}(\gb ,h)\leq \tf_{\gamma}(0,h)$).

\begin{corollary}
\label{cor:boundcriticalpoint2}
Assume that $\alpha\in (0,\frac12)$
or that $\alpha\in(0,1)$ and $\bbP\in \Pfk_{\infty}$.
Then,
there is some $\gb_1>0$ such that for all $\gb \in (0,\gb_1)$ 
we have $h_c^{\quen}(\gb) =0$ and $\tf_{\gamma}(0,h) \geq \tf_{\gamma}(\gb,h) \geq \hat L(1/h) h^{1/\alpha}$ for all $h\in (0,1)$.
\end{corollary}

\begin{proof}
Thanks to Proposition~\ref{prop:secondmomentcriticalpoint},
one simply need to show that for $\gb$ small enough  $Z_{n,h=0}^{\gb,\free}$ is bounded in $L^2(\bbP)$.
The estimate on $\bbE[(Z_{n \bone ,h=0}^{\gb,\free})^2]$ is obtained in the proof of Proposition~\ref{prop:degenerate}, see Section~\ref{sec:secondmomentbound}, more precisely Remark~\ref{rem:boundedL2}.
\end{proof}

\subsection{Some Comments}

\subsubsection{About disorder relevance} 
Theorem~\ref{conj:scalinggPS} and Proposition~\ref{prop:degenerate} provide a complete characterization for the existence of a non-trivial scaling limit for the gPS model with disorder $\go_\bi=V(\hat\go_{i_1},\bar\go_{i_2})$ and $\ga<1$. For the toy model $V(x,y)=xy$ of \cite{Leg21}, they confirm the prediction of \cite{CSZ16} claiming that this matches Harris' criterion for disorder relevance \cite{H74}, assuming that the \emph{dimension} of the disorder is described by the correlations of the field $(\gz_\bi)_{\bi\in\N^2}$ as in Lemma~\ref{lem:correl} (\emph{i.e.}, disorder is one-dimensional if and only if $\bbE[\gz_\bi\gz_\bj]\neq0$ for $\bi\aligne\bj$). Let us stress that we also proved that this limit is \emph{(partially) universal}, in the sense that the limiting continuous random field $\cM$ which defines $\bZ_{\bt,\hat h}^{\hat \gb, \cM, \quen}$ does not depend on $V(\cdot,\cdot)$ or $\bbP$, but only on the line-and-column correlation structure we chose for $\go$; however, the scaling at which the non-trivial limit holds depends strongly on the chosen disorder distribution, and ranges in a wide (countable) amount of possible values indexed by $r\in\N$ (where we provided explicit examples for $r=1,2,4$ and $8$ in Examples~\ref{example1}, \ref{example3}).

Let us mention that all this work is concerned with the denaturation transition. Regarding the condensation transitions in the gPS model,
the question of the influence of disorder has not yet been investigated.
However, the findings of~\cite{GH20} suggest that these (big-jump)
transitions are actually absent from the disordered version of the model.

\subsubsection{Relation to a stochastic (fractional) equation}

Let us write $\bZ_{\bt}:=\mathbf{Z}_{\bt, \hat h}^{\hat\gb, \cM, \quen} $ for the limiting (point-to-point) partition function in Theorem~\ref{conj:scalinggPS}. Then the chaos expansion~\eqref{eq:conjchaosexpansion} suggests that $\bZ_{\bt}$ satisfies the following stochastic equation on $\bbR_+^2$:
\begin{equation}
    \label{eq:Volterra}
  \bZ_\bt =  \delta_{\bzero} (\bt) + \int_{\bzero  \prec \bs \prec \bt}  \gp(\bt-\bs)  \,\bZ_{\bs} \, \xi (\dd \bs) \,,
\end{equation}
where $\xi = \sigma_r \hat \gb^r \cM + \hat h \,\mathrm{Leb}$ is the noise. 
The stochastic equation~\eqref{eq:Volterra} is deemed \textit{rough}, due to the presence of the singular kernel $\gp(\bt-\bs)$, which diverges as $\|\bt-\bs\|^{\alpha-2}$ when $\bs$ is close to $\bt$.

Note also that~\eqref{eq:Volterra} can be understood as a \textit{disordered} renewal equation, appearing as a continuum version of a renewal equation verified by the discrete partition function.
Let us stress that showing that $\mathbf{Z}_{\bt, \hat h}^{\hat\gb, \cM, \quen}$ defined in~\eqref{eq:conjchaosexpansion} indeed verifies~\eqref{eq:Volterra} would require some exchange of an infinite sum with an integral, which is not obvious; but we do not dwell further on this issue.

We mention that the one-dimensional version of~\eqref{eq:Volterra} is known as the \textit{stochastic Volterra equation} (SVE), which has attracted a lot of attention recently by the mathematical finance community.
The equation reads
\begin{equation}
    \label{eq:Volterra-d=1}
  Y_t = \eta  +  \sigma \int_{0}^t   (t-s)^{H-\frac12}\, Y_s\, \dd W_s \,,
\end{equation}
with $\eta$ an initial condition and $(W_s)_{s\geq 0}$ a standard Brownian motion, and it has been proved useful as a model for rough volatility; we refer to \cite{BFGMS20} and references therein for further information.
In~\eqref{eq:Volterra-d=1}, the case $H \in(0,\frac12)$ is known as \textit{subcritical}, and the solution is defined using either rough paths \cite{PT21}, paracontrolled calculus \cite{Com19} or regularity structures \cite{BFGMS20}; the critical case $H=0$ has recently been studied in~\cite{WY24}.

Let us mention that, as for~\eqref{eq:Volterra}, the SVE~\eqref{eq:Volterra-d=1} can be interpreted as a \textit{disordered} renewal equation, which appears in the scaling limit of the disordered pinning model; we refer to~\cite[\S1.3]{WY24} for further comments (see also~\cite[Thm.~3.1]{CSZ13} for the chaos expansion of the limiting partition function). 
In our context, we do not believe that~\eqref{eq:Volterra},  \textit{i.e.}\ a ``two-dimensional stochastic Volterra equation'', has ever appeared in the literature; we also do not know whether it could have applications in other domains.

\subsubsection{About the construction of a continuum gPS model}
 
In the relevant disorder case, a natural next step would be to construct (and study) a continuous version of the model, 
\textit{i.e.}\ a scaling limit of the Gibbs measure $\bP_{\bn,h}^{\gb,\go,\quen}$, see~\eqref{def:Pnh}.

For this, and in analogy of what is done for the standard pinning model~\cite{CSZ16_conti}, one needs to interpret $\btau$ as a closed subset of $\bbR_+^2$ which is increasing (for $\prec$).
Then, with $(\gb_n)_{n\geq 1}$ $(h_n)_{n\geq 1}$ satisfying~\eqref{def:scalings} as in Theorem~\ref{conj:scalinggPS}, one should be able to obtain a convergence of the measures, as $n\to\infty$
\begin{equation}
    \label{conv:Gibbs}
    \bP_{\bn,h_n}^{\gb_n,\go,\quen} \Big( \frac1n \btau \in \cdot \Big) \Longrightarrow \mu_{\hat \gb,\hat h}^{\cM} \,,
\end{equation}
where $\mu_{\hat \gb,\hat h}^{\cM}$ is a measure on increasing closed subsets of $\bbR_+^2$, dubbed as \textit{continuum gPS model} (as the universal scaling limit of the gPS model). Its corresponding continuum partition function is $\bZ_{\bt,\hat h}^{\hat \gb, \cM, \quen}$ defined in~\eqref{eq:conjchaosexpansion}.

Let us stress here that in~\eqref{conv:Gibbs}, the limit is \textit{disordered} and depends on the field $\cM$. As can be seen in the continuum partition function, it also carries a dependence on~$r$, but only through~$\sigma_r$ and the power of $\hat \gb$: this is why we call the scaling limit \textit{universal}, in the sense that the disorder field $\cM$ does not depend on the details of the disorder at the discrete level.

We do not prove~\eqref{conv:Gibbs} in the present paper, and in particular we do not define the continuum gPS model, but only its corresponding continuum partition function $\bZ_{\bt,\hat h}^{\hat \gb, \cM, \quen}$.
Let us however mention two possible lines of proof for obtaining~\eqref{conv:Gibbs} --- they could be implemented by adapting and extending our scheme of proof. 
A first strategy follows the ideas of~\cite{CSZ16_conti}: one should prove the convergence of a two-parameter family of point-to-point partition functions towards a process of continuum partition functions, see e.g.\ \cite[Thm.~16]{CSZ16_conti}, from which one should then able to construct the continuum model $\mu_{\hat \gb,\hat h}^{\cM}$.
Another strategy follows the ideas of~\cite{BL21_scaling,BL20_conti}: one should define partition functions restricted to functionals of $\btau$ and show their convergence to a continuum version (in the spirit of~\cite[Prop.~3.1]{BL21_scaling}), which then enables one to define the continuum model $\mu_{\hat \gb,\hat h}^{\cM}$ (in the spirit of~\cite[Thm.~2.3]{BL20_conti}).

\subsubsection{Extracting information on the critical point shift}
Similarly to what is argued in~\cite[Sec~1.3, \S2]{CSZ13},  one could be able to use the continuum partition to extract information on the free energy
$\tf_{\gamma}(\gb,h)$ in the weak-disorder limit, \textit{i.e.}\ when $\gb,h \to 0$. In particular, one can define the \textit{continuum free energy} as
\begin{equation}
\mathbf{F}_{\gamma}(\hat\gb,\hat h) := \lim_{\bt \to \infty, \frac{t_1}{t_2} \to \gamma} \frac{1}{t_1} \mathbb E\Big[ \log \mathbf{Z}_{\bt, \hat h}^{\hat\gb, \cM, \quen} \Big] \,.
\end{equation}
The fact that the limit exists and is finite is not immediate, but should follow from super-additivity and concentration arguments.
One is then led to conjecture that,  setting $\gb_{\epsilon} =  \hat \gb (\epsilon^{\alpha-\frac12} L(1/\epsilon)^{-1})^{\frac 1r}$ and $h_{\epsilon} = \hat h \epsilon^{\alpha} L(1/\epsilon)$ (see~\eqref{def:scalings}), we have that 
\begin{equation}
\tag{Conj.~1}
\label{conj:limits}
\lim_{\epsilon \downarrow 0}\epsilon^{-1} \tf_{\gamma} (\gb_{\epsilon},h_{\epsilon}) = \mathbf{F}_{\gamma}(\hat\gb,\hat h) \,.
\end{equation}
As argued in~\cite{CSZ13}, this amounts to exchanging the limits of \textit{infinite volume}, \textit{i.e.}\ letting the size of the system to infinity, and of \textit{weak disorder}, \textit{i.e.}\ letting the inverse temperature $\beta$ and the external field $h$ go to $0$.
This exchange of limits is in fact a delicate issue: it has been shown for instance in the context of the copolymer model with tail exponent $\alpha\in (0,1)$ in~\cite{BdH97,CG10} and for the pinning model with tail exponent $\alpha\in (\frac12,1)$ in~\cite{CTT17}; but is known not to hold for $\alpha>1$, see~\cite[Sec~1.3, \S3]{CSZ13} and~\cite{BCPSZ14}.

Analogously to what is done in~\cite{CTT17}, the exchange of limits~\eqref{conj:limits} would provide information on the behaviour of the critical point $h_c^{\quen}(\gb)$ defined in~\eqref{def:hcritic} in the weak-disorder limit $\gb\downarrow 0$. 
One should first prove that the critical point for the continuous model, defined by
\[
\mathbf{h}_c^{\quen}(\hat \gb) := \sup\big\{ \hat h \in \mathbb R ,  \mathbf{F}_{\gamma}(\hat\gb,\hat h)=0\big\} \,,
\] 
is positive and finite.
Then, using the scaling properties $\cM([0, c \bt]) \stackrel{\text{(d)}}{=} c^{3/2} \cM([0,\bt])$ and $\gp(c\bt) = c^{\alpha-2}\gp(\bt)$ for $\bt\succ 0$ and $c>0$, we get that $\bZ_{c \bt, \hat h}^{\hat \gb ,\cM, \quen} \stackrel{\text{(d)}}{=} \bZ_{ \bt, c^{\alpha} \hat h}^{c^{\frac1r (\alpha-\frac12)}\hat \gb ,\cM, \quen} $, recalling also that $\hat \gb$ appears with an exponent $r$ in the chaos expansion~\eqref{eq:conjchaosexpansion}. This in turns implies that
\[
\mathbf{F}_{\gamma} \big(c^{\frac1r (\alpha-\frac12)} \hat\gb, c^{\alpha}\hat h \big) = c\, \mathbf{F}_{\gamma} (\hat \gb, \hat h)\,,
\qquad \text{and} \qquad \mathbf{h}_c^{\quen}(\hat \gb) = \mathbf{h}_c^{\quen}(1)  \hat \gb^{\frac{2\alpha r}{2\alpha-1}} \,.
\]
Then, similarly to \cite[Thm.~2.4]{CTT17}, one could expect that a slightly stronger version of~\eqref{conj:limits} would yield the following \textit{universal} weak-disorder asymptotics 
\begin{equation}
\tag{Conj.~2}
 \lim_{\gb\downarrow 0} \frac{h_c^{\quen}(\gb)}{ \tilde L_{\alpha}(\frac{1}{\gb^r}) \gb^{\frac{2\alpha r}{2\alpha-1}} } = \mathbf{h}_c^{\quen}(1)  \,,
\end{equation}
where $\tilde L_{\alpha}$ is a slowly varying function, which is obtained by inverting the relation $\gb^r\sim n^{\frac12 -\alpha} L(n)$ as $n\sim \tilde L_{\alpha} (\gb^{-r} )^2 \gb^{ \frac{2r}{2\alpha-1}}$ as $\gb\downarrow0$, $n\uparrow\infty$ (so $h\sim \hat h L(n)n^{-\alpha}$ translates into $h \sim \hat h \tilde L_{\alpha} (\gb^{-r} ) \gb^{\frac{2\alpha r}{2\alpha-1}}$);
 see~\cite[Rem.~2.2]{CTT17} or \cite[Sec.~3.1]{CSZ13} for details in the context of the pinning model.
Let us stress that the constant $\mathbf{h}_c^{\quen}(1)$ depends only on $\alpha$
and not on the fine details either of the bivariate renewal (in particular not on the slowly varying function $L$) or of the disorder distribution $\go$, except trough the constant~$\sigma_r$.

\subsubsection{About an even more faithful modelling of DNA}
\label{sec:uniquesequence}
Recall that the aim of this paper is to study the gPS model with a disorder that represents the interaction between DNA strands. There are two directions in which this model could be extended to describe DNA even more faithfully.

First, one can consider the case of two complementary strands, that is where there is a single sequence of i.i.d.\ random variables $(\tilde \go_{i})_{i\geq 1}$
and where the disorder field is given by $\go_{\bi} := V(\tilde \go_{i_1},\tilde \go_{i_2})$. 
Define $\lambda(\gb) = \log \bbE[e^{\gb \go_{\bi}}]$ for $\bi$ not on the diagonal, \textit{i.e.}\ with $i_1\neq i_2$; 
note that in general we have $\log \bbE[e^{\gb \go_{\bi}}] \neq \lambda(\gb)$ for $\bi$ on the diagonal.
Define again $\zeta_{\bi} := e^{\gb\go_{\bi} -\lambda(\gb)} -1$ for any $\bi \in \bbN^2$; in particular, we have $\bbE[\zeta_{\bi}]=0$ if $\bi$ is not on the diagonal
and $\bbE[\zeta_{\bi}] \neq 0$ if $\bi$ is on the diagonal.

Now, note that $\go_{\bi}$ is independent of $\go_{\bj}$ except if $i_1=j_1$, $i_1=j_2$, $i_2=j_1$ or $i_2=j_2$;
in that case, we say that $\bi$ and $\bj$ are \emph{linked} and we write $\bi \leftrightsquigarrow \bj$.
We also separate the cases where $(i_1,i_2)=(j_1,j_2)$ or $(i_1,i_2)=(j_2,j_1)$, that is $\bi=\bj$ or $\bi$ is the symmetric of $\bj$ with respect to the diagonal, that  we denote as $\bi \rightleftharpoons\bj$. We provide a figure below: the lines represent the points that are \textit{linked} to $(i,j)$.

\begin{center}
\begin{tikzpicture}
\begin{tikzpicture}[scale=0.8]
\draw[dashed] (0,0) -- (4,4);
\draw[->] (0,0) -- (0,4);
\draw[->] (0,0) -- (4,0);
\draw (3,2.2) node[below right]{$(i,j)$};
\filldraw (3,2.2) circle (0.1);
\draw (2.2,3) circle (0.1);
\draw (2.2,3) node[above left]{$(j,i)$};
\draw[very thick] (2.2,0) -- (2.2,4);
\draw[very thick] (0,2.2) -- (4,2.2);
\draw[very thick] (3,0) -- (3,4);
\draw[very thick] (0,3) -- (4, 3);
\end{tikzpicture}
\end{tikzpicture}
\end{center}

Now, we can perform the same calculations as for Lemma~\ref{lem:correl}. One gets that
$\bbE\left[\zeta_{\bi}\zeta_{\bj}\right]=0$ 
if $\bi\not\leftrightsquigarrow \bj$ and that, 
for $\bi$ and $\bj$ not on the diagonal and $\bi \not \rightleftharpoons \bj$, as $\gb_n\to 0$ we have
\begin{equation}
\label{eq:correl:zeta3}
\bbE\left[\zeta_{\bi}\zeta_{\bj}\right] = 
\begin{cases}
\sigma_r^2\gb_n^{2r} + o(\gb_n^{2r})\quad &\text{if }\bi \leftrightsquigarrow \bj \text{ and }\bbP\in\Pfk_r ,\\
0 \quad &\text{if }\bi \leftrightsquigarrow \bj \text{ and }\bbP\in\Pfk_{\infty} , \\
\end{cases}
\end{equation}
One can also obtain estimates on $\bbE\left[\zeta_{\bi}\zeta_{\bj}\right]$ when $\bi$ and $\bj$
are on the diagonal or $\bi  \rightleftharpoons \bj$.
One should be able to adapt the proof of the convergence in Theorem~\ref{thm:cvgcM}, 
with a different covariance structure due to
the fact that
the correlations occur in a more intricate way (using the relation $\bi \leftrightsquigarrow \bj$ instead of $\bi\leftrightarrow \bj$).
After some calculations, we expect that the (rescaled) field $(\zeta_{\bi})_{\bi\in \N^2}$ converges to a Gaussian field with covariance function given by
\begin{equation}
\label{def:QcovcM}
Q(\bs,\bt) =  2 x^{(1)} x^{(2)} x^{(3)} + x^{(1)} x^{(4)} (x^{(2)} + x^{(3)} ) \,,
\end{equation}
with $x^{(1)}<x^{(2)}<x^{(3)}<x^{(4)}$ the ordered points of $t_1,t_2,s_1,s_2$.
A realization of such a Gaussian field is presented in Figure~\ref{fig:M2}.  Finally, a reasonable conjecture is that the statement of Theorem~\ref{conj:scalinggPS} also holds in this setting when replacing the field $\cM$ with the one described above. However, proving this result should involve even more technicalities than in our  setting (see in particular Sections~\ref{sec:convM},~\ref{sec:higherrank} and~\ref{sec:prooftronck>1} below) because of the more complex combinatorics appearing in the correlations. This is the reason why we do not develop on this further.

\begin{figure}[thbp]
\vspace{-0.7\baselineskip}
\begin{center}
\begin{tabular}{cc}
\begin{tabular}{c}
\includegraphics[scale=0.4]{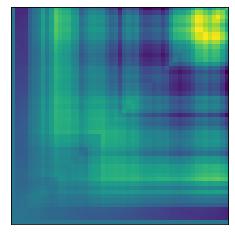}
\end{tabular}
\qquad
& 
\qquad
\begin{tabular}{c}
\includegraphics[scale=0.4]{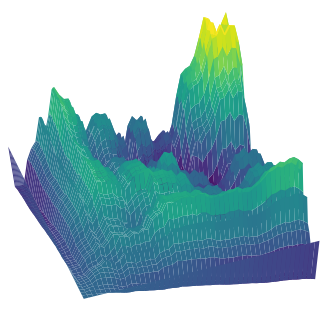}
\end{tabular}
\end{tabular}
\end{center}
\vspace{-0.7\baselineskip}
\caption{\small A realization of a Gaussian field with covariance $Q(\cdot,\cdot)$ defined in \eqref{def:QcovcM}. 
The covariance structure is different from the Gaussian field $\cM$,
to be compared with Figure~\ref{fig:M}.
}
\label{fig:M2}
\end{figure}
Second, one can replace the base pairing interaction with \emph{stacking energies} (see e.g.~\cite[App. D]{GO04}), which depend on pairs of adjacent bases in the two strands (i.e. when $i$ is pinned with $j$ and $i+1$ with $j+1$). These are known to describe nucleotides interactions more precisely than base pairing, and they can be modelled mathematically by assuming that the sequences $\hat\go$, $\bar\go$ are 1-dependent. We do not expect this assumption to yield results any different to the independent case, however proving it should also involve much more technicalities than in our setting; which is why we also refrain from it in this paper.

\subsubsection{About other models with long-range correlated disorder}
The question of the influence of disorder with long-range correlations on physical systems has been addressed widely in the physical literature, starting with the seminal paper by Weinrib and Halperin~\cite{WH83}.
In \cite{WH83}, the authors propose a modification of Harris' predictions on disorder relevance, depending on the rate of decay of the two-point correlation function: namely, if the disorder verifies $\bbE[\go_x \go_y] \asymp \|x-y\|^{-a}$ for some $a>0$, then disorder should be irrelevant if $\nu> 2/\min(d,a)$ and relevant if $\nu<2/\min(d,a)$, with $\nu$ the critical exponent of the homogeneous model.
In other words, Harris' criterion is modified if $a<d$.

As far as the standard (one-dimensional) pinning model is concerned, this question has been investigated in the mathematical literature, for instance  in~\cite{B13,Ber14,BL12,BP15,BT13,CCP19,P13}.
In particular, it has been proven in \cite{Ber14,BL12} that Weinrib--Halperin's prediction fail: disorder becomes always relevant as soon as $a<d=1$. The main idea is that the two-point correlations do not encapsulate the important features of the environment; instead one has to study the rare appearance of large regions of favourable disorder.
However, one could still hope to recover Weinrib--Halperin's predictions for the existence of a non-trivial intermediate disorder scaling limit of the model (at least for Gaussian disorder), tuning down the inverse temperature $\gb_n \downarrow 0$ at the correct scale.

As a first step, one would need to make sense of the following continuum partition function, which is the natural candidate for the limit of the (free) partition function:
\begin{equation}
\label{pinningcorrel}
\bZ_{\gb}(t) = 1 + \sum_{k=1}^{\infty} \gb^k \int_{0<s_1<\cdots<s_k<t}  \prod_{i=1}^{k} \vartheta(s_i-s_{i-1}) \prod_{i=1}^{k} W(\dd s_i) \,.
\end{equation}
In the above expression, $\vartheta(s) = s^{\alpha-1}$ corresponds to the scaled renewal mass function $\bP(i\in \tau)$ (with $\alpha \in (0,1)$ the tail exponent of $\bP(\tau_1>i)$, verifying $\alpha =1/\nu$) and $W$ a fractional Brownian Motion with Hurst index $H \in (\frac12,1)$, \textit{i.e.}\ a Gaussian field with covariance function $\bbE[W_sW_t] = \frac12 (|s|^{2H}+|t|^{2H}-|t-s|^{2H})$, which corresponds to the scaling limit of a Gaussian field $(\go_n)_{n\in \mathbb Z}$ with correlations $\mathbb E[\go_n,\go_{n+k}] \sim c |k|^{-a}$, $a=2(1-H) \in (0,1)$.
Then, one is able to compute (or at least estimate) the $L^2$ norm of each term in the sum. In particular, for $k=1$, one gets
\[
\left\| \int_{0<s<t} \vartheta(s) W(\dd s)\right\|_{L^2}^2  = c_{H} \int_{0<s'<s<1} \!\!\!s'^{\alpha-1}s^{\alpha-1} (s-s')^{2H-2} \dd s \dd s' 
=  c_H \frac{\Gamma(\alpha) \Gamma(2H-1)}{\Gamma(\alpha+ 2H-1)}  \int_{0}^1  s^{2\alpha+2H-3}  \dd s \,,
\]
which is finite if and only if $2\alpha +2H-2 =2/\nu - a  >0$, that is if and only if $\nu<2/a$, recovering Weinrib--Halperin's condition for disorder relevance.
This suggests that the expansion~\eqref{pinningcorrel} makes sense when $\nu<2/a$. However, when controlling the $L^2$ norm of the $k$-th term, we obtain a bound that is not summable in $k$.
Therefore new ideas are needed in order to decide whether it is possible to make sense of~\eqref{pinningcorrel} when $\nu<2/a$; if so, this would confirm Weinrib--Halperin's predictions, in the sense put forward in~\cite{CSZ16}.

Let us also mention that the effect of long-range correlations in the disorder has been studied in the directed polymer model, for instance in
\cite{Lac11,Rang20}. However, in these references, the disorder displays correlations only in the spatial dimension and not in the time dimension; it remains an open problem to study the model with correlations in time.

\subsection{Organisation of the rest of the paper}

Henceforth, the paper is organized as follows. In Section~\ref{sec:M} we prove Lemma~\ref{lem:correl}, from which we deduce the convergence of arbitrary moments of the rescaled field $M_n$ to those of $\cM$. Theorem~\ref{thm:cvgcM} then follows from standard arguments of finite-dimensional convergence and tightness. Let us also mention that Claim~\ref{claim:cvgps} below ensures us that in the remainder of the paper, we may reduce all questions of convergence, notably the main theorem, to $L^2$-convergences on a convenient $L^2(\bbP)$ space. 

In Section~\ref{sec:integration}, we discuss the integration against the random field $\cM$. We first provide general results for defining the integral against a random field by using its \emph{covariance measure}, notably Theorem~\ref{thm:stocint}. In Section~\ref{sec:covarM} we apply these to prove the well-posedness of $\int \phi \,\dd \cM$ in~\eqref{eq:convergencek=1}, and in Section~\ref{sec:higherrank} we proceed similarly for iterated integrals. In particular we prove that the series $\mathbf{Z}$ in~\eqref{eq:conjchaosexpansion} is a well-defined $L^2(\bbP)$ random variable.

Section~\ref{sec:proofthm} contains the most technical parts of the paper, which are required to prove Theorem~\ref{conj:scalinggPS}. Section~\ref{sec:reduc_k=0} shows that we may reduce the statement to the case $h_n\equiv0$, Sections~\ref{sec:proofthm_discrint}--\ref{sec:proofthm_k} prove the convergence of any term of the polynomial expansion~\eqref{expansion} to its continuous counterpart in~\eqref{eq:conjchaosexpansion}, and Section~\ref{sec:proofthm_mainproof} concludes with the convergence of the whole partition function to the series $\mathbf{Z}$.

Finally, Section~\ref{sec:homogeneous} displays the proofs of statements regarding the homogeneous gPS model (Proposition~\ref{prop:scalinghom}), and when the limit is trivial (Proposition~\ref{prop:degenerate}). Those results are postponed to the end of the paper since they rely on standard techniques, namely Riemann-sum convergences and estimates on bi-variate renewal processes.

Some estimates on bi-variate renewal processes and bivariate homogeneous pinning models are recalled in Appendix~\ref{app:renewal}.
In Appendix~\ref{app:intstoch} we prove Theorem~\ref{thm:stocint}: similar results can already be found in the literature (see \emph{e.g.} \cite[Theorem~2.5]{W86}), but for the sake of completeness we provide a full construction of the integral with covariance measures. In Appendix~\ref{app:example} we eventually provide examples of disorder distributions in $\Pfk_4$, $\Pfk_8$, as claimed in Example~\ref{example3}.

\section{Convergence of the field \texorpdfstring{$M_n$}{Mn} to \texorpdfstring{$\cM$}{M}: proof of Theorem~\ref{thm:cvgcM}}
\label{sec:M}

Let us comment on the meaning of the convergence in Theorem~\ref{thm:cvgcM}.
Define
\begin{equation}
L^\infty([\bzero,\bt])\,:=\, \big\{f:[\bzero,\bt]\to\R \ ;\ \text{$f$ is measurable and bounded}\big\}\,,
\end{equation}
and equip it with the norm $\|\cdot\|_\infty$ (and ensuing Borel sigma-algebra). Then, for any random variables $(W_n)_{n\geq 1}$ and $\cW$ in $L^\infty([\bzero,\bt])$, we have that $W_n$ converges in distribution to $\cW$ if for any bounded function $h:L^\infty([\bzero,\bt])\to\R$  that is continuous (for the aforementioned topology), $\lim_{n\to\infty}\bbE[h(W_n)]=\bbE[h(\cW)]$. Notice that the fields $M_n$ and $\cM$ defined above are a.s.\ bounded so this convergence is well-posed. 
In this section, we prove the convergence in Theorem~\ref{thm:cvgcM} and we also provide useful estimates on $(M_n(\bs))_{\bs \succeq \bzero}$.

To be able to distinguish the different notation, the $L^p$ norms on spaces of functions from $[\bzero,\bt]$ to $\R$, \textit{i.e.}\ $L^p([\bzero,\bt])$, will be noted $\|\cdot\|_p$, whereas on spaces of real random variables, \textit{i.e.}\ $L^p(\bbP)$, they will be noted $\|\cdot\|_{L^p(\bbP)}$ or $\|\cdot\|_{L^p}$, $p\in[1,\infty]$.

\begin{remark}
In this section, we prove the convergence of $M_n$ (once rescaled) towards a Gaussian field $\cM$, by showing the convergence of its moments. Let us mention that CLT-like results such as Theorem~\ref{thm:cvgcM} can usually be proven via Stein's method. However, the versions of Stein's method that we are aware of are not completely sufficient in the specific case of the gPS model (for instance, \cite[Thm.~3.5]{Ros11} yields the convergence only when $r=1$, or when $\ga<5/6$ under the scaling~\eqref{def:scalings}). It would be interesting to obtain an alternate proof of Theorem~\ref{thm:cvgcM} for a general $r\geq1$ and $\ga<1$, that uses a refinement of Stein's method.
\end{remark}

\subsection{Preliminary results: the covariance structure}

We start with some preliminaries, controlling the covariances of the
field $(\zeta_{\bi})_{\bi\in \N^2}$:
we prove Lemma~\ref{lem:correl}, then we show how the covariance function $K(\bu,\bv)$ appears and prove some other useful estimates. Recall that $\zeta_{\bi} = \zeta_{\bi} (\gb_n):= e^{\gb_n \go_{\bi} - \lambda(\gb_n)} -1$ has mean $0$, that $\go_{\bi} = V(\hat\go_{i_1}, \bar\go_{i_2})$.
Recall also the definition~\eqref{eq:partPfk} of the sets $(\Pfk_r)_{r\geq 1}$ and $\Pfk_{\infty}$ partitioning the set $\Pfk$ of all distributions.

\begin{proof}[Proof of Lemma~\ref{lem:correl}]
First of all, if  $i_1\neq j_1$ and $i_2 \neq j_2$, then $\go_\bi$ and $\go_{\bj}$ are independent, so we clearly have that $\bbE[\zeta_{\bi}\zeta_{\bj}]=0$. 

When $\bi=\bj$, by a simple (and classical) Taylor expansion, we find 
\[
\bbE[\zeta_{\bi}\zeta_{\bj}] 
= e^{\lambda(2\gb_n) - 2\lambda(\gb_n)} -1 \sim  \Var(\go_{\bone})  \gb_n^2
\qquad \text{ as } \gb_n \downarrow 0\,.
\]

It remains to treat the case $\bi \leftrightarrow \bj$ but $\bi\neq \bj$.
Let us assume that $i_1 \neq  j_1$ but $ i_2= j_2$
and write for simplicity $\varpi_1:=\hat \go_{i_1}$, $\varpi_2:=\bar \go_{i_2}$
and $\varpi_3= \hat \go_{j_1}$: this way we have $\go_{\bi} = V(\varpi_1,\varpi_2)$
and $\go_{\bj} =V(\varpi_2,\varpi_3)$.
Then, we write
\begin{equation}
\label{eq:developexpo}
\begin{split}
e^{2\lambda(\gb_n)}  \bbE[\zeta_{\bi} \zeta_{\bj}] &=\bbE[e^{\gb_n (V(\varpi_1,\varpi_2)+V(\varpi_2,\varpi_3))}] - \bbE[e^{\gb_n V(\varpi_1,\varpi_2)} ] \bbE[e^{\gb_n V(\varpi_1,\varpi_2)}]  \\
& = \sum_{k=0}^{+\infty} \frac{\gb_n^k}{k!} \sum_{j=0}^{k} \binom{k}{j} \big( a_{j,k}-b_{j,k} \big) \,,
\end{split}
\end{equation}
where for the last equality we have expanded the exponentials,
developed $(V(\varpi_1,\varpi_2)+V(\varpi_2,\varpi_3))^k$ 
and set
\[
a_{j,k} := \bbE[V(\varpi_1,\varpi_2)^jV(\varpi_2,\varpi_3)^{k-j}] \,,\quad 
b_{j,k} :=\bbE[V(\varpi_1,\varpi_2)^j] \bbE[V(\varpi_2,\varpi_3)^{k-j}]  \,.
\]

Now, we show that for $\bbP \in \Pfk_r$,
if $j<r$ then $a_{j,k}=b_{j,k}$ (and similarly for $k-j<r$, by symmetry).
Indeed, by definition of  $\Pfk_r$ the random variable $\bbE[V(\varpi_1,\varpi_2)^j \,| \, \varpi_1]$ is constant a.s., equal to $\bbE[V(\varpi_1,\varpi_2)^j]$.
Therefore, if $j<r$, conditioning with respect to $\varpi_1$ we get
\[
\begin{split}
a_{j,k} = \bbE[ V(\varpi_1,\varpi_2)^jV(\varpi_2,\varpi_3)^{k-j}] & = 
\bbE[ \bbE[V(\varpi_1,\varpi_2)^j \,| \, \varpi_1] \, V(\varpi_2,\varpi_3)^{k-j}]  \\
& = \bbE[V(\varpi_1,\varpi_2)^j] \bbE[V(\varpi_2,\varpi_3)^k] =b_{j,k}\,.
\end{split}
\]
Note that this also holds if $r=+\infty$.

Therefore, if $\bbP \in \Pfk_{\infty}$, we have $a_{j,k}=b_{j,k}$ for all $j,k\geq 0$, so $\bbE[\zeta_{\bi} \zeta_{\bj}]=0$, as announced.
If $\bbP\in \Pfk_r$ with $r \in \bbN$, 
then we have $\sum_{j=0}^k \binom{k}{j} (a_{j,k}-b_{j,k}) =0$
for all $k<2r$, and for $k=2r$,
$\sum_{j=0}^{2r} \binom{k}{j} (a_{j,k}-b_{j,k}) = \binom{2r}{r} (a_{r,r}-b_{r,r})$,
with 
\[
\begin{split}
a_{r,r}-b_{r,r} &= \bbE[ V(\varpi_1,\varpi_2)^rV(\varpi_2,\varpi_3)^{r}] - \bbE[ V(\varpi_1,\varpi_2)^r]\bbE[V(\varpi_2,\varpi_3)^{r}] \\
&= \bbE\big[ \bbE[ V(\varpi_1,\varpi_2)^r \,|\, \varpi_2]^2 \big] -
 \bbE[ V(\varpi_1,\varpi_2)^r]^2 \\
 &= \Var\big( \bbE[ V(\varpi_1,\varpi_2)^r \,|\, \varpi_2]\big) =:\tilde \sigma_r^2 \,,
\end{split}
\]
where we have used that conditionally on $\varpi_2$,
$V(\varpi_1,\varpi_2)$ and $V(\varpi_2,\varpi_3)$ are independent, and by symmetry of $V$ we have $\bbE[ V(\varpi_1,\varpi_2)^r \,|\, \varpi_2]= \bbE[ V(\varpi_2,\varpi_3)^r \,|\, \varpi_2]$.
Going back to~\eqref{eq:developexpo}, for $\bbP\in \Pfk_r$, we get that as $\gb_n\downarrow 0$,
\[
e^{2\lambda(\gb_n)}  \bbE[\zeta_{\bi} \zeta_{\bj}]
 = (1+o(1)) \frac{\gb_n^{2r}}{(2r)!} \binom{2r}{r} \tilde\sigma_r^2 
 = (1+o(1)) \frac{\tilde\sigma_r^2 }{(r!)^2} \gb_n^{2r} \,.
\]
Since $e^{2\lambda(\gb_n)} \to 1$, this concludes the proof 
of Lemma~\ref{lem:correl} with $\sigma^2_r:=\tilde\sigma^2_r/(r!)^2$.
%
\end{proof}

Let us define, for $\bs \in \bbR_+\times \bbR_+$,
\begin{equation}
\label{def:Mbar}
\overline M_n (\bs) := \frac{1}{\sigma_r n^{3/2} \gb_n^r} M_n(\bs) =   \frac{1}{\sigma_r n^{3/2} \gb_n^r} \sum_{\bi \in \llb \bone ,n\bs \rrb } \zeta_\bi \, ,
\end{equation}
where $\frac{1}{\sigma_r n^{3/2} \gb_n^r}$ is the scaling advertised in Theorem~\ref{thm:cvgcM}. Thanks to Lemma~\ref{lem:correl}, we easily identify the covariance structure of $(\overline M_n(\bs) )_{\bs \succeq \bzero}$. Let $\bs=(s_1,s_2), \bt=(t_1,t_2) \in \bbR_+^2$, and let us compute 
\begin{equation}\label{eq:olMn:development} 
\bbE\left[\overline M_n( \bs) \overline M_{n}(\bt) \right] = \frac{1}{\sigma_r^2 n^3 \gb_n^{2r}} \sum_{\bi \in \llb \bone ,n\bs \rrb } \sum_{\bj \in \llb  \bone,n\bt \rrb} \bbE \left[\zeta_{\bi}\zeta_{\bj}\right].
\end{equation}
Then, in view of Lemma~\ref{lem:correl} (or~\eqref{eq:correl:zeta}), we distinguish in the sum indices $(\bi,\bj)$ that are equal (there are $(1+o(1))(s_1\wedge t_1)(s_2 \wedge t_2)\,  n^2$ of them), aligned indices that are not equal (there are $(1+o(1)) (s_1\wedge t_1)(s_2 \wedge t_2) (s_1\vee t_1 + s_2 \vee t_2) \, n^3$ of them, see Figure~\ref{fig:alignment}), and other indices (which do not contribute to the sum).

\begin{figure}[htbp]\begin{center}
\includegraphics[scale=0.6]{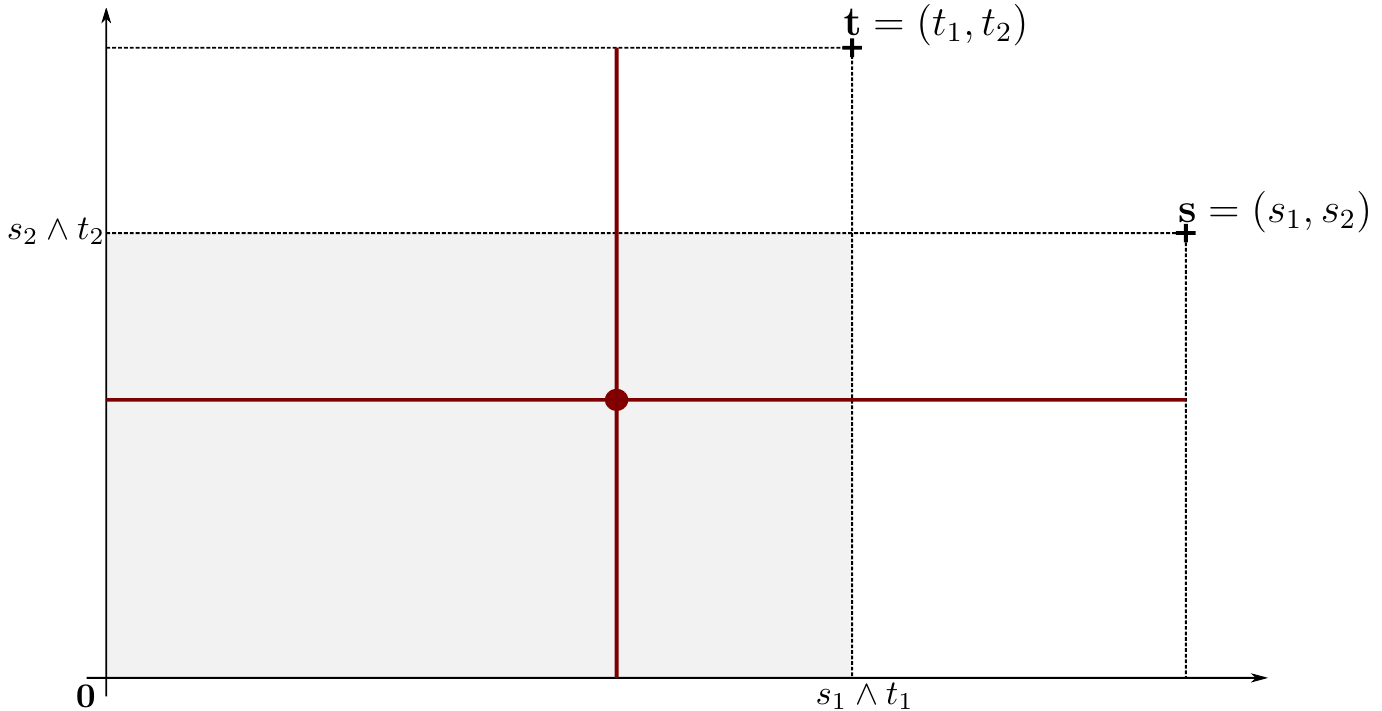}
\caption{ \footnotesize Graphical representation of indices $\bi \in \llb \bone, n\bs \rrb$, $\bj \in \llb \bone ,n \bt \rrb$ with $\bi \leftrightarrow \bj$ and $\bi \neq \bj$. One of $\bi,\bj$ must be in $\llb \bone, n\bs \rrb \cap \llb \bone, n\bt \rrb$ (there are $(s_1\wedge t_1) (s_2\wedge t_2) n^2$ possible locations, represented as the red dot), and the other one has to be aligned with it (there are $(s_1\vee t_1 + s_2 \vee t_2)n -1$ possibilities, represented by the red lines).}
\label{fig:alignment}
\end{center}\end{figure}

Therefore, in view of~\eqref{eq:correl:zeta}, as $n\to+\infty$ the covariance $\bbE\big[\overline M_n( \bs) \overline M_{n}(\bt) \big]$ is asymptotic to
\[
(s_1\wedge t_1)(s_2 \wedge t_2)\frac{\sigma^2}{\sigma_r^2n \gb_n^{2r}}  \pt+\pt (s_1\wedge t_1)(s_2 \wedge t_2) (s_1\vee t_1 + s_2 \vee t_2) \,.
\]
If $\lim_{n\to\infty}n\gb_n^{2r} = +\infty$, which is one assumption of Theorem~\ref{thm:cvgcM}, then we end up with
\begin{equation}
\label{conv:covariance}
 \lim_{n\to+\infty}  \bbE\left[ \overline M_n( \bs) \overline M_{n}(\bt) \right]  =  K(\bs, \bt) := (s_1\wedge t_1)(s_2 \wedge t_2) (s_1\vee t_1 + s_2 \vee t_2) ,
\end{equation}
which is the correlation function in Theorem~\ref{thm:cvgcM}.

\smallskip
Let us conclude this section by giving a lemma
that gives estimates on multi-point correlations---this will appear useful in the rest of the paper.
Let us define the classes of sets of ``$m$-aligned'' indices (with possible repetitions of the indices) as
\begin{equation}
\label{def:Cm}
\mathcal{A}_{m}  := \left\{  (\bi_1, \ldots, \bi_m)  \in  (\bbN^{2})^m\ ;\ \begin{aligned} &\forall\ 1\leq k,k'\leq m, \exists k_0=k,k_1,\ldots,k_p=k'\\& \text{such that } \forall 1\leq a\leq p,\  \bi_{k_a}\aligne \bi_{k_{a-1}} \end{aligned}
\right\}\, ,
\end{equation}
in other words, $(\bi_1, \ldots, \bi_m)\in \mathcal{A}_m$ if for all pair $(\bi_k,\bi_{k'})$, $1\leq k,k'\leq m$, there is a path from $\bi_k$ to $\bi_{k'}$ of subsequently aligned indices in $\{\bi_1, \ldots, \bi_m\}$. 
We refer to Figure~\ref{fig:m-aligned} below for an illustration of sets that are $m$-aligned.

 Moreover, a specific type of $m$-aligned sets will play an important role in the computation of the scaling limit of the partition function below, which may be formed by the union of two renewal trajectories $\btau,\btau'\subset\N^2$. A $m$-aligned set $(\bi_1,\ldots,\bi_m)\in\cA_m$ is called a \emph{$m$-chain of points} if there is no repetition of the indices and if we may reorder it into a non-decreasing sequence $\bi_1\preceq\bi_2\preceq\bi_3\preceq\ldots$ such that $\bi_1\prec\bi_3\prec\cdots$ and $\bi_2\prec\bi_4\prec\cdots$; an example is provided in Figure~\ref{fig:m-aligned}. For such sets, we improve our estimate on multi-point correlations by additionally controlling the dependence of the pre-factor in $m\in\N$.

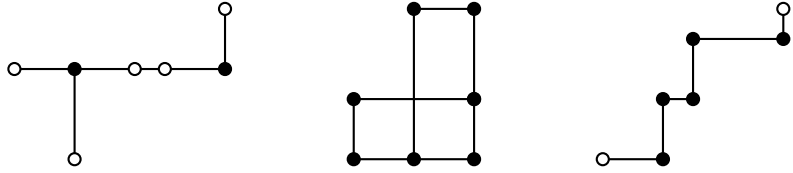
\begin{figure}[h!]
	\begin{center}
		\begin{tikzpicture}[scale=0.4]
			\draw[thick]  (1,0) -- (3,0) -- (5,0) -- (6,0) -- (8,0) -- (8,2) ;
			\draw[thick]  (3,0) -- (3,-3) ;
			\draw[thick, fill=white] (1,0) circle (0.2);
			\draw[thick, fill=black] (3,0) circle (0.2);
			\draw[thick, fill=white] (5,0) circle (0.2);
			\draw[thick, fill=white] (6,0) circle (0.2);
			\draw[thick, fill=black] (8,0) circle (0.2);
			\draw[thick, fill=white] (8,2) circle (0.2);
			\draw[thick, fill=white] (3,-3) circle (0.2);
		\end{tikzpicture}
		\qquad\qquad
		\begin{tikzpicture}[scale=0.4]
			\draw[thick]  (1,0) -- (5,0) -- (5,5) -- (3,5) -- (3,0);
			\draw[thick]  (1,0) -- (1,2) -- (5,2) ;
			\draw[thick, fill=black] (1,0) circle (0.2);
			\draw[thick, fill=black] (5,2) circle (0.2);
			\draw[thick, fill=black] (5,5) circle (0.2);
			\draw[thick, fill=black] (3,5) circle (0.2);
			\draw[thick, fill=black] (3,0) circle (0.2);
			\draw[thick, fill=black] (1,2) circle (0.2);
			\draw[thick, fill=black] (5,2) circle (0.2);
			\draw[thick, fill=black] (5,0) circle (0.2);
		\end{tikzpicture}
		\qquad\qquad
		\begin{tikzpicture}[scale=0.4]
			\draw[thick]  (0,0) -- (2,0) -- (2,2) -- (3,2) -- (3,4) -- (6,4) -- (6,5); 
			\draw[thick, fill=white] (0,0) circle (0.2);
			\draw[thick, fill=black] (2,0) circle (0.2);
			\draw[thick, fill=black] (2,2) circle (0.2);
			\draw[thick, fill=black] (3,2) circle (0.2);
			\draw[thick, fill=black] (3,4) circle (0.2);
			\draw[thick, fill=black] (6,4) circle (0.2);
			\draw[thick, fill=white] (6,5) circle (0.2);
		\end{tikzpicture}
	\end{center}
	\caption{\footnotesize Examples of sets that are $m$-aligned with $m=7$. 
	The indices that are alone on their column or on their line are represented with (empty) circles, the ones that are aligned with another index both vertically and horizontally are represented by filled dots. If we denote $b$ the number of points alone on their line or column then we have from left to right $b=5$, $b=0$, $b=2$.   More specifically, the third example is a $m$-chain of points. }
\label{fig:m-aligned}
\end{figure}

\begin{lemma}
\label{lem:multicorrel}
Assume that $\bbP\in \Pfk_r$ for some $r\in \bbN$.
For any $m\geq 2$ and $q_1,\ldots, q_{m} \in \bbN$, there is a constant $C=C_{m,q}$ such that
for any set $\bI=(\bi_1,\bi_2,\ldots, \bi_m) \in \cA_m $ of \emph{distinct} $m$-aligned points
we have
\begin{equation}
\label{eq:multicorrel}
0\leq \bbE\Big[ \prod_{p=1}^{m}  \zeta_{\bi_p}^{q_p} \Big]  \leq C\, \gb_n^{R}  \qquad \text{with }R :=  \sum_{p=1}^{m} q_{p} \vee r_p\,,
\end{equation}
where $r_p:=r$ if $\bi_p$ is alone on its line or on its column (in other words if $\bi_p^{(2)} \neq \bi_j^{(2)}$ for all $j\neq p$
or if $\bi_p^{(1)} \neq \bi_j^{(1)}$ for all $j\neq p$),
and $r_p = \lceil \frac r2 \rceil$ otherwise.

In addition, there exists a constant $C'>0$ such that for any $m\geq 2$, for any $m$-chain of points $\bI=(\bi_1,\bi_2,\ldots, \bi_m) \in \cA_m $, 
 we have
\begin{equation}
\label{chaincorrel}
0\leq \bbE\Big[ \prod_{p=1}^{m}  \zeta_{\bi_p} \Big] \leq (C')^m \gb_n^{2r + (m-2) \lceil \frac r2 \rceil} \,.
\end{equation}
\end{lemma}

\begin{proof}
The fact that correlations are non-negative simply follows from the observation that $\go\mapsto e^{\gb_n\go-\gl(\gb_n)}-1$ is non-decreasing and from the FKG inequality (see e.g. \cite{P74}), so we only have to prove the upper bound.

Let us
write for simplicity
$\varpi_{p,1} = \hat \go_{i_p^{(1)}}$,
$\varpi_{p,2} = \bar \go_{i_p^{(2)}}$:
this way, we have
 $\go_{\bi_p} = V(\varpi_{p,1},\varpi_{p,2})$.
Therefore, we can write
\begin{equation}
\label{expandzeta}
e^{\lambda(\gb_n) } \zeta_{\bi_p}  = e^{\gb_n V(\varpi_{p,1},\varpi_{p,2})} - \bbE[e^{\gb_n V(\varpi_{p,1},\varpi_{p,2})} ] = \sum_{k=0}^{+\infty} \frac{\gb_n^k}{k!} \Big(  V(\varpi_{p,1},\varpi_{p,2})^k  - \bbE\big[ V(\varpi_{p,1},\varpi_{p,2})^k\big] \Big)\,.
\end{equation}
Hence, expanding the power $q_p$ we have
\[
\Big(e^{\lambda(\gb_n) } \zeta_{\bi_p} \Big)^{q_p}
 = \sum_{\ell=0}^{+\infty} \gb_n^\ell W_{p,\ell} 
\]
where we have set
\[
W_{p,\ell} := \sum_{k_1 + \cdots + k_{q_p} =\ell} \frac{1}{\prod_{j=1}^{q_p} k_j!} \prod_{j=1}^{q_p} 
 \Big( V(\varpi_{p,1},\varpi_{p,2})^{k_j}  - \bbE\big[ V(\varpi_{p,1},\varpi_{p,2})^{k_j}\big]\Big)\,,
\]
and notice that $W_{p,\ell} =0$ for any $\ell<q_p$.
We can now expand the product of series and take the expectation: we get that
\begin{align}
\label{expansionofexpansion}
e^{\sum_{p=1}^{m} q_p\lambda(\gb_n) }
\bbE\Big[ \prod_{p=1}^{m}  \zeta_{\bi_p}^{q_p} \Big] 
= \sum_{\ell=0}^{+\infty} \gb_n^\ell  \sum_{\ell_1+\cdots +\ell_{m} = \ell} \bbE\Big[ \prod_{j=1}^{m} W_{j,\ell_j} \Big] \,.
\end{align}

Now, 
using that for any $k<r$ $\bbE[V(\varpi_{p,1},\varpi_{p,2})^{k}\,|\, \varpi_{p,2}] =\bbE[ V(\varpi_{p,1},\varpi_{p,2})^{k}]$ a.s.\ by definition of $\Pfk_r$,
one can easily check that $\bbE[W_{p,\ell} \,|\, \varpi_{p,2}] =0$ a.s.\ for any $\ell<r$.
With this in mind, if $\bi_p$ is alone on its line (\textit{i.e.}\ $\bi_p^{(2)} \neq  \bi_j^{(2)}$ for all $j\neq p$), 
since $\varpi_{p,1}$ appears only in $W_{p,\ell_p}$,
 conditioning with respect to $(\varpi_{j,1})_{j\neq p}$ and $(\varpi_{j,2})_{1\leq j\leq m}$,
 we get
\[
\bbE\Big[ \prod_{j=1}^{m} W_{j,\ell_j} \Big]  =
\bbE\Big[ \bbE[W_{p,\ell_p} \,|\, \varpi_{p,2}] \prod_{ j=1 ,j\neq p}^{m} W_{j,\ell_j} \Big]  = 0 \quad \text{ if } \ell_p <r \,.
\]
This obviously holds also in the case where $\bi_p$ is alone on its column
since then $\varpi_{p,2}$ appears only in $W_{p,\ell_p}$.
However, we cannot use the same trick if both
$\varpi_{p,1}$ and $\varpi_{p,2}$ appear in other terms $W_{j,\ell_j}$ with $j\neq p$:
in that case, 
we use Cauchy--Schwarz inequality to get
\[
\Big|\bbE\Big[ \prod_{j=1}^{m} W_{j,\ell_j} \Big] \Big|
\leq \bbE\Big[  W_{p,\ell_p}^2\Big]^{1/2} \bbE\Big[ \prod_{j \neq p} W_{j,\ell_j}^2 \Big]^{1/2}
=0 \quad \text{ if } 2\ell_p<r \,.
\]
where we have used that, analogously as above, $\bbE[(W_{p,\ell})^2 \,|\, \varpi_{p,2}] =0$ a.s.\ for any $2\ell<r$.

Overall, we get that $\bbE[ \prod_{j=1}^{m} W_{j,\ell_j} ]=0$
if there is some $j$ such that $\ell_{j}<q_j \vee r_j$, 
with $r_j$ as defined in the statement of the lemma.
Therefore, the first non-zero term in the series~\eqref{expansionofexpansion} is (possibly) for $\ell_{j} =q_j \vee r_j$:
and since $\exp(-\lambda(\gb_n))$ is bounded by~$1$, this concludes the proof of~\eqref{eq:multicorrel}.

\smallskip
For the second part of the lemma (in which $q_p=1$ for all $p$), note that starting from~\eqref{expandzeta}, we have similarly to~\eqref{expansionofexpansion}
\begin{equation}
\label{expandcorrel1}
e^{m\lambda(\gb_n) }
\bbE\Big[ \prod_{p=1}^{m}  \zeta_{\bi_p}\Big] 
= \sum_{k_1, \cdots , k_{m} \geq 0} \frac{(\beta_n)^{k_1+\cdots +k_m}}{\prod_{p=1}^{m} k_p!} \bbE\Big[ \prod_{p=1}^{m} Y_{p,k_p} \Big]\,,
\end{equation}
with $Y_{p,k}:=V(\varpi_{p,1},\varpi_{p,2})^{k}  - \bbE\big[ V(\varpi_{p,1},\varpi_{p,2})^{k}\big]$.
Then, exactly as above, the sum can be restricted to $k_p\geq r_p$ for all $p$, with here $r_1=r_m=r$ (the first and last index of the chain are both aligned with only one other index), and otherwise $r_p=\lceil \frac r2 \rceil$ for $2\leq p\leq m-1$.
Now, note that, using Cauchy--Schwarz inequality, we get that
\[
\bbE\Big[ \prod_{p=1}^{m} Y_{p,k_p} \Big] \leq \bbE\Big[ \prod_{p \text{ even}} (Y_{p,k_p})^2 \Big]^{1/2} \bbE\Big[ \prod_{p \text{ odd}} (Y_{p,k_p})^2 \Big]^{1/2} \leq  \prod_{p=1}^m\bbE\big[ (Y_{p,k_p})^2 \big]^{1/2}\,,
\]
where we used that $(Y_{p,k_p})_{p \text{ even}}$, resp.\ $(Y_{p,k_p})_{p \text{ odd}}$,  are independent, thanks to the structure of $\bI$ (the indices $\bi_p$ for $p$ even, resp.\ $p$ odd, are strictly increasing).
Since by assumption $V(\varpi_{p,1},\varpi_{p,2})$ admits a finite $\beta_0/2$ exponential moment, we get that there is a constant $C>0$ such that $\bbE[ V(\varpi_{p,1},\varpi_{p,2})^k]\leq C^k$ for all $k\geq 1$, hence $\bbE[(Y_{p,k_p})^2]^{1/2}\leq (C')^{k_p}$.
We therefore end up with 
\[
\bbE\Big[ \prod_{p=1}^{m}  \zeta_{\bi_p}\Big] 
= \sum_{k_1\geq r_1, \cdots , k_{m} \geq r_m} \frac{(C' \beta_n)^{k_1+\cdots +k_m}}{\prod_{p=1}^{m} k_p!}
\leq (C'\gb_n)^{\sum_{i=1}^m r_i} e^{m C' \gb_n}\,,
\]
 where we used that $\sum_{k\geq r}\frac{a^k}{k!}\leq a^r\sum_{k\geq0}\frac{a^k}{k!}=a^re^a$ for $a\geq0$.  Recalling that we already observed above that $\sum_{i=1}^m r_i =  2r + (m-2) \lceil \frac r2 \rceil$, this concludes the proof of~\eqref{chaincorrel}.
\end{proof}

\subsection{Finite-dimensional convergence}
\label{sec:convM}

In this section, we prove the convergence in distribution of  $(\overline M_n(\bs_1), \ldots , \overline M_n(\bs_m))$, for any $m\in \bbN$ and $\bs_1,\ldots ,\bs_m \in \bbR_+^2$.
Let $\Sigma_{\bs_1,\ldots ,\bs_m} (i,j):=  K(\bs_{i},\bs_{j})$ be the covariance matrix of $(\cM(\bs_1), \ldots , \cM(\bs_m))$, where we recall that $K(\cdot, \cdot)$ is defined in \eqref{def:KcovcM}.
In view of \eqref{conv:covariance}, $\Sigma$ is the limit of the sequence of the covariance matrices of $(\overline M_{n}(\bs_1), \ldots, \overline M_n(\bs_m))$; in particular, it is positive semi-definite.


\begin{proposition}
\label{prop:lim:M}
Let $m\in\N$ and $\bs_1,\ldots ,\bs_m \in \bbR_+^2$. As $n\to\infty$, if $lim_{n\to\infty}\gb_n = 0$ and $\lim_{n\to\infty}n\gb_n^{2r} = +\infty$, then
$( \overline M_{n}(\bs_1), \ldots, \overline M_n(\bs_m))$ converges in distribution toward a Gaussian vector, centered and with covariance matrix $\Sigma_{\bs_1,\ldots ,\bs_m}$.
\end{proposition}

Before we prove this proposition, let us start with the case $m=1$, which already encapsulates the combinatorial difficulty and will ease the understanding of the general case $m\in \bbN$.
We show the convergence of the moments of $\overline M_n(\bs)$ to the moments of a Gaussian variable, which implies the convergence in distribution.

\begin{lemma}\label{lem:moment:Mn}
Let $\ell\in\N$ and $\bs \in \bbR_+^2$. Then $\bbE\big[(\overline M_n(\bs))^\ell\big]$ is well defined if $n$ is large enough, and if $lim_{n\to\infty}\gb_n = 0$, $\lim_{n\to\infty}n\gb_n^{2r} = +\infty$, then we have
\begin{equation}\label{eq:moment:M_z}
\lim_{n\to+\infty}\bbE \Big[\big(\overline M_n(\bs)\big)^\ell \Big] \:=\:
\begin{cases}
0 & \quad \text{ if } \ell \text{ is odd},\\
\big(K(\bs,\bs) \big)^{\ell/2} \, \frac{\ell!}{ 2^{\ell/2} (\ell/2)!} & \quad \text{ if } \ell \text{ is even.}
\end{cases}
\end{equation}
where $K(\bs,\bt)$ is defined in \eqref{def:KcovcM}, so $K(\bs,\bs) = s_1 s_2 (s_1+s_2)$.
\end{lemma}

\begin{proof}
We write 
\begin{equation}
\label{develop:ell}
\bbE \Big[\big(\overline M_n(\bs)\big)^\ell \Big] =  \Big( \frac{1}{\sigma_r n^{3/2} \gb_n^r}\Big)^{\ell} \sum_{\bi_1 \in \llb \bone ,n\bs \rrb } \cdots  \sum_{\bi_\ell \in \llb \bone ,n\bs \rrb } \bbE\left[\zeta_{\bi_1}\cdots\zeta_{\bi_{\ell}}\right].
\end{equation}
Now, notice that $\bbE\left[\zeta_{\bi_1}\cdots\zeta_{\bi_{\ell}}\right]$ depends only on the relative positions of the indices $(\bi_k)_{k=1}^\ell$. For instance, if one of the $\bi_{k}$ is isolated (\textit{i.e.}\ not aligned with any other index $\bi$) then the expectation is equal to $0$.

Recall the definition~\eqref{def:Cm} of classes of sets of ``$m$-aligned'' indices (with possible repetitions of the indices).
Then, for any $\bI=(\bi_1 ,\ldots, \bi_{\ell})\in  \llb \bone ,n\bs \rrb ^\ell$, there is a unique partition $\mathbb J = \{J_1,\ldots, J_k\}$ of $\{1,\ldots,\ell\}$ 
 such that for $1\leq a\leq k$, $\{\bi_j\}_{j\in J_a}$ is a
 maximal set of ``$m$-aligned'' indices of $\bI$;
in particular $\{\bi_j \}_{j\in J_a} \in \cA_{|J_{a}|}$ for all $1\leq a\leq k$, and $\bi_j \not \leftrightarrow \bi_{j'}$ for any $j \in J_a$, $j' \in J_{b}$ with $a\neq b$.
One can view $(J_a)_{1\leq a \leq k}$ as  equivalence classes,
for the following equivalence relation (defined for $\bI$ fixed): 
$j  \rightleftarrows j'$ if and only if there exists a path $j_0=j,j_1,\ldots,j_q=j'$ in $\{1,\ldots,\ell\}$ satisfying $\bi_{j_p}\aligne\bi_{j_{p+1}}$ for all $0\leq p< q$.
For $\bI=(\bi_1 ,\ldots, \bi_{\ell})\in  \llb \bone ,n\bs \rrb ^\ell$, we denote $\Phi(\bI)=\bbJ$ this partition.
For $J\subset\{1,\ldots,\ell\}$ we let $\Phi_\bI^{-1}(J)\subset\bI$ be the set $\{\bi_{j}, j\in J\}$, with possible repetition of the indices: this way,
any partition $\mathbb J = \{J_1,\ldots, J_k\}$ of $\{1,\ldots,\ell\}$ induces a partition $\{\Phi_\bI^{-1}(J_1),\ldots,\Phi_\bI^{-1}(J_s)\}$ of $\bI$;
if $\mathbb J =\Phi(\bI)$, this corresponds to the partitioning of $\bI$ into maximal sets of ``$m$-aligned'' indices.

%
%
Therefore, if $\Phi(\bI) = \mathbb J =\{J_1,\ldots, J_k\}$, we may factorize
\begin{equation}
\label{eq:correlaligned}
\bbE\left[\zeta_{\bi_1}\cdots\zeta_{\bi_{\ell}}\right] = \prod_{a=1}^k \bbE\Big[ \prod_{j \in J_a} \zeta_{\bi_j} \Big] = \prod_{J\in \mathbb J} \bbE\Big[ \prod_{\bi \in \Phi_\bI^{-1}(J)} \zeta_{\bi} \Big] \, .
\end{equation}

As a consequence, denoting $\cJ_{\ell}$ the set of partitions of $\{1,\ldots, \ell\}$, we write
\begin{align}
\bbE \big[(\overline M_n(\bs))^\ell \big] =  \Big( \frac{1}{\sigma_r n^{3/2} \gb_n^r}\Big)^{\ell} 
  \sum_{ \mathbb J \in \cJ_{\ell} } \sumtwo{ \bI \in \llb \bone  , n \bs\rrb^{\ell} }{ \Phi(\bI) = \mathbb J}    \prod_{J\in \mathbb J} \bbE \Big[ \prod_{\bi \in \Phi_\bI^{-1}(J)} \zeta_{\bi} \Big]\, .
\label{eq:decomp-config}
\end{align}

First of all, in \eqref{eq:decomp-config}, we can restrict the sum to having only $|J| \geq 2$:  indeed if $|J|=1$ we obviously have $\bbE[\prod_{\bi \in \Phi_{\bI}^{-1}(J)}\zeta_{\bi} ]=\bbE[\zeta_\bone]=0$ (note that it also restricts the sum to $|\mathbb J|\leq \ell/2$).
Now, we claim that the main contribution in the sum comes from having $|J| =2$ for all $J\in \mathbb J$, or in other words from the term $|\mathbb J|= \ell/2$.

Indeed let us show that, for any $m \ge 3$,
\begin{equation}
\label{eq:m>3}
 \frac{1}{(\sigma_r n^{3/2}\gb_n^r)^m} \sum_{I\in \cA_m (n \bs)} \bbE\Big[   \prod_{\bi\in I}\zeta_{\bi}\Big]  \xrightarrow{n\to+\infty} 0 \, ,
\end{equation}
where we introduced the notation $\cA_m(n\bs) = \{I \in \cA_m , I \subset \llb \bone, n \bs \rrb\}$ (note that we still allow repetitions of the indices).
To that end, we perform a first simplification:
denote $\tilde \cA_m(n\bs)$ the set of $I\in \cA_m(n \bs)$ such that
all indices in $I$ are distinct.
Then, we clearly have that $|\cA_m \setminus \tilde \cA_m| \leq m |\cA_{m-1}|\leq C m \|\bs\|_1 n^{m}$ and hence
\[
 \frac{1}{(\sigma_r n^{3/2}\gb_n^r)^m} \sum_{I\in \cA_m (n \bs) \setminus \tilde \cA_m(n\bs)} \bbE\Big[  \prod_{\bi\in I}\zeta_{\bi}\Big]  \leq C_{r,m} \frac{1}{(n^{1/2}\gb_n^r)^m} \xrightarrow{n\to+\infty} 0 \, ,
\]
where we simply bounded $\bbE[   \prod_{\bi\in I}\zeta_{\bi}]$ by a constant and used that $n\gb_n^{2r}\to+\infty$.
We therefore only need to prove~\eqref{eq:m>3}
with $\cA_m(n\bs)$ replaced by $\tilde \cA_m(n\bs)$.
Now, with the idea of using Lemma~\ref{lem:multicorrel}, 
let us define (see Figure~\ref{fig:m-aligned} for an illustration)
\[
\tilde \cA_{m}^{(b)} (n\bs) := \Big\{ \{\bi_1,\ldots, \bi_m\} \in \tilde \cA_m(n\bs) \,, \text{ there are exactly $b$ indices alone on a line or a column} \Big\} \,.
\]
Now we just have to show~\eqref{eq:m>3} with $\tilde\cA_m^{(b)}(n\bs)$ in place of $\cA_m(n\bs)$, for any $b\in \{0,\ldots, m\}$.
Using Lemma~\ref{lem:multicorrel}-\eqref{eq:multicorrel}, we then have that for an $I\in \tilde \cA_{m}^{(b)} (n\bs)$ (recall all indices are distinct)
\[
\bbE \Big[ \prod_{\bi\in I}\zeta_{\bi} \Big] \leq  C\gb_n^{b r +(m-b)\lceil r/2 \rceil}  \leq  C \gb_n^{(b+m) r/2}  \,.
\]
Hence, using that $|\tilde \cA_{m}^{(b)} (n\bs)|\leq |\cA_{m}(n\bs)|\leq C \|\bs\|_1 n^{m+1}$, we have
\begin{equation}
\label{eq:sumA'}
 \frac{1}{(\sigma_r n^{3/2}\gb_n^r)^m} \sum_{I\in \tilde \cA_{m}^{(b)} (n\bs)} \bbE\Big[  \prod_{\bi\in I}\zeta_{\bi}\Big]  \leq  
\frac{C_{m,r}}{( n^{1/2}\gb_n^r)^m}  n \gb_n^{(b+m)  r/2}
 =  \frac{C_{m,r}}{( n^{1/2}\gb_n^r)^{m-2}}  \gb_n^{(b+m-4)  r/2}\,.
\end{equation}
Now, this goes to $0$ since $m\geq 3$ and  $n^{1/2}\gb_n^r \to \infty$, due also to the fact that 
$b+m \geq 4$: indeed, we cannot have $b=0$ if $m=3$ so we either have $b\geq 1$ or $m\geq 4$.

In addition to \eqref{eq:m>3}, recall that  when $m=2$ then~\eqref{conv:covariance} shows that 
\[
(\sigma_r n^{3/2}\gb_n^r)^{-m} \sum_{I\in \cA_m (n\bs)} \bbE \Big[   \prod_{\bi\in I}\zeta_{\bi} \Big]
\xrightarrow{n\to+\infty} K(\bs,\bs)\;,
\]  
in particular these terms are bounded. 
All together, for any fixed partition $\mathbb J\in \cJ_{\ell}$ with at least one $|J|\geq 3$, we have
\begin{equation}
\label{eq:m>3end}
\begin{split}
\frac{1}{(\sigma_rn^{3/2} \gb_n^r)^{\ell}} & \sumtwo{ \bI \in \llb \bone , n \bs \rrb^{\ell}}{\Phi(\bI)=\mathbb J} \prod_{J\in \mathbb J} \pt \bbE \Big[ \prod_{\bi \in \Phi_{\bI}^{-1}(J)} \zeta_{\bi} \Big]
 \leq\; \prod_{J\in \mathbb J} \Big( \frac{1}{(\sigma_rn^{3/2} \gb_n^r)^{|J|}} \sum_{ I \in \cA_{|J|}(n\bs) } \bbE \Big[ \prod_{\bi\in I}\zeta_{\bi} \Big] \Big)  \;\stackrel{n\to+\infty}{\longrightarrow}\; 0,
 \end{split}
\end{equation}
where we simply dropped the condition that $\bi_j \not \leftrightarrow \bi_{j'}$ if $j,j'$ are in different $J$'s.

\smallskip
Recall \eqref{eq:decomp-config}, where we have already said that we can restrict the sum to $\mathbb J \in \cJ_{\ell}$ having all $|J|\geq 2$. 

(i) If $\ell$ is odd, then it imposes that one $|J|$ is larger or equal than $3$. Hence $\bbE[(\overline M_n(\bs))^{\ell}]$ goes to $0$, which proves the first part of \eqref{eq:moment:M_z}.

(ii) If $\ell$ is even, then the only part contributing to the sum in \eqref{eq:decomp-config} comes from
$\mathbb J \in \cJ_{\ell}$ having all $|J|=2$:
we denote $\mathscr{P}_{\ell}$ the set of pairings of $\{1,\ldots, \ell\}$, i.e.\ the sets of partitions $\mathbb J\in \cP_{\ell} $ with $|J|=2$ for all $J\in \mathbb J$. We end up with
\begin{align}
\label{eq:elleven-0}
\bbE \big[(\overline M_n(\bs))^{\ell} \big] &=o(1)+  \frac{1}{n^{3\ell/2}} \sum_{ \mathbb J \in \mathscr{P}_{\ell} } \sumtwo{\bI \in \llb \bone, n\bs \rrb^{\ell} }{ \Phi(\bI) = \mathbb J}  \prod_{J\in \mathbb J} \frac{1}{\sigma_r^2 \gb_n^{2r}} \bbE \Big[ \prod_{\bi \in \Phi_{\bI}^{-1}(J)}\zeta_{\bi} \Big] \, .
\end{align}

Now, denote $\Upsilon_{\ell} = \{ (\bi_1,\ldots, \bi_\ell) \in (\bbN^2)^{\ell}, \exists \, j, j' \text{ s.t.} \bi_j =\bi_{j'} \}$. 
Analogously to \eqref{eq:m>3end}, we get that for a fixed $\mathbb J \in \mathscr{P}_{\ell}$,
the sum over $\bI \in \llb \bone , n \bs \rrb^\ell \cap \Upsilon_{\ell}$ with $\Phi(\bI) = \mathbb J$ goes to $0$:
indeed, there must be some $J =\{j,j'\}$ with $\bi_{j} =\bi_{j'}$, 
and Lemma~\ref{lem:correl} gives that $\bbE[ \zeta_{\bi_{j}}^2 ] = O(\gb_n^2)$, so that
$ (n^3 \gb_n^{2r})^{-1}\sum_{ \bi_{j} \in \llb \bone, n \bs \rrb} \bbE[ \zeta_{\bi_{j}}^2 ] = O( n^{-2} \gb_n^{2(1-r)})$ goes to $0$ (and all the other terms are bounded).

As a consequence, the restriction of the sum to $\bI \in \llb \bone , n \bs \rrb^\ell \cap \Upsilon_{\ell}$ in \eqref{eq:elleven-0} goes to $0$, and we have
\begin{equation}
\label{eq:elleven}
\bbE \big[(\overline M_n(\bs))^{\ell} \big]  = o(1) + \frac{1}{n^{3\ell/2}}  \sum_{ \mathbb J \in \mathscr{P}_{\ell} } \sumtwo{\bI \in \llb \bone, n\bs \rrb^{\ell} \cap \Upsilon_{\ell}^c }{ \Phi(\bI) =\mathbb J}  \prod_{J\in \mathbb J}  \frac{1}{\sigma_r^2\gb_n^{2r}} \bbE \Big[ \prod_{ \bi \in \Phi_{\bI}^{-1}(J)}\zeta_{\bi} \Big]\, .
\end{equation}
Then, Lemma~\ref{lem:correl} (or~\eqref{eq:correl:zeta3}) gives that $\bbE [\prod_{ \bi \in \Phi_{\bI}^{-1}(J)}\zeta_{\bi} ] =(1+o(1))  \sigma_r^2 \gb_n^{2r}$ for all $J$ in the product above.
Moreover, there are $((1+o(1)) s_1 s_2 (s_1+s_2) n^3)^{\ell/2}$ terms in the sum (recall Figure~\ref{fig:alignment}).
All together, we get the second part of \eqref{eq:moment:M_z}, using that the number of pairings $\mathrm{Card} (\mathscr{P}_{\ell})$ is the correct combinatorial factor
\end{proof}

\begin{proof}[Proof of Proposition~\ref{prop:lim:M}] We now prove the finite-dimensional convergence. 
 Let $m\in\N$, $\bs_1,\ldots,\bs_m\in \bbR_+^2$ and let $(\cM(\bs_1),\ldots, \cM(\bs_m))$ be a Gaussian vector of covariance matrix $\Sigma_{\bs_1,\ldots,\bs_m}$. 
We show that for any $u_1,\ldots, u_m \in \bbR$, $\big(\sum_{k=1}^m u_k \overline M_n(\bs_k) \big)_{n\geq1}$ converges in distribution to $\sum_{k=1}^m u_k \cM(\bs_k)$, by showing the convergence of its moments.
Let $\ell \in \bbN$, and let us compute
\begin{align}
\bbE\bigg[ \Big( \sum_{k=1}^m u_k \overline M_{n}(\bs_k) \Big)^{\ell} \bigg] = \sum_{k_1 =1}^m  \cdots \sum_{k_\ell =1}^m \bbE\Big[ \prod_{j=1}^{\ell}  u_{k_j} \overline M_{n}(\bs_{k_j}) \Big]\, .
\label{eq:powerell}
\end{align}

We fix $k_1,\ldots, k_{\ell} \in \{1,\ldots,m\}$ and we consider
\begin{align}
\label{develop:prodM}
\bbE\Big[ \prod_{j=1}^{\ell}  \overline M_{n}(\bs_{k_j}) \Big]
& = \Big(\frac{1}{\sigma_r n^{3/2}\gb_n^r} \Big)^{\ell}  \sum_{\bi_1 \in \llb \bone,\bs_{k_1} \rrb } \cdots \sum_{ \bi_{\ell} \in \llb \bone,\bs_{k_\ell} \rrb} \bbE\Big[  \prod_{j=1}^{\ell} \zeta_{ \bi_j} \Big] \, .
\end{align}
Then, we proceed as for the proof of Lemma~\ref{lem:moment:Mn}: to each $\ell$-uple $\bI = (\bi_1,\ldots , \bi_{\ell})$, we associate a partition $\Phi(\bI) = \mathbb J =\{J_1,\ldots , J_k\}$ 
 by decomposing $\bI$ into disjoint maximal ``$|J|$-aligned'' indices. As we showed above, cf.~\eqref{eq:m>3end}, the contribution of the terms with some $|J|\neq 2$ goes to $0$.
First of all, this implies that if $\ell$ is odd, then  $\bbE[ \prod_{j=1}^{\ell}  \overline M_{n}(\bs_{k_j})]$ goes to $0$ as $n\to +\infty$.
If $\ell$ is even, then analogously to \eqref{eq:elleven-0}-\eqref{eq:elleven}, the main contribution to \eqref{develop:prodM} comes from pairings $\mathbb J \in \mathscr{P}_{\ell}$, and from $\bI \in \llb \bone, n \bs_{k_1} \rrb \times \cdots \times \llb \bone , n \bs_{k_{\ell}}\rrb$ with distinct entries.
We therefore have that
\begin{align*}
\bbE\Big[ \prod_{j=1}^{\ell}  \overline M_{n}(\bs_{k_j}) \Big]& = o(1) + \frac{1}{n^{3\ell/2}}  \sum_{\mathbb J \in \mathscr{P}_{\ell} } \sumtwo{\bI \in  \llb \bone, n \bs_{k_1} \rrb \times \cdots \times \llb \bone , n \bs_{k_{\ell}}\rrb,  }{ \bI \notin \Upsilon_{\ell}, \,\Phi(\bI) = \mathbb J }  \prod_{J\in \mathbb J} \frac{1}{\sigma_r^2\gb_n^{2r}} \bbE \Big[ \prod_{ \bi \in \Phi_{\bI}^{-1}(J)} \zeta_{\bi} \Big] \\
& = (1+o(1))
\sum_{ \mathbb J \in \mathscr{P}_{\ell} } \prod_{J=\{j,j'\} \in \mathbb J}  K(\bs_{k_{j}} , \bs_{k_{j'}})\, .
\end{align*}
Here, we used again Lemma~\ref{lem:correl}, and that for any fixed $\mathbb J$, the number of terms in the sum over $\bI$ with $\Phi(\bI) = \mathbb J$ has $(1+o(1)) \prod_{J=\{j,j'\} \in \mathbb J}  K(\bs_{k_{j}} , \bs_{k_{j'}}) n^3 $ terms, in analogy with Figure~\ref{fig:alignment}.
 
Going back to~\eqref{eq:powerell}, we get that if $\ell$ is odd, then the $\ell$-th moment goes to $0$. If $\ell$ is even, we have that it is $(1+o(1)) $ times
 \begin{equation*}
\sum_{\mathbb J \in \mathscr{P}_{\ell} }  \sum_{k_1, \ldots, k_{\ell} =1}^m    \prod_{J=\{j,j'\} \in \mathbb J}  u_{k_{j}} u_{k_{j'}} K(\bs_{k_{j}} , \bs_{k_{j'}})
 = \sum_{\mathbb J \in \mathscr{P}_{\ell} }  \Big( \sum_{k,k' =1}^{m} u_k u_{k'}  K(\bs_{k} , \bs_{k'}) \Big)^{\ell/2}\!.
 \end{equation*}
 All together, we have shown that for any $u_1, \ldots, u_m \in \bbR$ and any even $\ell\in \bbN$
\begin{equation}\label{eq:lim:moments}
\bbE\bigg[ \Big( \sum_{k=1}^m u_k \overline M_{n}(\bs_k) \Big)^{\ell} \bigg] \xrightarrow{n\to+\infty} \frac{\ell!}{2^{\ell/2} (\ell/2)!}  \bigg(\sum_{k,k'=1}^m u_k u_{k'} K(\bs_k,\bs_{k'}) \bigg)^{\ell/2} \,.
\end{equation} 
(The limit is $0$ if $\ell$ is odd.) Note that the term raised to the power $\ell/2$ is the variance of $\sum_{k=1}^{m} u_k \cM(\bs_k)$. 
This shows that for any $\ell\in \bbN$,  the $\ell$-th moment of $\sum_{k=1}^m u_k \overline M_{n}(\bs_k)$ converges as $n\to \infty$ to the $\ell$-th moment of $\sum_{k=1}^{m} u_k \cM(\bs_k)$. Since $(\cM(\bs_k))_{1\leq k \leq m}$ is a Gaussian vector, this implies the convergence in distribution of $(\overline M_n(\bs_k))_{1\leq k \leq m}$ to $(\cM(\bs_k))_{1\leq k \leq m}$.
\end{proof}

\subsection{Convergence of $(\overline{M}_n(\mathbf{s}))_{ \mathbf{s}\in [\mathbf{0},\mathbf{t}]}$}

In this section we prove that the sequence $(\overline{M}_n(\bs) )_{ \bs\in [\bzero,\bt]}$ converges in distribution to $(\cM(\bs))_{ \bs\in [\bzero,\bt]}$. 
Let us first introduce a continuous interpolation of $\ol M_n$. For any $n\in\N$ and $\bs\in \bbR_+^2$, let $\bs^{[n]} := \tfrac1n \lfloor n\pt\bs\rfloor$ (not to be confused with the projection $\bs^{(a)}\in\R$, $a\in\{1,2\}$); notice that $\bs^{[n]}\in\frac1n\bbZ^2$, and $\|\bs^{[n]}-\bs\|_1\leq \frac2n$. Define for any $\bs\in[\bzero,\bt]$,
\begin{equation}
\label{tildeM}
\begin{split}
 \tilde M_n (\bs) =  &(1-\gamma_1)(1-\gamma_2) \overline M_n(\bs^{[n]} ) + \gamma_1 (1-\gamma_2) \overline M_n(\bs^{[n]} + \tfrac1n (1,0) ) \\
  &\qquad \qquad +  \gamma_2 (1-\gamma_1)\overline M_n(\bs^{[n]}+ \tfrac1n (0,1) )  +  \gamma_2 \gamma_1 \overline M_n(\bs^{[n]} + \tfrac1n (1,1) ) \;,
  \end{split}
\end{equation}
where $(\gamma_1, \gamma_2) := n (\bs-\bs^{[n]}) \in [0,1)^2 $, and where we use the convention $\overline M_n(\bs):=\overline M_n(\bs\wedge\bt)$ if $\bs\in\R_+^2\setminus [\bzero,\bt]$. Note that $\tilde M_n$ is a continuous random field, which satisfies $\tilde M_n(\bs^{[n]})=\overline M_n(\bs^{[n]})$ for all $\bs^{[n]}\in[\bzero,\bt]\cap\frac1n\bbZ$.

We prove the following Lemma.
\begin{lemma}\label{lem:controlM}
Under the assumptions of Theorem~\ref{thm:cvgcM}, one has for any $p<+\infty$,
\begin{equation}
\|\tilde M_n-\overline M_n\|_{\infty} \;\overset{L^p}{\underset{n\to\infty}{\longrightarrow}}\; 0\;.
\end{equation}
\end{lemma}
Recall that $\|\tilde M_n-\ol M_n\|_{\infty}:=\sup_{\bzero\preceq\bs\preceq\bt} |\tilde M_n(\bs)-\ol M_n(\bs)|$, which is a real-valued random variable. Recall also that for $p\in\N$, $\gz_\bone\in L^p(\bbP)$ as soon as $n$ is sufficiently large.

\begin{proof}
We prove the result for $p\in 2\N$ (we take $p>4r$) and $n$ sufficiently large, so that  in particular $\gz_\bone\in L^p(\bbP)$. First, notice that for $n\in\N$ and $\bs\in[\bzero,\bt]$, one has
\begin{equation}
\big|\tilde M_n(\bs)-\overline M_n(\bs)\big| \;\leq\; \max\left\{ \begin{aligned} &\big|\overline M_n(\bs^{[n]})-\overline M_n(\bs^{[n]}+(\tfrac1n,0))\big|\,,\\
&\big|\overline M_n(\bs^{[n]})-\overline M_n(\bs^{[n]}+(0,\tfrac1n))\big|\,,\\
&\big|\overline M_n(\bs^{[n]})-\overline M_n(\bs^{[n]}+(\tfrac1n,\tfrac1n))\big| \end{aligned} \right\} \,.
\end{equation}
Let us rewrite the last term, for $\bi\in\N_0^2$,
\[
\overline M_n(\tfrac1n\bi)-\overline M_n(\tfrac1n\bi+(\tfrac1n,\tfrac1n)) \;=\; \frac1{n^{3/2}\gb_n^r}\bigg(\sum_{j=1}^{i_1+1} \gz_{j,i_2+1} + \sum_{j=1}^{i_2} \gz_{i_1+1,j}\bigg)\,.
\]
Doing similarly with the two other terms, and using the inequality $(a+b)^p \leq 2^p(a^p+b^p)$ (recall $p\in 2\N$),
 we obtain
\begin{equation}\label{lem:controlM:eq1}
\begin{split}
(n^{3/2}\gb_n^r \pt \|\tilde M_n-\ol M_n\|_\infty)^p  & \;\leq\; C_\cntc\sup_{\bzero\preceq\bi\preceq n\bt} \Big(\sum_{j=1}^{i_1}\gz_{(j,i_2+1)}\Big)^p + C_{\arabic{cst}}\sup_{\bzero\preceq\bi\preceq n\bt} \Big(\sum_{j=1}^{i_2}\gz_{(i_1+1,j)}\Big)^p \\
&   \;\leq\;  C_{\arabic{cst}}  \sum_{\bzero\preceq\bi\preceq n\bt} \bigg(  \Big(\sum_{j=1}^{i_1}\gz_{(j,i_2+1)}\Big)^p + 
 \Big(\sum_{j=1}^{i_2}\gz_{(i_1+1,j)}\Big)^p \bigg)
\end{split}
\end{equation}
for some $C_{\arabic{cst}}>0$. 
H\"older's inequality and Lemma~\ref{lem:multicorrel}-\eqref{eq:multicorrel} give $\bbE[\gz_{\bi_1}\cdots\gz_{\bi_p}]\leq\bbE[\gz_\bone^p]\leq c_p (\gb_n)^p$, uniformly in $n\in\N$ and $\bi_1,\ldots,\bi_p\in(\N^2)^p$:
this implies
$\bbE\big[(\sum_{j=1}^{i_1}\gz_{(j,i_2+1)})^p\big]\leq  c_p (i_1)^p \gb_n^p$.
Therefore, \eqref{lem:controlM:eq1} gives
\begin{equation}
\bbE\big[( \|\tilde M_n-\ol M_n\|_\infty)^p\big]\;\leq\;  \frac{C_\cntc t_1 t_2 \|\bt\|_{p} }{(n^{3/2} \gb_n^r)^p}  \pt n^{p+2} \pt \gb_n^p \leq  \frac{C_{\arabic{cst}} t_1t_2 \|\bt\|_p }{(n^{1/2} \gb_n^r)^{p-4}}   \;,
\end{equation}
where we have used that $\gb_n^{p} \leq \gb_n^{4r}$ for the last inequality (recall we took $p>4r$).
This goes to $0$ since $p>4$,
recalling that $n^{1/2} \gb_n^r\to+\infty$.
\end{proof}

We now have all the required estimates to finish the proof of Theorem~\ref{thm:cvgcM}.

\begin{proof}[Proof of Theorem~\ref{thm:cvgcM}] First, let us prove that $(\tilde M_n)_{n\in\N}$ converges to $\cM$ in distribution. Lemma~\ref{lem:controlM} ensures us that for $m\in\N$ and $\bs_1,\ldots,\bs_m\preceq \bt$, the vector $( \tilde M_n(\bs_1)-\overline M_{n}(\bs_1), \ldots, \tilde M_n(\bs_m)-\overline M_n(\bs_m))$ converges to $(0,\ldots,0)\in\R^m$ in probability. Recalling Proposition~\ref{prop:lim:M} and applying Slutsky's theorem, this implies that $(\tilde M_n(\bs_k))_{k=1}^m$ converges to $(\cM(\bs_k))_{k=1}^m$ in distribution, yielding the finite-dimensional convergence. As for the tightness of $(\tilde M_n)_{n\geq1}$, it can be proven with a direct adaptation of Donsker's theorem (see \cite[Theorem~4.1.1 and Proposition~4.3.1, Chapter 6]{Kho02} for the multidimensional variant). In order not to overburden the presentation of this paper, we do not write the details here.

Finally, let $h: L^\infty([\bzero,\bt])\to\R$ be a bounded Lipschitz function, and let us prove that $\lim_{n\to+\infty} \bbE[h(\ol M_n)]=\bbE[h(\cM)]$. We write
\begin{equation}
\big|\bbE[h(\ol M_n)] - \bbE[h(\cM)]\big| \,\leq\, \big|\bbE[h(\tilde M_n)] - \bbE[h(\cM)]\big| + C\pt \bbE\big[\|\tilde M_n-\ol M_n\|_\infty\big]\,,
\end{equation}
for some $C>0$. The convergence in distribution of $(\tilde M_n)_{n\in\N}$ implies that the first term goes to 0 as $n\to\infty$, and Lemma~\ref{lem:controlM} shows the same for the second term. We conclude with Portmanteau's theorem \cite[Theorem~2.1]{Bill68}, which proves the convergence in distribution of $(\ol M_n)_{n\in\N}$ to $\cM$.
\end{proof}

\subsection{Convergence of the field in $L^p(\bbP)$}
Before moving on to the definition of the integral against $\cM$, let us state a last Claim 
which strengthens Theorem~\ref{thm:cvgcM} on a convenient probability space.

\begin{claim}\label{claim:cvgps}
There exist a probability space $(\widehat \gO, \widehat\cF, \widehat \bbP)$ and copies $\widehat M_n$, $n\geq1$ (resp.~$\widehat\cM$) of $\ol M_n$, $n\geq1$ (resp.\ of~$\cM$) on that space, such that for $\bs\in[\bzero,\bt]$ and $p\in[1,\infty)$, $\widehat M_n(\bs)\to\widehat\cM(\bs)$ in $L^p(\widehat\bbP)$ as $n\to\infty$.
\end{claim}

\begin{proof}
This is a consequence of Skorokhod's representation theorem. Since the set of continuous functions on $[\bzero,\bt]$ with the $\|\cdot\|_{\infty}$-topology is separable, there exist a probability space $(\widehat \gO, \widehat\cF, \widehat \bbP)$ and copies $\widehat M_n'$, $n\geq1$ (resp. $\widehat\cM$) of $\tilde M_n$, $n\geq1$ (resp. $\cM$) on that space, such that $\widehat M_n'\to \widehat\cM$ $\widehat \bbP$-almost surely. Using the moment estimates from Lemma~\ref{lem:moment:Mn} and dominated convergence theorem, it follows that for $\bs\in[\bzero,\bt]$, $\widehat M_n'(\bs)\to \widehat\cM(\bs)$ in $L^p(\widehat \bbP)$. Finally, let $\widehat M_n$ be a piecewise constant modification of $\widehat M_n'$, i.e. $\widehat M_n(\bs):=\widehat M_n'(\bs^{[n]})$ for $\bs\in[\bzero,\bt]$, $n\in\N$; then $\widehat M_n$ has the same law as $\ol M_n$, $n\in\N$, and we conclude the proof with Lemma~\ref{lem:controlM}.
\end{proof}

\section{Covariance measure of \texorpdfstring{$\cM$}{M} and stochastic integral}\label{sec:integration}
In this section we prove the well-posedness of $k$-iterated integrals against the field $\cM$, of any order $k\geq1$. In particular we prove that the series $\bZ:= \sum_{k=1}^{\infty} \int \psi_{\bt} \,\dd \cM$ defines a well-posed $L^2(\bbP)$-random variable.

\subsection{Integration with a covariance measure}
We first present a general theorem on the integration of a (deterministic) function against a $L^2(\bbP)$-random field $X$ on $\mathbb R^d$ which has a well-defined covariance measure. This theory is already well known in the literature (see e.g.~\cite[Ch.~2]{W86}), but to our knowledge its applications were so far mostly limited to orthogonal fields, that is when $\bbE[X(A)X(B)]=0$ for $A\cap B=\emptyset$. In our setting, it is applied to the non-orthogonal, non-martingale random field $\cM$, and it allows for the construction of the limiting random variable in Theorem~\ref{conj:scalinggPS}. More details on the general theory are presented in Appendix~\ref{app:intstoch}.

Let us introduce some definitions. With analogous notation to what is done in $\bbR^2$, for $\bu,\bv\in \bbR^d$ we denote $\bu\preceq \bv$ if all coordinates of $\bu$ are smaller or equal than those of $\bv$; we also denote $u^{(a)}$ the $a$-th coordinate of $\bu$. We let
\begin{equation}\label{eq:stocint:cS}
\cS_d\;:=\; \big\{[\bu,\bv)\,;\, \bu,\bv\in \R^d\,,\, \bu\preceq\bv\big\}\cup\{\emptyset\}
\end{equation}
be the semi-ring of sub-rectangles of $\bbR^d$, closed at the bottom-left and open at the top-right.

For any application $X: \bbR^d\to L^2(\bbP)$, 
we can define a random field on $\cS_d$ (which we also denote $X$),
by setting, for any  rectangle $A= [\bu_0,\bu_1) \in\cS_d$, $\bu_0\preceq\bu_1$, 
\begin{equation}\label{eq:stocint:defvariation}
X\big([\bu_0,\bu_1)\big) \;\coloneqq\;  \sum_{\gep \in \{0,1\}^d} (-1)^{d-\sum_{i=1}^d \gep_i} X\big( \bu_{\gep} \big)   \,,
\end{equation}
where for $\gep = (\gep_1, \ldots, \gep_d)\in \{0,1\}^d$ we have set $\bu_{\gep} = (u_{\gep_1}^{(1)}, \ldots, u_{\gep_{d}}^{(d)})$,
and $X(\emptyset):=0$. 
For $A\in \cS_d$, $X(A)$ is called the \emph{increment} of $X$ on $A$ (notice that it is coherent with the dimension $d=1$). 
Then, the ``integral'' of the function $\ind_A$ with respect to $X$ can be defined as $X(A)$;
and our goal is to extend this definition to more general measurable functions.

\begin{definition}
Let $X:\cS_d\to L^2(\bbP)$ be a random field.
For $A,B\in\cS_d$, define
\begin{equation}\label{eq:thmstocint:covar}
\nu(A\times B) = \bbE\big[X(A)X(B)\big] \,.
\end{equation}
If $\nu$ can be extended to a $\sigma$-finite
measure on $\Bor(\bbR^d\times \bbR^d)$,
 we call $\nu$ the \emph{covariance measure} of $X$ and we write $\nu_X:=\nu$.
\end{definition}

Assume $X$ is a random field on $\cS_d$ which admits some covariance measure $\nu=\nu_X$. 
Define
\begin{equation}\label{eq:defL2X}
L^2_\nu \;\coloneqq\; \left\{ g: \R^d\to\R \text{ measurable}\;;\; \int_{\R^d \times \R^d} |g(\bu)g(\bv)|\pt\dd \nu(\bu,\bv)<+\infty\right\},
\end{equation}
and for $g,h \in L^2_\nu$, 
\begin{equation}\label{eq:deflaranu}
\la g,h\ra_\nu \;\coloneqq\; \int_{\bbR^d\times \bbR^d}g(\bu)h(\bv)\pt\dd \nu(\bu,\bv) \qquad \text{and}\qquad \|g\|_\nu \;\coloneqq\;\sqrt{\la g,g\ra_\nu} \;.
\end{equation}
Then, $L^2_\nu$ is a vector space and $\la\cdot,\cdot\ra_{\nu}$ is (almost) a scalar product (see Appendix~\ref{app:intstoch} for more details). 
%
%
Let us now state the main theorem of this subsection, which defines an integral with respect to $X$ using its covariance measure~$\nu$ (the proof is discussed in Appendix~\ref{app:intstoch}).

\begin{theorem}\label{thm:stocint}
Let $X:\cS_d\to L^2(\bbP)$ be a random field which admits a $\sigma$-finite, non-negative covariance measure $\nu$ on $\Bor(\bbR^d\times \bbR^d)$. For $A\in\cS_d$, we define
\[\ind_A \diamond X\;:=\; X(A)\;\in\, L^2(\bbP)\;.\]
Then the application $g\mapsto g\diamond X$ can be extended into an isometry from $L^2_\nu$ to $L^2(\bbP)$: more precisely, for $g\in L^2_\nu$, there exists a random variable $g\diamond X$ defined almost everywhere on $(\gO,\cF,\bbP)$ such that
\begin{enumerate}
\item[(i)] $(\cdot)\diamond X$ is linear: for $g,h\in L^2_\nu$, $\gl\in\R$, $(g+\gl h)\diamond X = g\diamond X + \gl(h\diamond X)$ $\bbP$-a.s.;
\item[(ii)] for $g,h\in L^2_\nu$,
\begin{equation}\label{eq:thmstocint:covarbis}
\bbE\big[(g\diamond X)(h\diamond X)\big]\;=\; \la g,h\ra_\nu\;.
\end{equation}
\end{enumerate}
The random variable $g\,\diamond X$ is called the \emph{integral of $g$ against $X$} and will be denoted
$\int g\, \dd X := g\,\diamond X$.
\end{theorem}

\subsection{Application to the field $\cM$}\label{sec:covarM}
Recall that we identify an $L^2(\bbP)$-random function $X: \bbR^2\to L^2(\bbP)$
with the  field $(X(A))_{A\in\cS_2}$,
by setting $X(\emptyset):=0$ and for any rectangle $A = [\bu,\bv)$, $\bu\preceq \bv$,
\begin{equation}\label{eq:rectangleadditivity}
X(A) \;:=\; X(v_1,v_2) - X(u_1,v_2) - X(v_1,u_2) + X(u_1,u_2)\,,
\end{equation}
which is obtained by rewriting~\eqref{eq:stocint:defvariation} in dimension $d=2$. 
Let $\cM$ be a Gaussian field on $(\bbR_+)^2$ with covariance $K$ defined in \eqref{def:KcovcM}.
The goal of this section is twofold: first to determine a covariance measure for the random field $\cM$, \emph{i.e.} to define a measure~$\nu_\cM$ on $\Bor(\bbR_+^2 \times \bbR_+^2)$ such that for $A,B\in \cS_2$,
$\nu_\cM(A,B)\;:=\; \nu_\cM(A\times B)\;=\; \bbE\big[\cM(A)\cM(B)\big]$; and second to prove that the limiting renewal mass function $\phi$ is integrable against $\cM$.

\subsubsection*{Computation of the covariance measure of $\cM$} We have the following.
\begin{proposition}
\label{prop:covariancemeasure}
Let $\cM$ be a Gaussian field on $(\R_+)^2$ with zero-mean and covariance function $K(\bu,\bv)$ given in \eqref{def:KcovcM}.
There is a unique $\sigma$-finite measure $\nu_\cM$ on $\Bor(\bbR_+^2 \times \bbR_+^2)$ such that for any $A,B\in\cS_2\,$, $\nu_\cM(A\times B)=\bbE[\cM(A)\cM(B)]$. Moreover, for any non-negative measurable functions $g,h:\bbR_+^2\to\R$, one has
\begin{equation}\label{eq:stocint:exprint2}
\int_{\bbR_+^2\times \bbR_+^2}g(\bu)h(\bv) \pt\dd\nu_\cM(\bu,\bv) \;=\; \int_{\bbR_+^2}  g(\bu)  \Big( \int_{\bbR_+} h(x,u_2) \dd x + \int_{\bbR_+} h(u_1,y) \dd y \Big) \dd \bu \,.
\end{equation}
\end{proposition}

\begin{remark} 
\label{rem:whitenoise}
$(i)$ For the sake of comparison, let $W$ be a Gaussian white noise on $\R^d$. 
One has $\nu_W(A\times B):=\bbE[W(A)W(B)]\;=\; \gl_d(A\cap B)$ for $A,B \in \cS_{d}$, where $\gl_d$ denotes the $d$-dimensional Lebesgue measure. This can be extended to all $E\in\Bor(\bbR^d \times \bbR^d)$
by $\nu_W(E) \;=\; \gl_d\big(\pi_d(E)\big)$, where $\pi_d(E):=\{\bu\in \bbR^d; (\bu,\bu)\in E\}$. 

$(ii)$ More generally, a Gaussian field $(X_\bs)_{\bs\in \mathbb R^d}$ with covariance function $K(\bs,\bt)$ admits a covariance measure~$\nu_X$ which is formally given by $\dd \nu_X(\bs, \bt) = \frac{\partial^2}{\partial \bs \, \partial \bt} K(\bs,\bt)$ (\emph{i.e.} in the sense of distributional derivatives).
\end{remark}
%

\begin{proof}[Proof of Proposition~\ref{prop:covariancemeasure}]
Let us first compute $\bbE[\cM(A)\cM(B)]$ for any $A,B\in\cS_2$. 
Notice that, by~\eqref{eq:rectangleadditivity}, $\cM$ is an additive field on $\cS_2$: for any real numbers $x\leq y\leq z$ and $u\leq v$, one has
\begin{equation}\label{eq:stocint:rectdecomp}
\cM([u,v)\times[x,z)) \;=\; \cM([u,v)\times[x,y)) + \cM([u,v)\times[y,z))\,.
\end{equation}
Moreover for any rectangles $A,B\in\cS_2$, we can decompose them into finite unions of rectangles $A=\cup_{i=1}^p A_i$ and $B=\cup_{j=1}^q B_j$ such that for $1\leq i\leq p$, $1\leq j\leq q$ (we can take $p,q\leq 9$), one of the following holds:
\begin{enumerate}[label=(\alph*)]
\item $A_i=B_j$.
\item There exist $u_0\leq u_1$ and $s_0\leq s_1\leq t_0\leq t_1$ such that either $A_i=[u_0,u_1)\times [s_0,s_1)$ and $B_j=[u_0,u_1)\times [t_0,t_1)$ or $A_i=[s_0,s_1)\times[u_0,u_1)$ and $B_j=[t_0,t_1)\times [u_0,u_1)$, or the other way around.
\item For $a\in\{1,2\}$, the projections of $A_i$, $B_j$ on the $a$-th coordinate are disjoint.
\end{enumerate}
This implies that we only have to compute the covariances of increments $\cM(A)$, $\cM(B)$ for couples of rectangles $(A,B)$ satisfying one of the above: this will give us covariances of all rectangles thanks to \eqref{eq:stocint:rectdecomp} and the bilinearity of $(X,Y)\mapsto\bbE[XY]$. We do so in the following Lemma.

\begin{lemma}\label{lem:stocint:covcM}
Let $\cM$ be a Gaussian field on $\bbR_+^2$ with covariance $K$ defined in \eqref{def:KcovcM}. 
Let $u_0\leq u_1$ and $s_0\leq s_1\leq t_0\leq t_1$. 
\begin{enumerate}[label=(\alph*)]
\item If $A=B=[u_0,u_1)\times[s_0,s_1)$, then
\[\bbE\big[\cM(A)^2\big]\;=\;(u_1-u_0)(s_1-s_0)(u_1-u_0+s_1-s_0)\,.\]
\item If $A=[u_0,u_1)\times[s_0,s_1)$ and $B=[u_0,u_1)\times[t_0,t_1)$, then
\[\bbE\big[\cM(A)\cM(B)\big]\;=\;(u_1-u_0)(s_1-s_0)(t_1-t_0)\,.\]
\item If the projections of $A,B$ on the $a$-th coordinate are disjoint for $a\in\{1,2\}$, then $\bbE[\cM(A)\cM(B)]=0$.
\end{enumerate}
\end{lemma}

We postpone the proof of this lemma for now. Note that, for any $A,B\in\cS_2$, this yields that,
\begin{equation}\label{eq:stocint:exprint1}
\nu_\cM(A,B)\coloneqq \bbE\big[\cM(A)\cM(B)\big] = \int_{[\bzero,\bt)}  \ind_{A}(\bs)  \Big( \int_{0}^{t_1} \ind_{B}(x,s_2) \dd x + \int_0^{t_2} \ind_{B}(s_1,y) \dd y \Big) \dd \bs \,.
\end{equation}
where Lemma~\ref{lem:stocint:covcM} states this identity for $A,B$ satisfying (a), (b) or (c), and generic couples of rectangles are handled by bilinearity of the r.h.s. and the additivity of $\cM$, recall~\eqref{eq:stocint:rectdecomp}. Then, Proposition~\ref{prop:covariancemeasure} is a direct consequence of Proposition~\ref{prop:dominationmeasure} which provides a criterion to extend a function on $(\cS^2)^2$ into a measure on $\Bor(\R^2\times\R^2)$, and~\eqref{eq:stocint:exprint1} which allows us to identify its expression. Eventually, $\cM$ admits a well-defined covariance measure $\nu_\cM$ on $\Bor(\R^2_+\times\R^2_+)$ which verifies~\eqref{eq:stocint:exprint2}.
\end{proof}

Let us mention that~\eqref{eq:stocint:exprint1} can also be written as
\begin{equation}\label{eq:stocint:nuleb}
\begin{aligned}
\nu_\cM(A,B)\;&=\; \gl_3\big((x,y,z)\in\R^3;\ (x,y)\in A \text{ and }(x,z)\in B\big) \\
&\qquad \qquad + \gl_3\big((x,y,z)\in\R^3;\ (x,y)\in A \text{ and }(z,y)\in B\big)\;,
\end{aligned}
\end{equation}
where $\gl_3$ denotes the 3-dimensional Lebesgue measure.

\begin{proof}[Proof of Lemma~\ref{lem:stocint:covcM}]
We only detail the proof in the second case $A=[u_0,u_1)\times[s_0,s_1)$ and $B=[u_0,u_1)\times[t_0,t_1)$, since the other two are very similar. Let us first rewrite \eqref{eq:stocint:defvariation} into
\[\cM(A) \;=\; \sum_{i,j\in\{0,1\}} (-1)^{i+j} \cM(u_i,s_j)\;,\]
since $d=2$ is even here. Thus, 
\[\begin{aligned}\bbE\big[\cM(A)\cM(B)\big] \;&=\; \sum_{i,j,k,l\in\{0,1\}} (-1)^{i+j+k+l}\,\bbE\big[\cM(u_i,s_j)\cM(u_k,t_l)\big]\\
&=\; \sum_{i,j,k,l\in\{0,1\}} (-1)^{i+j+k+l} (u_i\wedge u_k)(s_j\wedge t_l)(u_i\vee u_k + s_j\vee t_l)\\
&=\; \sum_{i,j,k,l\in\{0,1\}} (-1)^{i+j+k+l} (u_i\wedge u_k)(s_j)(u_i\vee u_k + t_l)\,, \end{aligned}\]
where we used $s_0\leq s_1\leq t_0\leq t_1$. Let us develop the last factor to rewrite $\bbE[\cM(A)\cM(B)]$ as a sum of two terms: in the first one, we can factorize $\sum_{l=0}^1(-1)^l=0$, 
so it remains
\[\bbE\big[\cM(A)\cM(B)\big] \;=\; \bigg(\sum_{j=0}^1 (-1)^j \pt s_j\bigg)\bigg(\sum_{l=0}^1 (-1)^l \pt t_l\bigg)
\bigg(\sum_{i,k\in\{0,1\}} (-1)^{i+k} (u_i\wedge u_k)\bigg)\,,\]
and a straightforward computation gives the result.
\end{proof}

\subsubsection*{Integrability of $\phi$ and $\psi$ against $\cM$}

Recall that Theorem~\ref{thm:stocint} defines the \emph{integrals} $g\diamond \cM$ of functions $g\in L^2_{\nu_\cM}$,
where the measure $\nu_\cM$ is given explicitly in \eqref{eq:stocint:exprint2}. Let us now prove that the term $k=1$ in the expansion~\eqref{eq:conjchaosexpansion} is well-defined (at least for $\hat h=0$). To lighten notation, from now on we will write $\|\cdot\|:=\|\cdot\|_1$ for the $L^1$ norm on $\R^2$.
\begin{proposition}
\label{prop:integregp}
Fix $\bt\succ 0$ and let $g: \R_+^2 \to \R_+$ be the function defined by $g(\bs):= \|\bs\|^{\ga-2} \ind_{(\bzero,\bt)}(\bs)$ for $\bs\in \bbR_+^2\setminus\{\bzero\}$ and $\alpha\in (0,1)$. 
Then 
 $g 
 \in L^2_{\nu_\cM}$ if and only if $\alpha\in(\frac{1}{2},1)$.
As a consequence, $\int_{[0,\bt)} \gp(\bs) \gp(\bt-\bs) \dd \cM(\bs)$ is well-defined
if and only if $\alpha\in(\frac{1}{2},1)$.
\end{proposition}
 
\begin{proof}
Recalling~\eqref{eq:deflaranu} and~\eqref{eq:stocint:exprint2}, we have that
\begin{equation}\label{nunorm}
\begin{split}
\|g\|^2_{\nu} &\;=\; \int_{(\bzero,\bt)} g (\bs) \Big( \int_{0}^{t_1} g(x,s_2) \dd x + \int_{0}^{t_2} g(s_1,y) \dd y \Big) \dd \bs \\
&\; =\;  \frac{1}{1-\alpha} \int_{(\bzero,\bt)} (s_2^{\alpha-1} - (t_1+s_2)^{\alpha-1} +s_1^{\alpha-1} - (t_2+s_1)^{\alpha-1}) (s_1+s_2)^{\ga-2} \dd s_1 \dd s_2 \,.
\end{split}
\end{equation}
For $\ga\in(\frac12,1)$, we bound this from above by $\frac{1}{1-\alpha} \int_{(\bzero,\bt)} (s_2^{\alpha-1}+s_1^{\alpha-1}) (s_1+s_2)^{\ga-2} \dd s_1 \dd s_2 $. Then we have that
\begin{equation*}
\begin{split}
\int_{(\bzero,\bt)} s_1^{\alpha-1} (s_1+s_2)^{\ga-2} \dd s_1 \dd s_2  &\;=\; \int_0^{t_1} \frac{s_1^{\ga-1}}{\ga-1}(s_1^{\ga-1}-(s_1+t_2)^{\ga-1}) \dd s_1 \\
&\;\leq \; \int_0^{t_1} \frac{s_1^{2\ga-2}}{\ga-1}\dd s_1 \;=\; \frac{t_1^{2\ga-1}}{(1-\ga)(2\ga-1)},
\end{split}
\end{equation*}
so we obtain $\|g\|^2_{\nu}\leq \frac{t_1^{2\ga-1}+t_2^{2\ga-1}}{(1-\ga)^2(2\ga-1)}<+\infty$.

On the other hand, for $\ga\in(0,\frac12]$, we write 
\begin{equation*}
\|g\|^2_{\nu} \;\geq\; \frac{1}{1-\alpha} \int_{(\bzero,\bt)} (s_2^{\alpha-1}-t_1^{\ga-1}+s_1^{\alpha-1}-t_2^{\ga-1}) (s_1+s_2)^{\ga-2} \dd s_1 \dd s_2  \; .
\end{equation*}
then we use that $(t_1^{\alpha-1}+t_2^{\alpha-1})\int_{(\bzero,\bt)} (s_1+s_2)^{\ga-2} \dd s_1 \dd s_2 <+\infty$, and
\begin{equation*}
\int_{(\bzero,\bt)} s_1^{\ga-1}(s_1+s_2)^{\ga-2} \dd s_1 \dd s_2 \;\geq\; \int_0^{t_1}\bigg(\frac{s_1^{2\ga-2} - s_1^{\ga-1}t_2^{\ga-1}}{1-\ga}\bigg)\dd s_1 = +\infty\;.
\end{equation*}
This proves that $\|g\|^2_{\nu}=+\infty$ for $\ga\leq\frac12$.
%

If we define the function $\psi_{\bt} (\bs) = \gp(\bs) \gp(\bt-\bs) \ind_{\{\bzero \prec \bs \prec \bt\}}$,
we can bound $\psi_{\bt} (\bs) \leq C g(\bs) \gp(\bt) \ind_{\{\|\bs\|\leq\frac12\|\bt\|\}} + C \gp(\bt) g(\bt-\bs) \ind_{\{\|\bs\|>\frac12\|\bt\|\}}$.
Hence, $\|\psi_{\bt}\|_{\nu_{\cM}} <+\infty$ if $\ga\in (\frac12,1)$. 
If on the other hand we have $\ga\in (0,\frac12]$, using that $\psi_{\bt}(\bs) \geq C \gp(\bs) \gp(\bt) \ind_{(0,\frac12\bt)}(\bs)$,
we get that $\|\psi_{\bt}\|_{\nu_{\cM}} =+\infty$, since $\gp(\bs) \geq c g(\bs)$ uniformly for $\bs = r e^{i\theta}$ with $\theta \in (\frac16\pi, \frac13 \pi)$ (which is enough to conclude, with the same computation as above).
\end{proof}

\subsection{Integrals of higher rank against $\cM$}
\label{sec:higherrank}

The goal of this section is to define integrals of higher rank in the expansion of $\mathbf{Z}$ in~\eqref{eq:conjchaosexpansion} (at least when $\hat h=0$) and to prove that the series defines a well-posed random variable in $L^2(\bbP)$.

Recall that $\cS_d$ denotes the semi-ring of bounded sub-rectangles of $\R^d$ 
 and that $(\cS_d)^k\simeq\cS_{kd}$ is also a semi-ring: for $X:\cS_d\to L^2(\bbP)$ a random field, we define
the product field $X^{\otimes k}$ on $\cS_{kd} $ by 
$X^{\otimes k}(A) \;:=\; \prod_{i=1}^k X(A_i)$ for 
$A=A_1\times\ldots\times A_k \in \cS_{kd}$.
If $X$ is a random function, \textit{i.e.}\ $X:\bbR^d\to L^2(\bbP)$, then we may define $X^{\otimes k}(\bs_1,\ldots,\bs_k):=\prod_{i=1}^k X(\bs_i)$ a random function on $(\bbR^d)^k\simeq\R^{dk}$, and the above definition of the field $X^{\otimes k}(A)$ matches exactly the $dk$-dimensional increment of the function $X^{\otimes k}$ on $A\in \cS_{kd}$, see~\eqref{eq:stocint:defvariation}. With those notation, if $X^{\otimes k}$ admits some covariance measure $\nu_{X^{\otimes k}}$ on $\Bor(\R^{dk}\times\R^{dk})$, then Theorem~\ref{thm:stocint} may be applied as it is to define the stochastic integral against $X^{\otimes k}$. For $g:\R^{dk}\to\R$, $g\in L^2_{X^{\otimes k}}$, we will write
\[
g \overset{k}{\diamond} X := g \diamond (X^{\otimes k}) = \int g\, \dd (X^{\otimes k})\,.
\]

Henceforth, this section is analogous to the previous one: first, we prove that the field $\cM^{\otimes k}$ admits a well-defined, explicit covariance measure $\nu_{\cM^{\otimes k}}$ on $\Bor(\R_+^{2k}\times\R_+^{2k})$; then, we prove that the function $\psi_\bt$ in~\eqref{def:psit} is integrable with respect to $\nu_{\cM^{\otimes k}}$. Therefore the integral of $\psi_\bt$ against $\cM^{\otimes k}$ is well-posed, and we additionally prove that the series of integrals in~\eqref{eq:conjchaosexpansion}, \emph{i.e.} $\mathbf{Z}$, is well-defined in $L^2(\bbP)$.

\subsubsection*{Covariance measure of $\cM^{\otimes k}$}
We have the following result.
\begin{proposition}
\label{propbis:covariancemeasure:k>1}
Let $\cM$ be a Gaussian field on $\R_+^2$ with zero-mean and covariance matrix $K(\bu,\bv)$ given in \eqref{def:KcovcM} and let $k\geq1$. Then $\cM^{\otimes k}$ admits a unique non-negative  $\sigma$-finite covariance measure $\nu_{\cM^{\otimes k}}$ on $\Bor( (\bbR_+^2)^k)$ such that for any $A, B \in \cS_2^k$
we have $\nu_{\cM^{\otimes k}}(A,B)\;=\;\nu_{X^{\otimes k}}(A\times B)= \bbE\big[\cM^{\otimes k}(A)\pt \cM^{\otimes k}(B)\big]$.
The measure $\nu_{\cM^{\otimes k}}$ is characterized by the following: for any measurable non-negative function $g:(\bbR_+^2)^{k}\to\R$,
we have
\begin{equation}
\label{eq:stocint:exprint:k>1}
\begin{aligned}
&\int_{\bbR_+^{2k}}g(\bu_1,\ldots,\bu_{2k}) \pt\dd\nu_{\cM^{\otimes k}}(\bu_1,\ldots,\bu_{2k}) \\
&\qquad \;=\; \sum_{\mathcal{J} \in \mathscr{P}_{2k}} \int_{\bbR_+^{2k}}  \bigg(
\int_{\cA_{\bu_{i_1}} \times \cdots \times \cA_{\bu_{i_k}} }  
g(\bu_1,\ldots,\bu_{2k}) 
\dd \gl_{\bu_{i_1}}(\bu_{j_1})\ldots \dd \gl_{\bu_{i_k}}(\bu_{j_k})
\bigg) \dd\bu_{i_1}\ldots\dd\bu_{i_k},
\end{aligned}
\end{equation}
where  the sum is over all partitions of $\{1,\ldots,2k\}$ into pairs $\mathcal{J}=\{\{i_1,j_1\},\ldots,\{i_k,j_k\}\}$,
and for $\bu\in \bbR_+^2$,
\begin{itemize}
\item  $\cA_\bu$ denotes the set of points in $\bbR_+^2$ aligned with $\bu$, \textit{i.e.}\ 
$\cA_\bu:=(\bbR_+\times \{u_2\})\cup \big (\{u_1\}\times\bbR_+)$;
\item $\gl_\bu$ denotes the (one-dimensional) Lebesgue measure on $\cA_\bu$, \textit{i.e.}\ for $f:\bbR_+^2\to\R_+$, $\int_{\cA_\bu} f(\bv)\dd \gl_\bu(\bv) = \int_0^{\infty}f(x,u_2)\dd x + \int_0^{\infty}f(u_1,y)\dd y$.
\end{itemize}
\end{proposition}
This result is a direct analogue of Proposition~\ref{prop:covariancemeasure} for generic $k\geq1$. The formula \eqref{eq:stocint:exprint:k>1} can be obtained as an application of Wick's formula. Alternatively, one can obtain~\eqref{eq:stocint:exprint:k>1} for simple functions $g=\ind_A$, $A\in\cS_2^{k}$
thanks to Proposition~\ref{prop:lim:M}, \textit{i.e.}\ with the convergence of moments of $\ol M_n$ to those of $\cM$, and then extend it to all functions (indeed, recall that in the asymptotics of the moments of $\ol M_n$, we proved that the contributing configurations contain only pairs of aligned points, see~\eqref{eq:elleven-0}). In order not to overburden the presentation of this paper, we leave the details to the reader.

\begin{remark}
More generally, in the case of a Gaussian field $X$ with covariance measure $\nu_X$ (recall Remark~\ref{rem:whitenoise}), the covariance measure of $X^{\otimes k}$ can be obtained via Wick's formula:
\[
\dd \nu_{X^{\otimes k}} (\bu_1, \ldots, \bu_{2k}) = \sum_{\mathcal{J} \in  \mathscr{P}_{2k} } \prod_{\{a,b\}\in \mathcal J} \dd \nu_X(\bu_a,\bu_b) \,.
\]
\end{remark}

\subsubsection*{Integrability of $\psi_\bt$ against $\cM^{\otimes k}$}
For any $k\in \bbN$, we define $\psi_{\bt,k}:=\psi_\bt$ as in \eqref{def:psit}: 
\begin{equation*}
\psi_{\bt,k}(\bs_1,\ldots, \bs_k) := \gp(\bs_1) \gp(\bs_2 -\bs_1) \ldots \gp(\bs_k -\bs_{k_1}) \gp(\bt-\bs_k) \ind_{\{\bzero \prec \bs_1 \prec \cdots \prec \bs_k \prec \bt\}} \,.
\end{equation*}
To lighten notation, let us write $\nu_{k}:=\nu_{\cM^{\otimes k}}$ henceforth. Similarly to Section~\ref{sec:covarM}, we prove here that the integral of $\psi_{\bt,k}$ against $\nu_{k}$ is well defined when $\ga \in (\tfrac12, 1)$ and we give a bound on its dependence on $k$.

\begin{proposition}\label{prop:ub:psikk}
If $\ga \in (\tfrac12 ,1)$, then   $\psi_{\bt,k}$, $\psi_{\bt,k}^\free$ and $\psi_{\bt,k}^{\cond}$ are in $L^2_{\nu_{k}}$ for all $k\geq1$.
More precisely, there is a constant $C_\alpha>0$ such that for $k\in \bbN$, we have  
\[
\|\psi_{\bt,k}\|_{\nu_{k}}^2 \leq \frac{ (C_{\alpha})^{k+1}  C_{\bt,k,\alpha} }{\Gamma( k (\alpha-\frac12))} \qquad \text{with } C_{\bt,k,\alpha} := \frac{ (\bt^{(1)}\wedge \bt^{(2)})^{2(k+1)(\alpha-2)}}{ (\bt^{(1)} \bt^{(2)} (\bt^{(1)} \vee \bt^{(2)}) )^{k}}\,.
\] 
In particular $\psi_{\bt,k} \overset{k}{\diamond} \cM := \psi_{\bt,k}\diamond \cM^{\otimes k}$ is a well-defined $L^2(\bbP)$-random variable.
\end{proposition}

Notice that this proposition and the completeness of $L^2(\bbP)$ immediately imply that the series $\mathbf{Z}$ from~\eqref{eq:conjchaosexpansion} is well-posed (at least for $\hat h=0$,  we will see in Section~\ref{sec:reduc_k=0} that we can always reduce to this case).

\begin{corollary}\label{corol:ub:psikk}
For $\hat\gb\geq0$, one has $\sum_{k\geq1} \hat\gb^k \|\psi_{\bt,k}\|_{\nu_{k}}<\infty$. In particular $\sum_{k\geq1} \hat\gb^k (\psi_{\bt,k} \overset{k}{\diamond} \cM)$ is a well-posed random variable in $L^2(\bbP)$.
\end{corollary}

\begin{remark}\label{rem:freecond2} In the remainder of this paper we focus on the \emph{constrained} partition function, \textit{i.e.}\ on the integration of $\psi_{\bt,k}$ defined in~\eqref{def:psit}. We claim that the same results for $\psi_{\bt,k}^{\cond}$ and $\psi_{\bt,k}^{\free}$ follow naturally. Indeed, we have $\psi_{\bt,k}^{\cond}=\gp(\bt)^{-1} \psi_{\bt,k}$ so $\|\psi_{\bt,k}^{\cond}\|_{\nu_{k}}^2 \leq \gp(\bt)^{-2} \|\psi_{\bt,k}\|_{\nu_{k}}^2$. Using that $\gp(2\bt-\bs) \geq C \gp(\bt)$ uniformly for $s\in [\bzero,\bt)$ we also get that
\[
\psi_{\bt,k}^{\free}(\bs_1,\ldots, \bs_k) \leq   \frac1C \gp(\bt)^{-1}  \, \gp(\bs_1) \gp(\bs_2 -\bs_1) \ldots \gp(\bs_k -\bs_{k_1}) \gp(2 \bt-\bs_k) \ind_{\{\bzero \prec \bs_1 \prec \cdots \prec \bs_k \prec \bt\}}\,,
\]
hence $\psi_{\bt,k}^{\free}  \leq   \frac1C \psi_{2\bt,k}^{\cond}$.
\end{remark}

\begin{proof} 
%
%
Let us define $g_k$ similarly to $\psi_{\bt,k}$, but with $g(\bs) = \| \bs\|^{\ga-2}$ in place of $\gp(\bs)$, where $\|\cdot\|:=\|\cdot\|_1$. 
We show the proposition for $g_k$, which will imply the result for $\psi_{\bt,k}$ (recall Proposition~\ref{thm:renouv}).
Let us warn the reader that the proof is more technical than in the case $k=1$, due to the richer combinatorics in the correlation structure of $\nu_k$, see~\eqref{eq:stocint:exprint:k>1}.

We have
\begin{equation}\label{kintegral}
\|g_k\|_{\nu_k}^2=\, \idotsint \limits_{ \substack{ \bzero \prec \bu_1 \prec \cdots \prec \bu_k \prec \bt \\ \bzero \prec \bv_1 \prec \cdots \prec \bv_k \prec \bt}} \!
 g_k(\bu_1, \ldots, \bu_k)  g_k (\bv_1, \ldots, \bv_k) \dd\nu_k( \bu_1, \ldots, \bu_k, \bv_1, \ldots, \bv_k).
\end{equation}
Note that by a change of variable, we can reduce to the case where $\bt =\bone$, at the cost of a factor at most
\[
C_{\bt,k,\alpha} = \frac{ (\bt^{(1)}\wedge \bt^{(2)})^{2(k+1)(\alpha-2)}}{ (\bt^{(1)} \bt^{(2)} (\bt^{(1)} \vee \bt^{(2)}) )^{k}}\,,
\]
using that $\bt^{(1)} |x| + \bt^{(2)} |y| \geq (\bt^{(1)}\wedge \bt^{(2)}) (|x|+|y|)$.
Note that we can bound $C_{\bt,k,\alpha}\leq (\bt^{(1)}\wedge \bt^{(2)})^{(k+1)(2\alpha-7) +3}$.

Now, in view of the expression of $\nu_k$ (recall~\eqref{eq:stocint:exprint:k>1}), for any fixed $\bzero \prec \bu_1 \prec \cdots \prec \bu_k \prec \bt $, the integral over $\bv_1,\ldots, \bv_k$ is concentrated on the ``grid'' set
\begin{align*}
\cG(\bu_1,\ldots, \bu_k)  = \bigcup_{i=1}^k \cA_{\bu_i} \, ,
\end{align*}
where we recall that for $\bu \in [\bzero,\bone]$, $\cA_{\bu}$ is the set of points aligned with $\bu$, that we write as $\cA_{\bu}=\cL_\bu^{(1)} \cup \cL_{\bu}^{(2)}$ with $\cL_\bu^{(1)} = [0,1] \times \{u_2\}$ and $\cL_\bu^{(2)} = \{u_1\} \times [0,1]$.
Moreover, the integral \eqref{kintegral} is concentrated on $(\bv_1,\ldots, \bv_k) \in \cG(\bu_1,\ldots, \bu_k)$ where there must be some  $\bv_i$ in  $\cA_{\bu_j} $ for every $1\leq j \leq k$ (so that all points $\bu_j$ are aligned with one $\bv_i$): since $\bv_1 \prec \cdots \prec \bv_k$, there is a permutation $\sigma$ of $\{1,\ldots, k\}$ such that $\bv_{i} \in \cA_{\bu_{\sigma(i)}}$ for all~$i$ (using also that the Lebesgue measure of points in $\cA_{\bu_j} \cap \cA_{\bu_{i}}$ is equal to $0$). We refer to Figure~\ref{figGrid} for an illustration.

\begin{figure}[htbp]
\begin{center}
\includegraphics[scale=0.7]{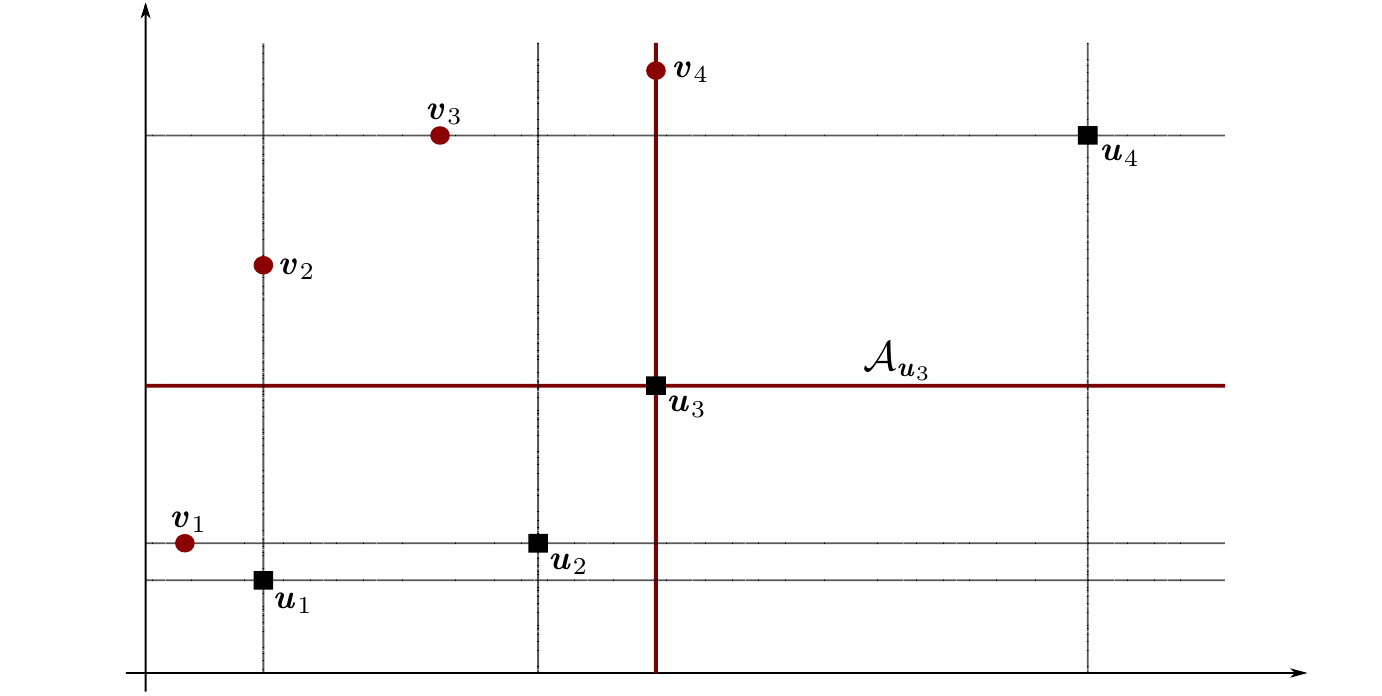}
\caption{\label{figGrid}
\footnotesize
Illustration of the grid set $\cG(\bu_1,\ldots, \bu_k)$, with $k=4$ points. The black squares represent the positions of $\bu_1, \ldots, \bu_4$, and the set $\cA_{\bu_3}$ is represented in a red solid line. The red dots represent the positions of $\bv_1, \ldots, \bv_4$: here we have that $\bv_1\in \cA_{\bu_2}$ ($\sigma(1)=2$), $\bv_2\in \cA_{\bu_1}$ ($\sigma(2)=1$), $\bv_3\in \cA_{\bu_4}$ ($\sigma(3)=4$)  and $\bv_4\in \cA_{\bu_3}$ ($\sigma(4)=3$). 
}
\end{center}
\end{figure}

Let us denote $\bar S_k$ the set of all permutations $\sigma$ of $\{1,\ldots, k\}$ that are admissible pairings of the $\bv_i$'s with $\bu_j$'s, in the sense that they are compatible with the condition $\bu_1\prec \cdots \prec \bu_k$ and $\bv_1\prec \cdots \prec \bv_k$.
All together, we have
\[
\|g_k\|_{\nu_k}^2= \!\!
\idotsint \limits_{ \substack{ \bzero \prec \bu_1 \prec \cdots \prec \bu_k \prec \bone \\ \bzero \prec \bv_1 \prec \cdots \prec \bv_k \prec \bone}} \!\!
 \ind_{\{\exists \sigma\in \bar S_k\,, \bv_{\sigma(i)} \in \cA_{\bu_{i}} \forall i \}} g_k(\bu_1, \ldots, \bu_k)  g_k (\bv_1, \ldots, \bv_k) \dd\nu_k( \bu_1, \ldots, \bu_k, \bv_1, \ldots, \bv_k).
\]
Let us stress that $\bar{S}_k$ does not contain all permutations: indeed, we cannot have $\bu_1\prec \bu_2 \prec \bu_3$ and $\bv_1\prec \bv_2\prec \bv_3$ with the following ``alignment pattern'':  $\bv_3 \leftrightarrow \bu_1$, $\bv_2\leftrightarrow \bu_2$, $\bv_1\leftrightarrow \bu_3$. Hence, admissible permutations  must avoid the pattern $(3\,2\,1)$.
Since there are $C_n = \frac{1}{n+1} \binom{2n}{n} \leq 4^n$ such permutations (see \textit{e.g.} \cite{Janson20} for a recent account on permutations avoiding patterns of length $3$), we get that 
\begin{equation}\label{eq:barSk}
|\bar S_k|\;\leq\; 4^k\;.
\end{equation}
Therefore, recalling that $g_k(\bu) = \prod_{i=1}^{k+1} g(\bu_i-\bu_{i-1})  $
with by convention $\bu_0=\bzero$, $\bu_{k+1}=\bone$, the proof then consists in showing that, for any $\sigma \in \bar S_k$,
\begin{multline}
\label{integralongrid2}
\idotsint \limits_{ \substack{ \bzero \prec \bu_1 \prec \cdots \prec \bu_k \prec \bone \\ \bzero \prec \bv_1 \prec \cdots \prec \bv_k \prec \bone}} 
 \prod_{i=1}^{k+1}\Big(  g(\bu_{\sigma(i)} -\bu_{\sigma(i)-1}) g(\bv_{i} -\bv_{i-1}) \ind_{\{\bv_{i} \in \cA_{\bu_{\sigma(i)}} \}} \Big)  \dd\nu_k ( \bu_1, \ldots, \bu_k, \bv_1, \ldots, \bv_k)\\[-15pt]
  \leq \frac{C^k}{\Gamma\big( k(\alpha-\frac12)\big)} \,.
\end{multline}
This is a consequence of the following proposition.

\begin{proposition}\label{prop:induction}
There exists some (explicit) $C_\ga>0$ such that for $\bzero\preceq\bu_0\prec\bu_1\preceq\bone$, $\bzero\preceq\bv_0\prec\bv_1\preceq\bone$ and $k\in\N\cup\{0\}$,
\begin{equation}\label{eq:prop:induction}\begin{aligned}
&\int_{[\bu_0,\bu_1]\times[\bv_0,\bv_1]} \|\bu-\bu_0\|^{k(\ga-\frac12)}g(\bu-\bu_0)g(\bv-\bv_0)g(\bu_1-\bu)g(\bv_1-\bv) \dd \nu_\cM(\bu,\bv) \\&\hspace{3cm}\leq\; C_\ga \,\tilde\gGa(k) \, \|\bu_1-\bu_0\|^{(k+1)(\ga-\frac12)} g(\bu_1-\bu_0)g(\bv_1-\bv_0)\,,\end{aligned}
\end{equation}
where $\tilde\gGa(k):=\frac{\gGa(k(\ga-\frac12))}{\gGa((k+1)(\ga-\frac12))}$ if $k\in\N$, and $\tilde\gGa(0):=1$.
\end{proposition}

Applying Proposition~\ref{prop:induction} iteratively, we obtain~\eqref{integralongrid2}, which concludes the proof.
\end{proof}

\begin{proof}[Proof of Proposition~\ref{prop:induction}]
%
%
The way to estimate the integral in~\eqref{eq:prop:induction} depends on the respective locations of $\bu_0,\bu_1,\bv_0,\bv_1$. We only treat the case $\bu_0^{(1)}<\bv^{(1)}_0<\bu^{(1)}_1<\bv^{(1)}_1$ and $\bv^{(2)}_0<\bu^{(2)}_0<\bv^{(2)}_1<\bu^{(2)}_1$, see Figure~\ref{fig:induction}; other cases are analogous (or easier) and can be treated with similar techniques. 
We will actually estimate the integral in~\eqref{eq:prop:induction} restricted to $\bu,\bv$ being on the same column, \textit{i.e.}\ to $\bu^{(1)}=\bv^{(1)}$; the case where $\bu,\bv$ are on the same line is similar.

We introduce the following notation, which will be used throughout the proof (we refer to Figure~\ref{fig:induction} for a graphical representation):
\begin{equation}
\begin{aligned}
&x := \bu^{(1)} - \bv_0^{(1)} = \bv^{(1)} -\bv_0^{(1)} \,, & \quad y := \bu^{(2)} - \bu_0^{(2)} \,, & \qquad z  := \bv^{(2)} - \bv^{(2)}_0  \,, \\
&\ol x := \bu^{(1)}_1 - \bv^{(1)}_0\,, & \quad \ol y := \bu^{(2)}_1 - \bu^{(2)}_0\,, &\qquad \ol z  := \bv^{(2)}_1 - \bv^{(2)}_0   \,,
\end{aligned}
\end{equation}
and also $a:= \bv^{(1)}_0 - \bu^{(1)}_0$ and $ b:= \bv^{(1)}_1 - \bu^{(1)}_1$.

\begin{figure}[ht]\begin{center}\vspace{-7mm}
\includegraphics[scale=0.9]{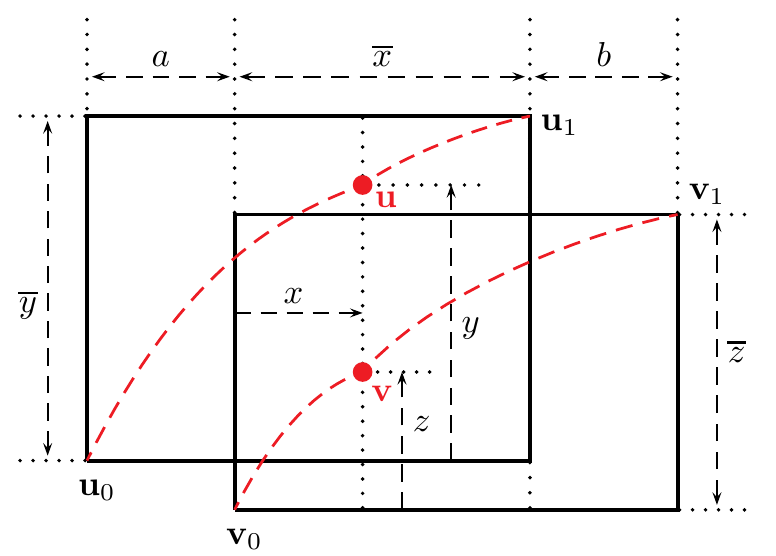}\vspace{-5mm}
\caption{\footnotesize Illustration of the relative position of the points in the integral over $\bu,\bv$ on the same column (in red). $\bu$ (resp. $\bv$) is constrained to remain in a rectangle of size $\ol x \times \ol y$ (resp. $\ol x \times \ol z$), and their distances to $\bu_0,\bu_1,\bv_0$ and $\bv_1$ may be expressed in terms of $x,y,z$ and the constants $a,b,\ol x,\ol y$ and $\ol z$.}
\label{fig:induction}
\end{center}
\end{figure}

With these notation, the integral in~\eqref{eq:prop:induction} restricted to $\bu,\bv$  on the same column can be written as
\begin{align*}
\int_0^{\ol x}\int_0^{\ol y}\int_0^{\ol z} (x+y+a)^{k(\ga-\frac12)}\big[(x+y+a)(x+z)(\ol x + \ol y - x - y)(\ol x  + b+ \ol z - x - z)\big]^{\ga-2} \dd x \,\dd y \,\dd z\,. 
\end{align*}
Using a change of variable $(x,s,t)=(x,x+y+a,x+z)$, we get
\begin{align*}
 &\int_0^{\ol x} \int_{x+a}^{x + \ol y + a} \int_{x}^{x+\ol z}  s^{k(\ga-\frac12)} \big[s\,t\,(\ol s - s)(\ol t - t)\big]^{\ga-2}\,\dd x \,\dd s \,\dd t\\
&\qquad =\; \Big(\frac{\ol s}2\Big)^{k(\ga-\frac12)}\Big(\frac{\ol s \, \ol t}{4}\Big)^{2\ga-3} \int_0^{\ol x} \int_{2(x+a)/\ol s}^{2(x + \ol y + a)/\ol s} \int_{2x/\ol t}^{2(x+\ol z)/\ol t} s^{k(\ga-\frac12)} \big[s\,t\,(2 - s)(2 - t)\big]^{\ga-2}\,\dd x \,\dd s \,\dd t
\,,
\end{align*}
where we have set 
$\ol s := \ol x + \ol y + a = \|\bu_1-\bu_0\|$, $\ol t := \ol x + \ol z + b \;=\; \|\bv_1-\bv_0\|$,
and then rescaled the last two integrals in $\ol s/2$, $\ol t/2$ respectively. Exchanging the integrals to first integrate with respect to $x$, this is equal to
\begin{equation}\label{eq:prop:induction:1}
\Big(\frac{\ol s}2\Big)^{k(\ga-\frac12)}\Big(\frac{\ol s \, \ol t}{4}\Big)^{2\ga-3} \int_0^2\int_0^2 s^{k(\ga-\frac12)} \big[s\,t\,(2 - s)(2 - t)\big]^{\ga-2} |A_{s,t}| \,\dd s\,\dd t\,,
\end{equation}
where $A_{s,t}\subset[0,1]$ is defined by
\begin{equation}\label{eq:prop:induction:setA}
A_{s,t}\;:=\;\left\{x\in[0,\ol x]\;;\; x+a \in[\tfrac{s}{2} \ol s - \ol y, \tfrac{s}{2} \ol s]\quad \text{and} \quad x\in[\tfrac{t}{2} \ol t - \ol z, \tfrac{t}{2} \ol t]\right\} \,,
\end{equation}
and $|A_{s,t}|$ denotes its Lebesgue measure.
Therefore, the proof will be over once we show
\begin{equation}\label{eq:prop:induction:2}
\int_0^2\int_0^2 s^{k(\ga-\frac12)} \big[s\,t\,(2 - s)(2 - t)\big]^{\ga-2} |A_{s,t}| \,\dd s\,\dd t \;\leq\; C_\ga\, 2^{k(\ga-\frac12)} \,\tilde\gGa(k) \, \Big(\frac{\ol s \,\ol t}{4}\Big)^{\frac12}\,,
\end{equation}
for some $C_\ga>0$. Indeed, plugging~\eqref{eq:prop:induction:2} into~\eqref{eq:prop:induction:1},
we get that the integral in~\eqref{eq:prop:induction} is bounded by
\[
4^{\frac{5}{2} -2\alpha} C_\ga\, \tilde\gGa(k) \, 
\bar s^{(k+1) (\alpha-\frac12)+(\alpha-2)} \, \bar t^{(\alpha-2)+(\alpha-\frac12)}\,,
\] 
which concludes the proof of~\eqref{eq:prop:induction} since $\ol t^{(\ga-\frac12)} \leq 2^{\ga-\frac12}$ (recall $\bv_1,\bv_0\in [\bzero,\bone]$).

\medskip
\noindent
\textit{Proof of~\eqref{eq:prop:induction:2}.}
Let us first state an inequality which will prove useful henceforth:
\begin{equation}\label{eq:prop:induction:ineqwedge}
x \wedge y \;\leq\; \sqrt{(x\wedge y)(x\vee y)} \;=\; \sqrt{x\,y}\;,\qquad \forall\, x,y>0\,.
\end{equation}
We split the l.h.s. of~\eqref{eq:prop:induction:2} into four integrals over the sets $I_1=\{s, t\leq1\}$, $I_2=\{1\leq s, t\}$, $I_3=\{t\leq 1\leq s\}$ and $I_4=\{s\leq 1\leq t\}$ respectively, and we compute an upper bound for each term.

\smallskip
\textit{Integral over $I_1$.} Since $A_{s,t}\subset[0,\frac{s}2 \ol s]\cap[0,\frac{t}2 \ol t]$, we have
$|A_{s,t}| \leq (\tfrac{\ol s}2s)\wedge(\tfrac{\ol t}2t) \leq (\frac{\ol s\ol t}{4})^{1/2} (st)^{1/2}$
thanks to~\eqref{eq:prop:induction:ineqwedge}.
Therefore, the integral over $I_1$ verifies
\begin{align*}
\int_0^1\int_0^1 s^{k(\ga-\frac12)}  \big[s\,t\,(2 - s)(2 - t)\big]^{\ga-2} |A_{s,t}| \,\dd s\,\dd t  &\leq\;\Big(\frac{\ol s\ol t}{4}\Big)^{\frac12} \int_0^1 s^{k(\ga-\frac12) + \ga-\frac32} \,\dd s \int_0^1 t^{\ga-\frac32} \,\dd t  \\
&=  \frac1{\ga-\frac12} \frac{1}{(k+1)(\ga-\frac12)} \Big(\frac{\ol s\ol t}{4}\Big)^{\frac12}
\end{align*}
where the last identity holds because $\ga>1/2$ and $k\geq0$. Then, we observe that there is $k_\ga\in\N$ such that $\gGa$ is increasing on $[k_\ga(\ga-\frac12),\infty)$, hence for $k\geq k_\ga$, 
\begin{equation}\label{eq:prop:induction:ineqgGa}
\frac1{(k+1)(\ga-\frac12)}\leq \frac1{k(\ga-\frac12)} = \frac{\gGa(k(\ga-\frac12))}{\gGa(k(\ga-\frac12) + 1)} \;\leq\; \tilde\gGa(k) \;\leq\; 2^{k(\ga-\frac12)}\, \tilde\gGa(k) \;.
\end{equation}
Fixing a suitable $C_\ga>0$, this proves the upper bound~\eqref{eq:prop:induction:2} for the integral restricted on $I_1=\{s,t\leq1\}$.

\smallskip
\textit{Integral over $I_2$.} Recalling~\eqref{eq:prop:induction:setA}, we have $A_{s,t}\subset [\frac{s}2 \ol s-\ol y -a, \ol x] \cap [\frac{t}2 \ol t-\ol z, \ol x]$, hence
\[
|A_{s,t}|\;\leq\; \big(\ol x+\ol y+a - \tfrac{\ol s}{2}s\big) \wedge \big(\ol x+\ol z - \tfrac{\ol t}{2}t\big) \;\leq\; \big(\ol s(1-\tfrac s2)\big) \wedge \big(\ol t(1-\tfrac t2)\big) \;\leq\; \Big(\frac{\ol s \ol t}{4}\Big)^{\frac12} \big((2-s)(2-t)\big)^{\frac12}\,,
\]
where we used~\eqref{eq:prop:induction:ineqwedge}. Therefore,
\begin{align}
\notag
\int_1^2\int_1^2 s^{k(\ga-\frac12)} \big[s\,t\,(2 - s)(2 - t)\big]^{\ga-2} |A_{s,t}| \,\dd s\,\dd t 
 & \leq  \Big(\frac{\ol s\ol t}{4}\Big)^{\frac12} \int_1^2 s^{k(\ga-\frac12)-1}(2-s)^{\ga-\frac32} \,\dd s \int_1^2 (2-t)^{\ga-\frac32} \,\dd t \\
 & = \frac1{\ga-\frac12}  \Big(\frac{\ol s\ol t}{4}\Big)^{\frac12} \int_1^2 s^{k(\ga-\frac12)-1}(2-s)^{\ga-\frac32}  \,\dd s
 \label{eq:prop:induction:3}
 \end{align}
where we also used that $s^{\ga-1}\leq 1$ for $s\in[1,2]$ in the first inequality. 
For $k=0$ the integral above is finite, which proves~\eqref{eq:prop:induction:2};  for $k\geq 1$, we recall that $\int_0^1 x^{a-1}(1-x)^{b-1} \dd x = \frac{\gGa(a)\gGa(b)}{\gGa(a+b)}$, $a,b>0$, which yields (since $\ga>1/2$)
\begin{equation}\label{eq:prop:induction:4}
\int_1^2 s^{k(\ga-\frac12)-1}(2-s)^{\ga-\frac32}  \,\dd s \;\leq\; \int_0^2 s^{k(\ga-\frac12)-1}(2-s)^{\ga-\frac32}  \,\dd s\;\leq\; C_\ga \,2^{k(\ga-\frac12)} \, \frac{\gGa(k(\ga-\tfrac12))\gGa(\ga-\frac12)}{\gGa((k+1)(\ga-\tfrac12))}\,.
\end{equation}
We conclude by recognizing the definition of $\tilde\gGa(k)$.

\smallskip
\textit{Integral over $I_3$.}
Similarly to the previous cases, we have $A_{s,t}\subset [\frac{s}2 \ol s-\ol y -a, \ol x] \cap [0,\frac{t}2 \ol t]$, so~\eqref{eq:prop:induction:ineqwedge} implies
\[
|A_{s,t}|\leq \big( \ol s (1-\tfrac{s}{2})\big) \wedge\big(\tfrac{t}{2} \ol t\big) \;\leq\; \Big(\frac{\ol s \ol t}{4}\Big)^{\frac12} \big((2-s)t\big)^{\frac12}\,. 
\]
Thus,
\begin{align*}
\int_1^2\int_0^1 s^{k(\ga-\frac12)} \big[s\,t\,(2 - s)(2 - t)\big]^{\ga-2} |A_{s,t}| \,\dd s\,\dd t  \leq \Big(\frac{\ol s\ol t}{4}\Big)^{\frac12} \int_1^2 s^{k(\ga-\frac12)-1}(2-s)^{\ga-\frac32} \,\dd s \int_0^1 t^{\ga-\frac32} \,\dd t\,,
\end{align*}
and a straightforward change of variable $t\to (2-t)$ yields the same upper bound as in~\eqref{eq:prop:induction:3}-\eqref{eq:prop:induction:4}.

\smallskip
\textit{Integral over $I_4$.}
With the same argument as before we have $|A_{s,t}| \leq (\tfrac{\ol s \ol t}4)^{\frac12} (s(2-t))^{\frac12}$, so we get
\begin{align*}
\int_0^1\int_1^2 s^{k(\ga-\frac12)} \big[s\,t\,(2 - s)(2 - t)\big]^{\ga-2} \gl(A_{s,t}) \,\dd s\,\dd t 
 & \leq \Big( \frac{\ol s\ol t}{4}\Big)^{\frac12} \int_0^1 s^{k(\ga-\frac12) + \ga-\frac32} \,\dd s \int_1^2 (2-t)^{\ga-\frac32} \,\dd t\\
& = \frac1{\ga-\frac12} \frac{1}{(k+1)(\ga-\frac12)} \Big(\frac{\ol s\ol t}{4}\Big)^{\frac12}\;.
\end{align*}
Then we conclude the proof as for the term $I_1$, with~\eqref{eq:prop:induction:ineqgGa}.
\end{proof}

\section{Convergence of the polynomial chaos expansion}\label{sec:proofthm}

In this section we consider the polynomial expansion of the partition function.
Similarly to Proposition~\ref{prop:ub:psikk} we focus on the constrained partition function (the cases of the conditioned or free partition function are analogous,  recall Remark~\ref{rem:freecond2}):
similarly to~\eqref{expansion}, we have 
\begin{equation}
\label{expansion_constrained}
Z_{\lfloor n\bt \rfloor,h_n}^{\gb_n,\go, \quen} 
= e^{\gb_n \go_{\lfloor n\bt \rfloor} -\lambda(\gb_n) +h_n}\sum_{k=0}^{ \lfloor nt_1 \rfloor\wedge \lfloor n t_2 \rfloor} \sum_{ \bzero=\bi_0 \prec \bi_1 \prec \ldots \prec   \bi_{k+1}
= \lfloor n\bt \rfloor } \prod_{l=1}^{k}  (e^{h_n}\zeta_{\bi_l}  +e^{h_n}-1)\, \prod_{l=1}^{k+1} u(\bi_l-\bi_{l-1})  \,.
\end{equation}

Let us highlight the different steps of the proof.
First, we show that one can reduce to treating the case $h_n=0$, simply by expanding the product $\prod_{l=1}^{k} (e^{h_n}\zeta_{\bi_l}  +e^{h_n}-1)$ and factorizing in the homogeneous partition function.
Then, we show the $L^2$ convergence of the $k$-th term of the expansion to a (multivariate) integral against $\cM$, 
for each $k\geq 1$ separately: this is the purpose of Proposition~\ref{prop:convk>1}.
To conclude the proof, we control the $L^2$ norm of each term of the expansion: we show that the $L^2$ norm of the $k$-th term is bounded by some constant $c_k$ uniformly in $n$, with $c_k$ summable---this is the content of Proposition~\ref{lemma:tronck>1}.

\smallskip
Before starting the proof, let us recall the assumptions of Theorem~\ref{conj:scalinggPS} we work with in this section: 
we have $\ga\in(\frac12,1)$, $\bbP\in\Pfk_r$ for some $r\in\N$, and the scaling relations~\eqref{def:scalings} for $h_n,\gb_n$.
Recall also the definition~\eqref{def:psit} of $\psi:=\psi_{\bt}$, where we drop the index $\bt$ to lighten notation---we work with a fixed $\bt$.
Also, to lighten notation, from now on we write $n\bt, nt_1,nt_2$ omitting the integer part.

\subsection{Reducing to the case $h_n=0$}\label{sec:reduc_k=0}
Let us first explain how to reduce to the case $h_n=0$.

At the continuous level, note that expanding the product $\prod_{j=1}^k ( \sigma_r \hat\gb^r \,\dd \cM(\bs_j) + \hat h \, \dd \bs_j)$ in~\eqref{eq:conjchaosexpansion} and summing over points between indices where $\dd \cM(\bs_j)$ appears, we get that
the continuum partition function~\eqref{eq:conjchaosexpansion} can be rewritten as
\begin{align}
\bZ_{\bt, \hat h}^{\hat\gb, \cM, \quen} & =  \sum_{\ell=0}^{+\infty} (\sigma_r \hat \gb^r)^\ell \idotsint \limits_{ \bzero \prec \bs_1 \prec \cdots \prec \bs_\ell \prec \bt } \ 
 \prod_{j=1}^{\ell +1}  \bigg( \sum_{k_j=0}^{+\infty} \hat h^{k_j} \idotsint \limits_{ \bs_{j-1} \prec \bs'_1 \prec \cdots \prec \bs'_{k_j} \prec \bs_j } \prod_{i=1}^{k_j+1} \gp(\bs'_i -\bs'_{i-1}) \prod_{i=1}^{k_j} \dd \bs'_i  \bigg)  \prod_{j=1}^{\ell}   \dd \cM(\bs_j) \notag\\
 &  =  \sum_{\ell=0}^{+\infty} (\sigma_r \hat \gb^r)^\ell \idotsint \limits_{ \bzero \prec \bs_1 \prec \cdots \prec \bs_\ell \prec \bt }  \psi^{(\hat h)}_{\bt} \big( \bs_1,\ldots, \bs_\ell  \big)  \prod_{j=1}^{\ell}  \dd \cM(\bs_j) \,.
 \label{eq:factorisationZ}
\end{align}
where, analogously to~\eqref{def:psit}, we have defined
\begin{equation}
\label{def:psith}
 \psi^{(\hat h)}_{\bt} \big( \bs_1,\ldots, \bs_k  \big) := \ind_{\{\bzero =:\bs_0 \prec \bs_1 \prec \cdots \prec \bs_k \prec \bs_{k+1}:=\bt \}}  \prod_{j=1}^{k+1}  \bZ_{\bs_j-\bs_{j-1}, \hat h} \,,
\end{equation}
recalling the definition of $\bZ_{\bs,\hat h}$ in Proposition~\ref{prop:scalinghom}.

At the discrete level, analogously to what is done in~\eqref{eq:factorisationZ},
expanding the product $\prod_{l=1}^{k+1}(e^{h_n}\zeta_{\bi_l}  +e^{h_n}-1)$ and rearranging the terms,
we can rewrite~\eqref{expansion_constrained} as
\begin{multline*}
e^{ -(\gb_n \go_{n\bt} -\lambda(\gb_n))} Z_{\lfloor n\bt \rfloor,h_n}^{\gb_n,\go, \quen} \\
 = \sum_{\ell=0}^{\infty} \sum_{ \bzero=\bi_0 \prec \bi_1 \prec \ldots \prec \bi_\ell \prec  \bi_{\ell+1} = n\bt } \prod_{j=1}^{\ell}  \zeta_{\bi_j}  
 \prod_{j=1}^{\ell+1} e^{h_n} \bigg( \sum_{k_j=1}^{\infty}  (e^{h_n}-1)^{k_j}\sum_{\bi_{j-1}=\bi'_0 \prec \bi'_1 \prec \ldots \prec \bi'_{k_j} \prec  \bi'_{k_j+1} = \bi_{j} } \prod_{a=1}^{k_j} u(\bi_{a}-\bi_{a-1}) \bigg)  \\
 = \sum_{\ell=0}^{\infty} \sum_{ \bzero=\bj_0 \prec \bj_1 \prec \ldots \prec \bj_\ell \prec  \bj_{\ell+1} = n\bt } \prod_{j=1}^{\ell}  \zeta_{\bi_j}  \prod_{j=1}^{\ell+1} Z_{\bi_{j}-\bi_{j-1},h_n} \,,
\end{multline*}
where we recognized the expansion of the homogeneous partition function $Z_{\bi, h_n}$ to get the last identity, see~\eqref{eq:expandhomogeneous}.

Then, we can set $u^{(h_n)} (\bi) := Z_{\bi,h_n}$, and observe that thanks to Proposition~\ref{prop:scalinghom} (which is proven in Section~\ref{sec:homogeneous} by a standard Riemann-sum approximation) we have, for any $\bs\succ \bzero$,
\[
\lim_{n\to\infty} n^{2-\alpha} L(n) u^{(h_n)} ( \lfloor n \bs \rfloor ) = \bZ_{\bs,\hat h} =: \gp^{(\hat h)} (s) \,.
\]
Note also that Lemma~\ref{lem:homogene} which provides the uniform bound $ u^{(h_n)} (\bi)\leq C_{\hat h} L(\|\bi\|_1)^{-1} \|\bi\|_1^{\alpha-2}$.

These are the two key properties that allow us to adapt the proof of Theorem~\ref{conj:scalinggPS}, performed in the case $h_n\equiv 0$, to a general sequence $(h_n)_{n\geq 1}$ (satisfying~\eqref{def:scalings}).
Indeed, we simply need to replace the renewal mass function $u(\cdot)$ with $u^{(h_n)} (\cdot)$ and use Lemma~\ref{lem:homogene} instead of $u (\bi)\leq C L(\|\bi\|_1)^{-1} \|\bi\|_1^{\alpha-2}$, which comes from~\cite[Thm.~4.1]{B18}.
In the limit, the $k$-points correlation function $\psi(\bs_1,\ldots, \bs_k) = \prod_{i=1}^{k+1} \gp(\bs_i-\bs_{i-1}) $ for $\bzero\prec \bs_0\prec \bs_1 \prec \cdots \prec \bs_{k+1}=\bt$ from~\eqref{def:psit} is simply replaced by $\psi^{(\hat h)} =  \prod_{i=1}^k \gp^{(\hat h)}(\bs_i-\bs_{i-1}) $ defined in~\eqref{def:psith}, as appears in~\eqref{eq:factorisationZ}.

\subsection{Rewriting of the $k$-th term as a discrete integral}\label{sec:proofthm_discrint}
We now focus on the following expansion, for $h_n=0$:
\begin{equation*}
  n^{2-\alpha} L(n) Z_{n\bt,h_n=0}^{\gb_n,\go, \quen}  =
e^{ \gb_n \go_{ n\bt} -\lambda(\gb_n)}\sum_{k=0}^{\infty} \sum_{ \bzero=\bi_0 \prec \bi_1 \prec \ldots \prec \bi_k \prec  \bi_{k+1} = n\bt} n^{2-\alpha} L(n) \prod_{l=1}^{k}  \zeta_{\bi_l}  \prod_{l=1}^{k+1}u(\bi_{l}-\bi_{l-1}) \,,
\end{equation*}
which leads us to define
\begin{equation}
\label{def:tildeZ}
\tilde Z_{n,k} := \sum_{ \bzero=\bi_0 \prec \bi_1 \prec \ldots \prec \bi_k \prec  \bi_{k+1} = n\bt} n^{2-\alpha} L(n) \prod_{l=1}^{k}  \zeta_{\bi_l} \prod_{l=1}^{k+1}u(\bi_{l}-\bi_{l-1})\,.
\end{equation}
Note that the pre-factor $e^{-(\gb_n \go_{n\bt}- \gl(\gb_n))}$ is irrelevant since it converges to 1 in $L^2$.
The main idea is to rewrite the $k$-th term $\tilde Z_{n,k}$ as some ``integral'' of a discrete approximation $\psi_m$ of $\psi$ against the product discrete field $\ol M_n^{\otimes k}$. Note that we use different indices $m,n$ for the approximation of the correlation function $\psi$ and for the approximation of the field $\cM$: in the proof, the idea is to first let $n\to\infty$ and then $m\to\infty$.

Let us introduce some notation. For $m\in\N$, let $\Delta_m:=[\bzero,\frac1m\bone)$ and $\cD_m:=\frac1m\bbZ^2\cap[\bzero,\bt]$. 
For $k,m,n\in\N$ and a function $g_{m}:[\bzero,\bt]^k\to\R$ constant on each $\prod_{l=1}^k(\bu_l + \Delta_m)$, $\bu_1,\ldots,\bu_k\in \cD_m$, we define the $k$-iterated ``discrete integral'' of $g_m$ against $\ol M_n$,
\begin{equation}\label{eq:intdiscrete:k>1}
g_{m} \overset{k}\cdot \ol M_n \;:=\; \sum_{\bu_1,\ldots,\bu_k \in \cD_m} g_{m}(\bu_1,\ldots,\bu_k)\pt \prod_{l=1}^k \ol M_n \big(\bu_l - \tfrac1n\bone + \Delta_m\big)\,,
\end{equation}
where we refer to~\eqref{def:Mbar} for the definition of $\ol M_n$ and to~\eqref{eq:stocint:defvariation} for the definition of increments of a field that appear in the last product.
Note that the term $(- \tfrac1n\bone)$ is added to ensure that, if $n=m$ and $\frac1n\bi \in \cD_n=\cD_m$, then we get $\ol M_n (\frac1n(\bi - \bone) + \Delta_n) = \frac1{\sigma_rn^{3/2}\gb_n^r} \gz_{\bi,n}$, where we write $\gz_{\bi,n} := \gz_{\bi} = e^{\gb_n \go_{\bi} -\lambda(\gb_n)}-1$ for $\bi\in\N^2$ to keep track of the dependence on $n$.

Now, define
\begin{equation}\label{defphin}
\phi_m(\bu) \;:=\; m^{2-\ga} L(m) \bP \big(  \lfloor m \bu \rfloor \in \btau \big)\;,\qquad \bu\in[\bzero,\bt]\;,
\end{equation}
which is piecewise constant on each $(\bu+\Delta_m)\cap[\bzero,\bt]$, $\bu\in \cD_m$. By Proposition~\ref{thm:renouv} (from \cite{Will68}), $\phi_m$ converges simply to $\phi$ as $m\to\infty$.  
With this at hand we define the piecewise constant approximation $\psi_m$ of $\psi$ by replacing $\phi$ with $\phi_m$ in its definition~\eqref{def:psit}:
\[
\psi_m (\bu_1,\ldots, \bu_k) = \gp_m(\bu_1) \, \gp_m( \bu_2 - \bu_1) \, \cdots\, \gp_m( \bt - \bu_k ) \ind_{\{\bzero \prec \bu_1 \prec \cdots \prec \bu_k \prec \bt \}}\,.
\]
With this notation, we can rewrite the $k$-th term $\tilde Z_{n,k}$, see~\eqref{def:tildeZ}, as
\begin{equation}\label{eq:intdiscrete:k>1:bis}
\tilde Z_{n,k} = \bigg( \frac{\sigma_r n^{3/2} \gb_n^r}{n^{2-\ga} L(n)} \bigg)^k \, \psi_n\overset{k}\cdot \ol M_n \;.
\end{equation}
One of our main goals is to prove the following.

\begin{proposition}\label{prop:convk>1}
Let $k\geq1$. Under the assumptions of Theorem~\ref{thm:cvgcM}, one has
\begin{equation*}
\psi_n\overset{k}\cdot \ol M_n \quad \xrightarrow[n\to\infty]{(d)}\quad \psi \overset{k}\diamond \cM\;.
\end{equation*}
This convergence holds in $L^2(\hat\bbP)$ on a convenient probability space $(\hat\gO, \hat\cF,\hat\bbP)$.
\end{proposition}
Recall from Proposition~\ref{prop:ub:psikk} that the stochastic integral $\psi \overset{k}\diamond \cM$ is well defined.

\begin{remark}
The ``integral'' $\psi_n\overset{k}\cdot \ol M_n$ that we introduced does not fall under the definition from  Theorem~\ref{thm:stocint}. Even though one could define the covariance measure of $\ol M_n^{\otimes k}$ with a collection of Dirac masses, one cannot use a direct argument (such as dominated convergence)  to get that
$\psi_n\overset{k}\diamond \ol M_n^{\otimes k}$ converges towards $\psi\overset{k}\diamond\cM$.  Nonetheless, the development of such arguments would be an interesting expansion of our results towards a general methodology to study the influence of (correlated) disorder on physical systems.
\end{remark}

\subsection{Proof of Proposition~\ref{prop:convk>1}}\label{sec:proofthm_k}

Recall from Claim~\ref{claim:cvgps} that, up to a change of probability space, we may assume  that $(\ol M_n)_{n\ge 1}$ converges $\bbP$-a.s.\ to $\ol\cM$ and that all pointwise convergences hold in $L^p(\bbP)$ for $p\geq1$ (we do not change notation for simplicity's sake). Hence we simply need to prove the $L^2(\mathbb P)$ convergence on that probability space.

In order to deal with the fact that $\phi$ blows up around $\bzero$, let us introduce a truncated version of $\gp_m,\psi_m,\gp,\psi$. For $\delta\geq 0$, let $\phi_m^\delta(\bu):=\phi_m(\bu) \ind_{\{\|\bu\| \geq \delta\}}$ (note that $\phi_m^0=\phi_m$)and
\[
\psi_m^{\delta}(\bu_1,\ldots, \bu_k) :=\gp_m^{\delta}(\bu_1) \, \gp_m^{\delta}( \bu_2 - \bu_1) \, \cdots\, \gp_m^{\delta}( \bt - \bu_k ) \ind_{\{\bzero \prec \bu_1 \prec \cdots \prec \bu_k \prec \bt \}} \,.
\]
We define similarly $\phi^\delta(\bu):=\phi(\bu)  \ind_{\{\|\bu\| \geq \delta\}}$ and $\psi^\delta$.

We write for $n\geq m$, $\delta > 0$,
\begin{equation}\label{eq:lem:controlunifk>1:decompo}
\| \psi_n\overset{k}\cdot \ol M_n - \psi \overset{k}\diamond \cM\|_{L^2} \;\leq\; I_1 + I_2 + I_3 + I_4 + I_5\;,
\end{equation}
where
\begin{equation}
\begin{array}{llll}
I_1 \;:=\; & \| \psi_n\overset{k}\cdot \ol M_n - \psi_n^\delta \overset{k}\cdot \ol M_n\|_{L^2}\,,\quad &
I_2 \;:=\; & \| \psi_n^\delta\overset{k}\cdot \ol M_n - \psi_m^\delta \overset{k}\cdot \ol M_n\|_{L^2}\,, \smallskip\\
I_3 \;:=\; & \| \psi_m^\delta \overset{k}\cdot \ol M_n - \psi_m^\delta \overset{k}\diamond \cM\|_{L^2}\,,\quad &
I_4 \;:=\; & \| \psi_m^\delta \overset{k}\diamond \cM - \psi^\delta \overset{k}\diamond \cM\|_{L^2}\,, \smallskip\\
I_5 \;:=\; & \| \psi^\delta \overset{k}\diamond \cM - \psi \overset{k}\diamond \cM \|_{L^2}\,. &&
\end{array}
\end{equation}
We now need to control each term separately.
We show that we can fix $\delta$ sufficiently small so that the terms~$I_1$ and $I_5$ are small, uniformly in $n$ for~$I_1$.
Then, for a fixed $\delta$, we show that $I_2, I_4$ can be made arbitrarily small by choosing $m$ large, uniformly in $n\geq m$ for~$I_2$. It then remains to see that, for any fixed $m$ and $\delta$, $I_3$ vanishes as~$n\to\infty$.


\subsubsection{Term $I_5$}
First of all, notice that we have
$\lim_{\delta\to 0 }\psi^\delta =\psi$, so by Proposition~\ref{prop:ub:psikk}, the isometry property~\eqref{eq:thmstocint:covarbis} and dominated (or monotone) convergence, we get that we can chose $\delta$ sufficiently small to make~$I_5$ arbitrarily small.

\subsubsection{Term $I_1$}
We use the following key result, which shows that $I_1$ is arbitrarily small for small $\delta$, uniformly in $n$ large. Its proof is postponed to Section~\ref{sec:prooftronck>1} below; it can be viewed as the core of the proof, and contains some of the most technical part of the paper.
As a first part, we also include a bound on the  $L^2$ norm of the discrete integral for the non-truncated $\psi_n$: this part is the discrete analogue of Proposition~\ref{prop:ub:psikk} and will prove useful later.

\begin{proposition}\label{lemma:tronck>1}
$(i)$ There exist constants
$C,c>0$ and some $n_0\geq 1$ such that for any $k\in\N$, $n\geq n_0$, one has
\begin{equation}\label{eq:tronck>1:bound1}
\|\psi_{n} \overset{k}\cdot \ol M_n \|_{L^2}^2 \;\leq\; \frac{C^k}{\gGa(k(\ga-\frac12))} + C^k \gb_n^{ck}\;.
\end{equation}
$(ii)$ For  every $k\in \mathbb N$, we have
\begin{equation}\label{eq:tronck>1:bound2}
\lim_{\delta \downarrow 0} \limsup_{n\to\infty} \|\psi_{n} \overset{k}\cdot \ol M_n - \psi_{n}^\delta \overset{k}\cdot \ol M_n\|_{L^2}^2 =0 \;.
\end{equation}
\end{proposition}


\subsubsection{Term $I_4$}
It is already clear that $\lim_{m\to\infty} I_4 =0$: this follows from the fact that $\psi_m^\delta(\bu)\to\psi^\delta(\bu)$ as $m\to\infty$ for all $\bu\in[\bzero,\bt]^k$, together with Proposition~\ref{prop:ub:psikk} and dominated convergence. 

\subsubsection{Term $I_2$}
Let us show that the convergence $\lim_{m\to\infty} \psi_m^\delta =\psi^\delta$ actually holds for the $\|\cdot\|_\infty$ norm on $[\bzero,\bt]^k$. 
First, the convergence $\lim_{m\to\infty}\phi_m(\bu) =\phi(\bu)$, $m\to\infty$ holds uniformly in $\bu\in[\bzero,\bt]\setminus[0,\gd)^2$ for $\gd>0$ (see \cite{Will68}). We therefore get that for $\delta>0$ fixed, $\lim_{m\to\infty} \|\phi_m^\delta-\phi^\delta\|_\infty = 0$. 
Now, let $\bu:=(\bu_1,\ldots,\bu_k)$ with $\bzero=:\bu_0\preceq\bu_1\prec\ldots\prec \bu_k\preceq\bu_{k+1}:=\bt$. Then, by the simple fact (proven e.g.\ by recurrence) that 
\begin{equation}
\label{prodsum}
\bigg| \prod_{i=1}^k a_i -\prod_{i=1}^k b_i\bigg| \leq \Big(\max_{1\leq i\leq k}\{|a_i|,|b_i|\}\Big)^{k-1} \sum_{i=1}^k |a_i-b_i| \,,
\end{equation}
we get that
\begin{align*}
\bigg|\prod_{l=1}^{k+1} \phi_m^\delta(\bu_l-\bu_{l-1}) - \prod_{l=1}^{k+1} \phi^\delta(\bu_l-\bu_{l-1})\bigg| 
\leq\; (k+1)  \Big(\sup_{[\bzero, \bt]\setminus [0,\delta)^2} (|\gp_m(\bu)| +|\gp(\bu)|)   \Big)^{k} \;  \big\| \phi_m^\delta-\phi^\delta\big\|_\infty\,.
\end{align*}
Since, $|\gp_m(\bu)|,|\gp(\bu)|$ are bounded by a constant $C_{\delta}$ uniformly for $\|\bu\| \geq \delta$,
this indeed shows that $\lim_{m\to\infty} \|\psi_m^\delta-\psi^\delta\|_\infty = 0$.


Let us now prove that we can choose $m_1\in\N$ to make $I_2$ arbitrarily small uniformly in $n\geq m\geq m_1$. Notice that for $n\geq m$ and $\bu\in \cD_p$, $p\in\{n,m\}$, we may rewrite
\[\ol M_n(\bu-\tfrac1n + \Delta_p) \;=\; \sum_{\bw\in (\bu+\Delta_p)\cap \cD_{nm}} \ol M_n(\bw-\tfrac1n+\Delta_{nm})\;,
\]
which gives, with the definition \eqref{eq:intdiscrete:k>1},
\begin{equation}\label{eq:lem:controlunifk>1:2}
\psi^\delta_m\overset{k}\cdot \ol M_n - \psi^\delta_n\overset{k}\cdot \ol M_n \;=\; \sum_{\bw\in (\cD_{nm})^k} \big(\psi^\delta_m(\bw)-\psi^\delta_n(\bw)\big) \prod_{l=1}^k \ol M_n(\bw_l-\tfrac1n+\Delta_{nm})\;.
\end{equation}
Therefore, recalling that the correlations of the field $\ol M_n$ are non-negative, we have
\[\begin{aligned}
\| \psi^\delta_m\overset{k}\cdot \ol M_n - \psi^\delta_n\overset{k}\cdot \ol M_n \|_{L^2}^2 \,&\leq\, \big(\|\psi_m^\delta-\psi_n^\delta\|_\infty\big)^2 \!\!\sum_{\bw\in (\cD_{nm})^{2k}}\! \bbE\Big[\prod_{l=1}^{2k}\ol M_n(\bw_l-\tfrac1n+\Delta_{nm})\Big]\\
&\leq\, \big(\|\psi_m^\delta-\psi_n^\delta\|_\infty\big)^2 \pt \bbE\big[\ol M_n([\bzero,\bt))^{2k}\big]\;=\; \big(\|\psi_m^\delta-\psi_n^\delta\|_\infty\big)^2 \pt \bbE\big[\ol M_n(\bt)^{2k}\big]\;.
\end{aligned}\]

Recall that $(\bbE[\ol M_n(\bt)^{2k}])_{n\geq1}$ converges, so it is bounded. Thus, we conclude the proof by choosing $m_1\in\N$ such that, for $n\geq m\geq m_1$, $\|\psi_m^\delta-\psi_n^\delta\|_\infty$ is sufficiently small.

\subsubsection{Term $I_3$} 
Let $m\in\N$ be fixed and $\delta \geq 0$ (we allow $\delta=0$). Since $\psi_m^\delta$ (or $\psi^0_m=\psi_m$) is constant on each $\prod_{l=1}^k(\frac1m\bi_l + \Delta_m)$, $\bi_1,\ldots,\bi_k\in\N_0^2$, there exists a family $\{ a_{\bw}\in\R ; \bw =(\bw_1,\ldots,\bw_k)\in (\cD_m)^k \}$ such that for $(\bu_1,\ldots,\bu_k)\in[\bzero,\bt)^k$,
\begin{equation}
\psi_m^\delta(\bu_1,\ldots,\bu_k)\;=\;\sum_{\bw \in (\cD_m)^k} a_{\bw} \prod_{l=1}^k \ind_{[\bzero,\bw_l)}(\bu_l)\;.
\end{equation}
Thus, starting from the definition~\eqref{eq:intdiscrete:k>1}, we get 
\begin{equation}
\label{eq:convk>1:ipp}
\begin{aligned}
\psi_m^\delta \overset{k}\cdot \ol M_n \;&=\; 
\sum_{\bu\in (\cD_m)^k} \bigg(\sum_{\bw\in (\cD_m)^k} a_\bw  \prod_{l=1}^k \ind_{[\bzero,\bw_l)}(\bu_l) \bigg) \prod_{l=1}^k \ol M_n \big(\bu_l - \tfrac1n\bone + \Delta_m\big) \\
&=\; \sum_{\bw\in (\cD_m)^k} a_\bw \sum_{\substack{\bu\in (\cD_m)^k \\ \bu\in \prod_{l=1}^k [\bzero,\bw_l)}} \prod_{l=1}^k \ol M_n \big(\bu_l - \tfrac1n\bone + \Delta_m\big) =\;\sum_{\bw\in (\cD_m)^k} a_\bw \pt\prod_{l=1}^k \pt \ol M_n(\bw_l -\tfrac1n\bone)\;.
\end{aligned}
\end{equation}
Recalling the definition of the integral $\overset{k}\diamond$, we have $(\prod_{l=1}^k \ind_{[\bzero,\bw_l)}) \overset{k}\diamond \cM =  \prod_{l=1}^k \cM(\bw_l)$. Hence we may also write with a similar computation
\begin{equation}\label{eq:convk>1:ippcM}
\psi_m^\delta\overset{k}\diamond \cM \;=\; \sum_{\bw\in (\cD_m)^k} a_\bw \pt \prod_{l=1}^k \pt \cM(\bw_l)\;. 
\end{equation}
Using Proposition~\ref{prop:lim:M}, it is clear that for fixed $\delta>0, m\in \mathbb N$ we have the convergence
\begin{equation}\label{eq:convk>1:ippcM2}
\sum_{\bw\in (\cD_m)^k} a_\bw \pt\prod_{l=1}^k \pt \ol M_n(\bw_l) \,\;\;\xrightarrow[n\to\infty]{L^2}\,\;\; \sum_{\bw\in (\cD_m)^k} a_\bw \pt \prod_{l=1}^k \pt \cM(\bw_l)\,. 
\end{equation}
It remains to replace $\ol M_n(\bw_l -\tfrac1n\bone)$ with  $\ol M_n(\bw_l)$ in~\eqref{eq:convk>1:ipp}.
For $\bw_1,\ldots,\bw_k\in[\bzero,\bt)$, we have
\begin{align*}
&\Big\|\prod_{l=1}^k\ol M_n(\bw_l)-\prod_{l=1}^k\ol M_n(\bw_l -\tfrac1n\bone)\Big\|_{L^2}\\
& \quad\leq \; \sum_{l=1}^k\Big\|\big(\ol M_n(\bw_l)-\ol M_n(\bw_l -\tfrac1n\bone)\big)\prod_{j=1}^{l-1}\ol M_n(\bw_l)\prod_{j=l+1}^k\ol M_n(\bw_l-\tfrac1n\bone)\Big\|_{L^2}\\
& \quad\leq \;\sum_{l=1}^k\Bigg(\Big\|\ol M_n(\bw_l)-\ol M_n(\bw_l -\tfrac1n\bone)\Big\|_{L^{2k}} \prod_{j=1}^{l-1}\Big\|\ol M_n(\bw_l)\Big\|_{L^{2k}}\prod_{j=l+1}^k\Big\|\ol M_n(\bw_l -\tfrac1n\bone)\Big\|_{L^{2k}}\Bigg)\;,
\end{align*}
where we used H\"older's inequality. In each term of the sum, the $k-1$ last factors are all uniformly bounded by $\|\ol M_n(\bt)\|_{L^{2k}}\leq C\|\cM(\bt)\|_{L^{2k}}<\infty$ (recall Claim~\ref{claim:cvgps}).
Then, the first factor goes to 0, thanks to Lemma~\ref{lem:controlM}, which proves that $\|\prod_{l=1}^k\ol M_n(\bw_l)-\prod_{l=1}^k\ol M_n(\bw_l -\tfrac1n\bone)\|_{L^2}$ goes to~$0$ as $n\to\infty$. 
Recollecting (\ref{eq:convk>1:ipp}--\ref{eq:convk>1:ippcM2}), this concludes the proof that $\lim_{n\to\infty} I_3 =0$.
Note that all the proof was also valid in the case $\delta=0$, so we also have that $\lim_{n\to\infty}\|\psi_m \overset{k}\cdot \ol M_n - \psi_m \overset{k}\diamond \cM \|_{L^2} =0$.\qed

\subsection{Conclusion of the proof of Theorem~\ref{conj:scalinggPS}}\label{sec:proofthm_mainproof}
With the help of Proposition~\ref{prop:convk>1} and thanks to the first item of Proposition~\ref{lemma:tronck>1}, we are able to conclude the proof of Theorem~\ref{conj:scalinggPS}.

Indeed, in view of \eqref{eq:intdiscrete:k>1:bis} and using the fact that $\lim_{n\to\infty} \frac{\beta_n^r}{n^{\frac12-\alpha} L(n)}= \hat \beta$ by~\eqref{def:scalings}, Proposition~\ref{prop:convk>1} shows that, for any fixed $k$,
\[
\tilde Z_{n,k} \xrightarrow[n\to\infty]{L^2} (\sigma_r \hat \gb)^k \, \psi \overset{k}\diamond \cM = \idotsint \limits_{ \bzero \prec \bs_1 \prec \cdots \prec \bs_k \prec \bt }  \psi_{\bt} \big( \bs_1,\ldots, \bs_k  \big) \prod_{j=1}^k \sigma_r \hat\gb^r \,\dd \cM(\bs_j)  \,.
\]
Then, item (i) of Proposition~\ref{lemma:tronck>1} shows that there is a constant $C>0$ and some $n_0\geq 0$ such that, for all $k\in \mathbb N$ and $n\geq n_0$,
\[
\| \tilde Z_{n,k}\|_{L^2} \leq \frac{C^k}{\Gamma( k(\alpha-\frac12))^{1/2}} + (C \gb_n^{c/2})^k 
\leq \frac{C^k}{\Gamma( k(\alpha-\frac12))^{1/2}}  +2^{-k}\,.
\]
Hence, we have the following bound on the $L^2$ norm of rest of the series, valid uniformly for $n\geq n_0$:
\[
\Big\| \sum_{k\geq k_0} \tilde Z_{n,k} \Big\|_{L^2}
\leq  \sum_{k\geq k_0} \| \tilde Z_{n,k}\|_{L^2} \leq \sum_{k\geq k_0} \Big(\frac{C^k}{\Gamma( k(\alpha-\frac12))^{1/2}}  +2^{-k} \Big)\,,
\]
which can be made arbitrarily small by choosing $k_0$ large.
Hence, fixing $k_0>0$, letting $n\to\infty$ and then $k_0\to\infty$,  we get that
\[
\sum_{k =0}^{\infty} \tilde Z_{n,k} \xrightarrow[n\to\infty]{L^2} \sum_{k=0}^{\infty} (\sigma_r \hat \gb)^k \, \psi \overset{k}\diamond \cM  = \bZ_{\bt, \hat h=0}^{\hat \gb, \cM,\quen} \,,
\]
using also Corollary~\ref{corol:ub:psikk} to ensure that the r.h.s.\ is well-defined.\qed

\subsection{Proof of Proposition~\ref{lemma:tronck>1}}\label{sec:prooftronck>1}

The only thing that remains to be proven is now  Proposition~\ref{lemma:tronck>1}.
We start with the proof of the first item and then we build on that proof to deal with the second item.

\subsubsection{Proof of item (i) of Proposition~\ref{lemma:tronck>1}}
Notice that from~\eqref{eq:intdiscrete:k>1:bis},
$\|\psi_n  \overset{k}\cdot \ol M_n\|_{L^2} \leq C^k \|\tilde Z_{n,k}\|_{L^2}$.
We therefore need to control the $L^2$ norm of $\tilde Z_{n,k}$ defined in~\eqref{def:tildeZ}, that can be written more compactly as
\begin{equation}
\label{L2start1}
\tilde Z_{n,k} = n^{2-\alpha} L(n) \sum_{\bI \in \cI_k } u_n(\bI)  \prod_{\bi\in \bI}  \zeta_{\bi} \,,
\end{equation}
where we have denoted 
\[
\mathcal I_k = \mathcal I_k (\bt) = \Big\{\bI=(\bi_1, \ldots, \bi_k) , \bzero =: \bi_0 \prec \bi_1 \prec \cdots \prec \bi_k \prec n\bt =:\bi_{k+1}\Big\}
\]
the set of increasing subsets of $\llbracket \bone, n \bt \rrbracket$ with cardinality $k$, and defined $u_n(\bI) := \prod_{l=1}^{k+1} u(\bi_l-\bi_{l-1})$ (with by convention $
bi_0=\bzero$, $\bi_{k+1}=n\bt$).
With these notation, we need to control
\begin{equation}
\label{L2start}
 \|\tilde Z_{n,k}\|_{L^2}^2
 = (n^{2-\alpha} L(n))^2 \sum_{\bI,\bJ \in \cI_k} u_n(\bI) u_n(\bJ)   \bbE\Big[  \prod_{\bi\in \bI, \bj \in \bJ}  \zeta_{\bi} \zeta_{\bj} \Big] \,.
\end{equation}

\subsubsection*{Step 1. Controlling the correlation term}
Since the indices in $\bI,\bJ$ are increasing, the set $\bI\cup \bJ$ can be uniquely partitioned (as done in \cite[Prop.~3.4]{Leg21}) into disjoint sets
\begin{equation}
\label{eq:decompset}
 \biota \cup \bnu  \cup \bigcup_{m\in \bbN} \bsigma_m \,,
\end{equation}
where:
\begin{itemize}[leftmargin=20pt,topsep=1pt]
\item $\biota$ is the set of \emph{isolated} points, \textit{i.e.}\ indices $\bi \in \bI$ (resp.\ in $\bJ$) such that $\bj \not \aligne \bi$ for any $\bj \in \bJ$ (resp.\ in~$\bI$);

\item $\bnu$ is the set of \emph{intersection} points, \textit{i.e.}\ indices that are both in $\bI$ and $\bJ$;

\item $\sigma_m$ are \textit{chains} of indices, \textit{i.e.}\ indices $\bi_1, \ldots, \bi_p$ (necessarily alternating between $\bI$ and $\bJ$) such that $\bi_{l+1} \aligne \bi_l$ for all $1\leq l\leq p-1$ and such that any other index in $\bI\cup \bJ$ is not aligned with any of the~$\bi_l$; the integer $p\geq 2$ is called the length of the chain. 
\end{itemize}
By independence of the $\zeta_{\bi}$ for non-aligned sets, we get that
\begin{align*}
\bbE\Big[  \prod_{\bi\in I, \bj \in J}  \zeta_{\bi} \zeta_{\bj} \Big]
& = \bbE[ \zeta_{\bone}^2 ]^{|\bnu|} \bbE[ \zeta_{\bone} ]^{|\biota|}
\prod_{m\in \bbN} \bbE\Big[ \prod_{\bi \in \bsigma_m}  \zeta_{\bi} \Big] 
\end{align*}
For fixed $\bI,\bJ$, let us denote $N_0(\bI,\bJ) =|\biota|$ the number of isolated points, $N_1(\bI,\bJ) = |\bnu|$ the number of intersection points and $N_{p}(\bI,\bJ) = |\{m, |\bsigma_m|=p\}|$ the number of chains of length $p$.
The correlation is equal to~$0$ when $N_0(\bI,\bJ) \geq 1$, and in the case $N_{0}(\bI,\bJ)=0$ we get, thanks to Lemma~\ref{lem:multicorrel}-\eqref{chaincorrel},
\begin{equation}
\label{correl:structure}
\bbE\Big[  \prod_{\bi\in \bI, \bj \in \bJ}  \zeta_{\bi} \zeta_{\bj} \Big]\leq  C^k  \gb_n^{2 N_1(\bI,\bJ)}   \prod_{p=2}^k \big( \gb_n^{2r + (p-2)  \lceil \frac r2 \rceil}\big)^{N_p(\bI,\bJ)} 
\end{equation}
where we also used that $|\bI|=|\bJ|=k$ to get that the power of the constant is $2|\bnu|+|\biota|+ \sum_{m\geq 1} |\bsigma_m| = 2k$.

Going back to~\eqref{L2start}, we can decompose $ \|\tilde Z_{n,k}\|_{L^2}^2$ into two parts.
The first part contains the main contribution, which comes from sets of indices $\bI,\bJ$ such that $\bI\cup \bJ$ contains only chains of length $2$, \textit{i.e.}\ such that $N_2(\bI,\bJ)=k$ (and necessarily $N_p(\bI,\bJ)=0$ for all $p\neq 2$):
using~\eqref{correl:structure}, this is bounded by $C^k$ times
\begin{equation}
\label{L2main}
\cK_1 = \cK_1(k,n) := (n^{2-\alpha} L(n))^2\, (\gb_n^{2r})^k \sum_{\bI,\bJ \in \cI_k,\, N_2(\bI,\bJ)=k} u_n(\bI) u_n(\bJ) \,.
\end{equation}
The  second part, containing contributions from all other sets of indices $\bI,\bJ$, will be negligible: decomposing over the value of $N_p(\bI,\bJ)$ and using~\eqref{correl:structure}, we get that it is bounded by 
$C^k$ times
\begin{equation}
\label{L2rest}
\cK_2 = \cK_2 (k,n) := (n^{2-\alpha} L(n))^2 \!\!\! \sumtwo{q_1,\ldots, q_{2k}\geq 0, q_2<k}{2q_1+ \sum_{p=2}^{2k} pq_p = 2k}
\sumtwo{\bI,\bJ \in \cI_k}{N_p(\bI,\bJ)= q_p } u_n(\bI) u_n(\bJ)  \gb_n^{2q_1 + \sum_{p=2}^{k} q_p (2r+(p-2) \lceil \frac r2 \rceil)}\,.
\end{equation}

\begin{lemma}
\label{lemma:L2main}
There exists a constant $C>0$ and $n_0\geq 1$ such that 
for all $n\geq n_0$ and $k\in \mathbb N$,
\[
\cK_1  \leq \frac{ C^k  }{\Gamma(1+ k(\alpha-\frac12))} \,.
\]
\end{lemma}

\begin{lemma}
\label{lemma:L2rest}
There exist constants $C,c>0$ and $n_0\geq 1$ such that 
for all $n\geq n_0$ and $k\in \mathbb N$,
\[
\cK_2  \leq 
\frac{  C^k \gb_n}{\Gamma(1+ k(\alpha-\frac12))} + C^k \gb_n^{c k} \,.
\]
\end{lemma}

\noindent
These two lemmas readily conclude the proof of item (i) in Proposition~\ref{lemma:tronck>1}, so let us now prove them.

\begin{proof}[Step 2. Proof of Lemma~\ref{lemma:L2main}]
Note that one can rewrite $u_n(\bI) = \bP(\bI \subset \btau , n \bt \in \btau)$.
Therefore, letting  $\btau, \btau'$ be two independent bivariate renewals with joint law denoted $\bP^{\otimes 2}$, 
we have that
\begin{equation}
\label{eq:rewritingreplicas}
\begin{split}
\cK_1 & = (\gb_n^{2r})^k   \sum_{\bI,\bJ \in \cI_k,\, N_2(\bI,\bJ)=k} (n^{2-\alpha} L(n))^2 \bP^{\otimes 2} (\bI\subset \btau, \bJ\subset \btau', n \bt \in \btau\cap \btau')  \\
    & \leq  C_{\bt} (\gb_n^{2r})^k  \sum_{\bI,\bJ \in \cI_k,\, N_2(\bI,\bJ)=k} \bP^{\otimes 2} \big( \bI\subset \btau, \bJ\subset \btau' \mid n \bt \in \btau\cap \btau' \big) \\
    & =C_{\bt} (\gb_n^{2r})^k \bE^{\otimes 2}_n \Big[ \big| \big\{ (\bI, \bJ)\in \cI_k^2\,,N_2(\bI,\bJ)=k,  \bI\subset \btau, \bJ\subset \btau' \big\}\big| \Big] \,,
\end{split}
\end{equation}
where we have used that $n^{2-\alpha} L(n) \leq C'_{\bt} \bP( n\bt \in \btau)^{-1}$ for some constant $C'_{\bt}$, see Proposition~\ref{thm:renouv}. We also introduced the notation $\bP^{\otimes 2}_n (\cdot) = \bP^{\otimes 2} ( \;\cdot \mid n \bt \in \btau\cap \btau' )$.
Now, we denote
\[
\cC_2(\btau,\btau') = \big|\big\{(\bi,\bj) \in \llbracket \bone, n\bt\rrbracket^2\,, \bi\in \btau, \bj\in \btau', \bi\aligne \bj \big\}\big|\,,
\]
\textit{i.e.}\ the number of pairs of aligned points in $(\btau\cup \btau')\cap \llbracket \bone, n\bt \rrbracket$, so we have that
\[
\big| \big\{ (\bI, \bJ)\in \cI_k^2\,, \bI\subset \btau, \bJ\subset \btau' \big\}\big| \leq \binom{\cC_2(\btau,\btau')}{k} \leq \frac{1}{k!} \cC_2(\btau,\btau')^k \,.
\]
We end up with
\[
\cK_1\leq  \frac{ C^k}{k!} \bE^{\otimes 2}_n \Big[ \big( \gb_n^{2r} \cC_2(\btau,\btau')\big)^k \Big] \,.
\]
Recall that the projection of $\btau$ on its $a$-th coordinate is denoted $\btau^{(a)}$, $a\in\{1,2\}$.  Now, notice that $\cC_2(\btau,\btau') \leq |\rho^{(1)}\cap [0,nt_1]| + |\rho^{(2)} \cap [0,nt_2]| $, 
where we have denoted $\rho^{(1)} = \btau^{(1)}\cap \btau'^{(1)}$ and $\rho^{(2)}=\btau^{(2)}\cap \btau'^{(2)}$ for simplicity. Using that $(x+y)^k\leq 2^k (x^k+y^k)$, 
and that the law of $\rho^{(a)}$ conditionally on $nt_a\in\rho^{(a)}$ is symmetric in $\frac12 nt_a$, we have for $a\in\{1,2\}$ the upper bound
\[
\bE^{\otimes 2} \bigg[ \bigg( \gb_n^{2r} \sum_{i=1}^{ nt_a} \ind_{\{i \in \rho^{(a)}\}}  \bigg)^k \;\bigg|\; n t_a\in \rho^{(a)} \bigg] \,\leq\, C \, 2^{k+1} \, \bE^{\otimes 2} \bigg[ \bigg( \gb_n^{2r} \sum_{i=1}^{ \frac12 nt_a} \ind_{\{i \in \rho^{(a)}\}}  \bigg)^k  \bigg] \,,
\]
where we used \cite[Lem.~A.2]{GLT10} to remove the conditioning, at the cost of a constant factor.

Then, we can  bound the term above thanks to Lemma~\ref{lem:tailandmoments}. 
Indeed, from Remark~\ref{rem:appliclemmatail}, we have that
\[
U_{\frac12 n t_a} =\sum_{i=1}^{\frac12 nt_a} \bP ( i \in \rho^{(a)}) \sim c_{\alpha,t_a}  n^{2\alpha-1} L(n)^{-2}\sim C_{\alpha,\hat \gb, t_a} \gb_n^{ -2r}  \qquad \text{ as } n\to\infty \,,
\]
where we plugged in the definition~\eqref{def:scalings} of $\gb_n$ for the last identity.
Therefore, letting $\gamma:=2\alpha-1$ and $0<\delta < \gamma$ in Lemma~\ref{lem:tailandmoments}, 
we have that there exists a constant $C>0$ such that 
\[
\bE^{\otimes 2} \bigg[ \bigg( \gb_n^{2r} \sum_{i=1}^{ \frac12 nt_a} \ind_{\{i \in \rho^{(a)}\}}  \bigg)^k  \bigg] \leq C^k  \, \Gamma\big( k(1-\gamma+\delta) +1 \big) \,.
\]
This concludes the proof, since $\Gamma( k(1-\gamma+\delta)+1 ) \leq C'^k \frac{k!}{ \Gamma(k(\gamma-\delta)+1)}$ thanks to Stirling's formula,
then taking $\delta =\frac12\gamma = \alpha-\frac12$.
\end{proof}

\begin{remark}
\label{rem:finiteexpmoment}
Let us note for future use that we have proven above that for any $\delta\in(0,2\ga-1)$ there is a constant $C>0$ such that for any $k\geq 1$
\begin{equation*}
\frac{1}{k!}\bE^{\otimes 2}_n \bigg[ \bigg( \gb_n^{2r}  \sum_{i=1}^{ nt_a} \ind_{\{i \in \rho^{(a)}\}} \bigg)^k  \bigg] \leq \frac{C^k}{k!} \Gamma\big( k( 2(1-\alpha)+\delta) +1 \big)  \leq  \frac{C'^k}{ \Gamma(k(2\alpha-1-\delta)+1 ) }\,.
\end{equation*}
Using that $\sum_{j\geq 0} \frac{u^j}{\Gamma(bj+1)}\leq C e^{2 u^{1/b}}$, see e.g.~\cite[Thm.~1]{Ger12}, we get that there is a constant $c>0$ such that for any $u>0$, 
\begin{equation}
\label{finiteexpmoment}
\bE^{\otimes 2}_n \bigg[ \exp \bigg( u \gb_n^{2r}  \sum_{i=1}^{ nt_a} \ind_{\{i \in \rho^{(a)}\}} \bigg)   \bigg] 
\leq  c^{-1} e^{ c u^{1/(2\alpha-1+\delta)}} <+\infty \,.
\end{equation}
\end{remark}

\begin{proof}[Step 3. Proof of Lemma~\ref{lemma:L2rest}]
We proceed similarly as above, the proof being more involved.
Rewriting $u_n(\bI) = \bP(\bI \subset \btau , n \bt  \in \btau)$, we get as in~\eqref{eq:rewritingreplicas}
that $\cK_2$ can be bounded by a constant $C_{\bt}$ times
\[
\sumtwo{q_1,\ldots, q_{2k}\geq 0, q_2<k}{2q_1+ \sum_{p=2}^{2k} pq_p = 2k}
 \gb_n^{2q_1 + \sum_{p=2}^{k} q_p(2 r + (p-2) \lceil \frac r2 \rceil)} \bE^{\otimes 2}_n \Big[ \big| \big\{ (\bI, \bJ)\in \cI_k^2\,,N_p(\bI,\bJ)=q_p \, \forall 1\leq p \leq 2k,  \bI\subset \btau, \bJ\subset \btau' \big\}\big|  \Big] \,.
\]
Now, similarly as above, we easily get that
\[
\big| \big\{ (\bI, \bJ)\in \cI_k^2\,,N_p(\bI,\bJ)=q_p \, \forall 1\leq p \leq 2k,  \bI\subset \btau, \bJ\subset \btau' \big\}\big| 
\leq \binom{|\btau\cap \btau' \cap \llbracket \bone, n\bt \rrbracket |}{q_1} \prod_{p=2}^{2k} \binom{\cC_p(\btau,\btau')}{q_p} \,,
\]
where $\cC_p(\btau,\btau')$ is the number of chains of length $p$ contained in $(\btau \cup \btau')\cap \llbracket \bone,n\bt\rrbracket$. In the end,
we need to bound
\begin{equation}
\label{vgk22}
\bE^{\otimes 2}_n \Bigg[  \sumtwo{q_1,\ldots, q_{2k}\geq 0, q_2<k}{2q_1+ \sum_{p=2}^{2k} pq_p = 2k} \frac{1}{ \prod_{p=1}^k q_p!}  \Big( \gb_n^2  |\btau\cap \btau' \cap \llbracket \bone, n\bt\rrbracket| \Big)^{q_1} \prod_{p=2}^{2k} \Big( \gb_n^{2r + (p-2) \lceil \frac r2 \rceil} \cC_p(\btau,\btau')\Big)^{q_p}  \Bigg] \,.
\end{equation}

Let us now make an observation: for any $\lambda =(\lambda_1, \ldots, \lambda_{2k})$ with $\lambda_i\geq 0$, we can write
\[
\sumtwo{q_1,\ldots, q_{2k}\geq 0, q_2<k}{2q_1+ \sum_{p=2}^{2k} pq_p = 2k}  \prod_{p=1}^{2k} \frac{\lambda_p^{q_p}}{ q_p!}  
= e^{\sum_{p=1}^{2k} \lambda_p} 
\bP_{\lambda} \Big( 2Q_1+ \sum_{p=2}^{2k} pQ_p =2k \Big) \,,
\]
where $\bP_\lambda$ is the law of independent Poisson random variables $(Q_p)_{1\leq p \leq 2k}$ with respective parameters $(\lambda_p)_{1\leq p \leq 2k}$ (by convention $Q_p=0$ if $\lambda_p=0$).
We can therefore rewrite \eqref{vgk22} as
\begin{equation}
\label{vgk23}
\bE^{\otimes 2}_n \bigg[ e^{\sum_{p=1}^{2k} \blamb_p}  \bP_{\blamb} \Big( 2Q_1+ \sum_{p=2}^{2k} pQ_p =2k, Q_2<k \Big)  \bigg] \,,
\end{equation}
where $\blamb = \blamb^{(n)}(\btau,\btau')$ is defined by 
\begin{equation}
\label{def:blamb}
\blamb_1 = \gb_n^2  \,|\btau\cap \btau' \cap \llbracket \bone, n\bt \rrbracket|\,,
\qquad 
\blamb_p= \gb_n^{2r + (p-2) \lceil \frac r2 \rceil} \cC_p(\btau,\btau') \text{ for }p\geq 2 \,.
\end{equation}
Using H\"older's inequality and Cauchy--Schwarz's inequality,  for any $\gep>0$ (fixed small enough), we have that~\eqref{vgk23} is bounded by
\begin{equation}
\label{vgk24}
\bE^{\otimes 2}_n \bigg[ e^{2\frac{1+\gep}{\gep} \blamb_1  }\bigg]^{\frac{\gep}{2(1+\gep)}} 
\bE^{\otimes 2}_n \bigg[ e^{2\frac{1+\gep}{\gep} \sum_{p=2}^{2k} \blamb_p}    \bigg]^{\frac{\gep}{2(1+\gep)}} 
\bE^{\otimes 2}_n \bigg[  \bP_{\blamb} \Big( 2Q_1+ \sum_{p=2}^{2k} pQ_p =2k , Q_2<k \Big)^{1+\gep}\bigg]^{\frac1{1+\gep}} \,.
\end{equation}

\medskip
\noindent
{\it First term in~\eqref{vgk24}.}
From \cite[Prop.~A.3]{BGK20}, we know that when $\alpha <1$ the intersection $\btau\cap \btau'$ is terminating, so $|\btau\cap \btau'|$ is a geometric random variable. Lemma~\ref{lem:interesctcond} below states that the conditioning does not change this very much: it gives that 
\[
\bE^{\otimes 2}_n \bigg[ e^{2\frac{1+\gep}{\gep} \blamb_1}\bigg]
\leq \bE^{\otimes 2}_n \bigg[ e^{2\frac{1+\gep}{\gep} \gb_n^2 |\btau\cap \btau'|}\bigg]
\leq \sum_{k\geq 0} e^{2\frac{1+\gep}{\gep} \gb_n^2 k} \bP_n^{\otimes 2} \big( |\btau\cap \btau'| > k \big)
\leq  C \,,
\]
where the last line holds for $n$ large enough (so that $2\frac{1+\gep}{\gep}\gb_n^2$ is smaller than half the constant $c$ appearing in Lemma~\ref{lem:interesctcond}).

\medskip
\noindent
{\it Second term in~\eqref{vgk24}.} Since  a chain in $\btau\cup \btau'$ of length $p'\geq p$ contains exactly $p'-p+1$ chains of length $p$, we get that the number of $p$-chains included in $\btau\cup \btau'$ is
\begin{equation}
\label{relCN}
\cC_p(\btau,\btau') = \sum_{p'\geq p} (p'-p+1) N_{p'}(\btau,\btau')\,,
\end{equation}
where we recall that $N_p(\btau,\btau')$ is the number of (maximal) chains of length $p$ in the decomposition~\eqref{eq:decompset} of $(\btau\cup \btau')\cap \llbracket \bone, n\bt \rrbracket$. 
We therefore get that,
\[
\sum_{p=2}^{2k} \blamb_p = \sum_{p=2}^{2k} \gb_n^{2r + (p-2) \lceil \frac r2 \rceil} \cC_p(\btau,\btau')
 =  \gb_n^{2r} \sum_{p'=2}^{2k} N_{p'}(\btau,\btau')  \sum_{p=2}^{p'} (p'-p  +1) \gb_n^{(p-2)\lceil \frac r2 \rceil } 
 \leq  2 \gb_n^{2r}  \sum_{p'=2}^{2k} p'  N_{p'}(\btau,\btau')  \,,
\]
where we have used that  $ \sum_{p=2}^{p'} (p'-p+1) \gb_n^{ (p-2) \lceil \frac r2 \rceil} \leq 2 p'$ provided that $n$ is large enough so that $\gb_n^{ \lceil \frac r2 \rceil}\leq 1/2$.
Notice also that we have
\begin{equation}
\label{relCN2}
\sum_{p'=2}^{2k} p' N_{p'}(\btau,\btau') \leq 2\big( |\rho^{(1)}\cap [0,nt_1]| +|\rho^{(2)} \cap [0,nt_2]| \big) \,,
\end{equation}
 where we recall that $\rho^{(a)} = \btau^{(a)}\cap \btau'^{(a)}$, $a\in\{1,2\}$.
Indeed, the left-hand side is the total length of all the chains of length larger than $2$ in $\btau\cup \btau'$
and point in a chain belongs either to $\rho^{(1)}=\btau^{(1)}\cap \btau'^{(1)}$, to $\rho^{(2)}=\btau^{(2)}\cap \btau'^{(2)}$ or to both (one may also refer to~\cite[Eq.~(3.22)]{Leg21}).
Hence, we get that
\[
\bE^{\otimes 2}_n \bigg[ e^{2\frac{1+\gep}{ \gep} \sum_{p=2}^n \blamb_p}    \bigg]\leq 
\bE^{\otimes 2}_n \bigg[ \exp\bigg( 8\frac{1+\gep}{\gep} \gb_n^{2r} \Big( \sum_{i=1}^{ nt_1} \ind_{\{i \in \rho^{(1)}\}} + \sum_{i=1}^{ nt_2} \ind_{\{i \in \rho^{(2)}\}}\Big) \bigg)\bigg]  \,,
\]
which is finite thanks to Remark~\ref{rem:finiteexpmoment}, see~\eqref{finiteexpmoment} (after using Cauchy--Schwarz inequality to deal with $\rho^{(1)}$ and $\rho^{(2)}$ separately).

\medskip
\noindent
{\it Third term in~\eqref{vgk24}.} Denoting $\tilde Q_3 := \sum_{p\geq 3} p Q_p$, we want to show that there is a constant $C,c>0$ such that, for $n$ large enough (how large must not depend on $k$)
\begin{equation}
\label{vgk25}
\bE^{\otimes 2}_n \bigg[  \bP_{\blamb} \Big( 2Q_1+2Q_2 +\tilde Q_3 =2k , Q_2<k \Big)^{1+\gep}\bigg]^{\frac1{1+\gep}} 
\leq  \frac{  C^k \gb_n }{ \Gamma(k (\alpha-\frac12)+1)}   + C^k \gb_n^{ck}\,.
\end{equation}
We separate the estimate into three parts, according to the three following events: we fix some $\gh\in(0,1)$ (its precise value is given below), (i)~$\tilde Q_3 =0$ and $Q_1\geq 1$; (ii)~$Q_1+Q_2  \geq (1-\eta) k, \tilde Q_3 \geq 1$;
(iii) $\tilde Q_3 \geq 2\eta k$.

\smallskip
{\it Case (i).}
On the event that $\tilde Q_3=0$, we get that
\[
\begin{split}
\bP_{\blamb} \Big( 2Q_1+ 2Q_2 +\tilde Q_3 =2k,  Q_2<k , \tilde Q_3 =0\Big) & \le \bP_{\blamb} \Big( Q_1+ Q_2 =k , Q_2<k\Big) \\
&= \frac{e^{-( \blamb_1+\blamb_2)}}{k!} \sum_{\ell=1}^{k} \binom{n}{k} \blamb_1^\ell \blamb_2^{k-\ell}
 \leq \frac{2^k}{k!}  \big( \blamb_1 \blamb_2^{k-1} + \blamb_1^{k} \big) \,.
\end{split}
\]
where we have bounded $(\blamb_1/\blamb_2)^{\ell}$ by the maximum of $\blamb_1 /\blamb_2$ and $(\blamb_1/ \blamb_2)^{k}$ and bounded the sum of the binomial factors by $2^k$.
Plugging this in the l.h.s.\ of ~\eqref{vgk25}, and using that $(x+y)^\gamma\leq 2^\gamma (x^\gamma+y^\gamma)$ for $\gamma\geq 1$ and $(x+y)^{\gamma'}\leq x^{\gamma'}+y^{\gamma'}$ for $\gamma'<1$ so that
\begin{equation}
\label{usefulineq}
\bE\big[ (A+B)^{\gamma} \big]^{\gamma'} \leq 2^{\gamma\gamma'} \big( \bE\big[ A^{\gamma} \big]^{\gamma'} + \bE\big[ B^{\gamma} \big]^{\gamma'}  \big) \,,
\end{equation}
we get that
\begin{equation}
\label{eq:casei-0}
\bE^{\otimes 2}_n \bigg[  \bP_{\blamb} \Big( 2Q_1+2Q_2 +\tilde Q_3 =2k , Q_2<k , \tilde Q_3=0\Big)^{1+\gep}\bigg]^{\frac1{1+\gep}} 
\!\!\leq \frac{2^{k+1}}{k!} \bigg( \bE^{\otimes 2}_n \Big[  (\blamb_1 \blamb_2^{k-1})^{1+\gep}\Big]^{\frac{1}{1+\gep}} 
+ \bE^{\otimes 2}_n \Big[  \blamb_1^{(1+\gep)k}\Big]^{\frac{1}{1+\gep}}  \bigg) .
\end{equation}
For the first term, we use H\"older's inequality to get
\[
\bE^{\otimes 2}_n \Big[  (\blamb_1 \blamb_2^{k-1})^{1+\gep}  \Big]^{\frac{1}{1+\gep}}  
\leq \bE^{\otimes 2}_n \Big[  (\blamb_1)^{ \frac{(1+\gep)^2}{\gep}}   \Big]^{\frac{\gep}{(1+\gep)^2}}   \bE^{\otimes 2}_n \Big[   \blamb_2^{(1+\gep)^2 (k-1)}   \Big]^{\frac{1}{(1+\gep)^2}}  \,.
\]
Now, recalling that $\blamb_1 = \gb_n^2  \,|\btau\cap \btau' \cap \llbracket \bone, n\bt \rrbracket|$, 
we get that 
\[
\bE^{\otimes 2}_n \Big[  (\blamb_1)^{ \frac{(1+\gep)^2}{\gep}}  \Big]^{\frac{\gep}{(1+\gep)^2}} = \gb_n^{2} \bE^{\otimes 2}_n \Big[  |\btau\cap \btau'\cap  \llbracket \bone, n\bt \rrbracket|^{ \frac{(1+\gep)^2}{\gep}} \Big]^{\frac{\gep}{(1+\gep)^2}} 
\leq C_{\gep} \gb_n^{2} \,,
\]
using Lemma~\ref{lem:interesctcond} for the last inequality.
Using that $\blamb_2 \leq \gb_n^{2r} ( |\rho^{(1)}\cap [0,nt_1]| +|\rho^{(2)} \cap [0,nt_2]|)$ and applying~\eqref{usefulineq} with $\gamma = (1+\gep)^2(k-1), \gamma'=\frac{1}{(1+\gep)^2}$, we also get that for $\delta \in (0,2\alpha-1)$,
\[
\begin{split}
 \bE^{\otimes 2}_n \Big[   \blamb_2^{(1+\gep)^2 (k-1)}   \Big]^{\frac{1}{(1+\gep)^2}} 
 &\leq 2^{k-1} \sum_{a=1,2} \bE^{\otimes 2}_n \bigg[   \Big(\gb_n^{2r} \sum_{i=1}^{nt_a} \ind_{\{i \in \rho^{(a)}\}} \Big)^{(1+\gep)^2 (k-1)}   \bigg]^{\frac{1}{(1+\gep)^2}}  \\
& \leq C^k \Gamma\big( (1+\gep)^2k(2(1-\alpha)+\delta) +1 \big)^{1/(1+\gep)^2}
 \leq (C')^k \Gamma( k(2(1-\alpha)+\delta) +1) \,,
 \end{split}
\]
where first we have used  Remark~\ref{rem:finiteexpmoment} and then Stirling's asymptotics for Gamma functions.
Using again Stirling's formula, setting $\delta = \alpha-\frac12$, we have $\frac{1}{k!}\Gamma( k(1-(\alpha-\frac12)) +1)\leq C^k \Gamma(k(\alpha-\frac12) +1)^{-1}$.

The second term we need to estimate in~\eqref{eq:casei-0} is 
\[
\bE^{\otimes 2}_n \Big[   (\blamb_1)^{(1+\gep)^2 k}   \Big]^{\frac{1}{1+\gep}} = \gb_n^{2k } \bE^{\otimes 2}_n \Big[ |\btau\cap \btau'\cap  \llbracket \bone, n\bt \rrbracket |^{(1+\gep)^2k}  \Big]^{\frac{1}{1+\gep}} \leq C^k \gb_n^{2k} \,,
\]
using again Lemma~\ref{lem:interesctcond} for the last inequality. 

All together, we conclude that
\begin{equation}
\label{vgkCase1}
\bE^{\otimes 2}_n \bigg[  \bP_{\blamb} \Big( 2Q_1+2Q_2 +\tilde Q_3 =2k , Q_2<k , \tilde Q_3=0\Big)^{1+\gep}\bigg]^{\frac1{1+\gep}} 
\leq  \frac{ C^{k}\gb_n^2  }{\Gamma(k(\alpha-\frac12) +1)}   + \frac{C^k}{k!} \gb_n^{2k}\,.
\end{equation}

\smallskip
{\it Case (ii).} 
Let $\eta>0$ (fixed below) and consider the event $Q_1+Q_2 \geq (1-\eta) k$,
$\tilde Q_3 \geq 1$.
We have
\[
\begin{split}
\bP_{\blamb} \Big( 2Q_1+ 2Q_2 +\tilde Q_3 & =2k   ,    Q_1+Q_2 \geq (1-\eta) k, \tilde Q_3 \geq 1\Big) \\
&\leq 
\sum_{\ell=(1-\eta) k}^{k-1} \bP_{\blamb} \Big(\tilde Q_3 =2(k-\ell),  Q_1+ Q_2 =\ell \Big)\\
& \leq \max_{(1-\eta) k \leq \ell \leq k-1} \bP_{\blamb} \Big(Q_1+ Q_2 =\ell \Big) \bP_{\blamb} \Big(\tilde Q_3 \geq 1 \Big) 
\leq  \frac{ (1 \vee (\blamb_1+\blamb_2)^{k-1})}{ ((1-\eta)k)!} \sum_{p\geq 3}p \blamb_p\,,
\end{split}
\]
where we have used Markov's inequality for the last line; recall that $\tilde Q_3 = \sum_{p\geq 3} pQ_p$, with $(Q_p)_{p}$ independent Poisson random variables of respective parameter $\blamb_p$ given in~\eqref{def:blamb}. (We also omitted the integer part of $(1-\eta)k$ for simplicity.)
Now, notice that 
\[
\begin{split}
\sum_{p\geq 3} p \blamb_p = \sum_{p = 3}^{2k} \gb_n^{2r + (p-2) \lceil \frac r2 \rceil} p\pt\cC_p(\btau,\btau')
& \leq 2k   \gb_n^{2r+\lceil \frac r2 \rceil}   \sum_{p= 2}^{2k} \cC_p(\btau,\btau') \\
& \leq 8  k^2  \gb_n \times \gb_n^{2r} (|\rho^{(1)}\cap[0,nt_1]| +|\rho^{(2)}\cap [0,nt_2]|) \,,
\end{split}
\]
where we have also used~\eqref{relCN}-\eqref{relCN2} to get that $ \cC_p(\btau,\btau') \leq 2 (|\rho^{(1)}\cap[0,nt_1]| +|\rho^{(2)}\cap [0,nt_2]|)$ for all $p$.
All together,  bounding $(1 \vee (\blamb_1+\blamb_2)^{k-1}) \leq 1+ 2^k \blamb_1^{k-1} + 2^k \blamb_2^{k-1}$ and recalling also that 
$\blamb_1 = \gb_n^2  \,|\btau\cap \btau' \cap \llbracket \bone, n\bt \rrbracket|$, $\blamb_2 \leq \gb_n^{2r} (|\rho^{(1)}\cap[0,nt_1]| +|\rho^{(2)}\cap [0,nt_2]|)$, we obtain that
\begin{align*} 
&\bE^{\otimes 2}_n\bigg[
\bP_{\blamb} \Big( 2Q_1+2Q_2+ \tilde Q_3 = 2k , Q_1+Q_2 \geq (1-\eta) k, \tilde Q_3 \geq 1\Big)^{1+\gep} \bigg]^{\frac{1}{1+\gep}}
\\
& \leq \frac{ C^k \gb_n }{ ((1-\eta)k)!} \sum_{a=1,2}
\bigg(
 \bE^{\otimes 2}_n\bigg[  \Big( \gb_n^{2r} \sum_{i=1}^{nt_a} \ind_{\{i\in \rho^{(a)}\}} \Big)^{1+\gep}\bigg]^{\frac{1}{1+\gep}} \\
& \quad\quad +\bE^{\otimes 2}_n\bigg[  \Big( \big( \gb_n^2 |\btau\cap \btau'\cap  \llbracket \bone, n\bt \rrbracket| \big)^{k-1} \,  \gb_n^{2r} \sum_{i=1}^{nt_a} \ind_{\{i\in \rho^{(a)}\}} \Big)^{1+\gep} \bigg]^{\frac{1}{1+\gep}} 
+ \bE^{\otimes 2}_n\bigg[  \Big( \gb_n^{2r} \sum_{i=1}^{nt_a} \ind_{\{i\in \rho^{(a)}\}} \Big)^{k(1+\gep)} \bigg]^{\frac{1}{1+\gep}}
\bigg) \,,
\end{align*}
where we have again used~\eqref{usefulineq} (with $\gamma=1+\gep$, $\gamma' =\frac{1}{1+\gep}$).
The first expectation is bounded by a constant
thanks to Remark~\ref{rem:finiteexpmoment}.
The second expectation is also bounded by a constant. Indeed, using Cauchy--Schwarz inequality to treat both quantities separately, we use Lemma~\ref{lem:interesctcond} to show that $\bE^{\otimes 2}_n[  \big( \gb_n^2 |\btau\cap \btau'\cap \cap  \llbracket \bone, n\bt \rrbracket| \big)^{b} ]$ is bounded by some universal constant $C$ (in fact, the constant goes to $0$ as $b\to \infty$, provided that $\gb_n$ is small enough),
and then we use Remark~\ref{rem:finiteexpmoment} for the other term.
For the last term, we can again use  Remark~\ref{rem:finiteexpmoment} to get that
it is bounded by $C^k \Gamma( (1+\gep)k(2(1-\alpha)+\delta) +1)^{1/(1+\gep)} \leq C'^k \Gamma( k(2(1-\alpha)+\delta) +1)$, the last inequality following from Stirling's asymptotics.

Again by Stirling's formula, setting $\delta = \frac12 (\alpha-\frac12)$, $\eta =  \frac12 (\alpha-\frac12)$, we get that 
$\frac{1}{ ((1-\eta)k)!}\Gamma( k(2(1-\alpha)+\delta) +1) \leq C^k \Gamma( k(\alpha-\frac12)+1)^{-1}$.
All together, we have obtained that
\begin{equation}
\label{vgkCase2}
\bE^{\otimes 2}_n\bigg[
\bP_{\blamb} \Big( 2Q_1+2Q_2+ \tilde Q_3 = 2k,  Q_1+Q_2 \geq (1-\eta) k, \tilde Q_3 \geq 1\Big)^{1+\gep} \bigg]^{\frac{1}{1+\gep}}
\leq \frac{ C^k \gb_n}{\Gamma( k(\alpha-\frac12)+1)}\,.
\end{equation}

\smallskip
{\it Case (iii).}
It remains to control the case where $Q_1+Q_2 < (1-\eta) k$ and hence $\tilde Q_3 \geq 2\eta k$.
Let $a_n:= -\frac14 \log \gb_N $ and denote $\bA_n$ the event $\{ \sum_{p\geq 3} \blamb_p e^{pa_n} \leq k\}$.
We have that 
\[
\begin{split}
\bE^{\otimes 2}_n\bigg[
\bP_{\blamb} \Big( 2Q_1+2Q_2+ \tilde Q_3 = 2k, \tilde Q_3 \geq 2\eta k \Big)^{1+\gep} \bigg]^{\frac{1}{1+\gep}} \leq \bP^{\otimes 2}_n\big( \bA_n^c \big)^{\frac{1}{1+\gep}}
+ \bE^{\otimes 2}_n\Big[
\bP_{\blamb} \big( \tilde Q_3 \geq 2\eta k\big)^{1+\gep}  \ind_{\bA_n}\Big]^{\frac{1}{1+\gep}} \,.
\end{split}
\]

For the first term, notice that by the definition~\eqref{def:blamb} of $\blamb$ and recalling~\eqref{relCN}-\eqref{relCN2}, we get
\begin{align*}
\sum_{p=3}^{2k} \blamb_p e^{p a_n} & \leq 
\gb_n^{2r}\sum_{p'=3}^{2k} p' N_{p'}(\btau,\btau') 
\sum_{p=3}^{p'} \gb_n^{(p-2) \lceil \frac r2\rceil} e^{pa_n} \\
& \leq 2 \gb_n^{2r} \big( |\rho^{(1)}\cap [0,nt_1]| +|\rho^{(2)}\cap [0,nt_2]| \big)  \sum_{p=3}^{\infty} \gb_n^{p/3} e^{pa_n} \,,
\end{align*}
where we have used that $p-2\geq p/3$ for $p\geq 3$.
Now, using the definition of $a_n$ we get that $\gb_n^{p/3} e^{pa_n} = \gb_n^{p/12}$ so the last sum is bounded by a constant times $\gb_n^{1/4}$.
Hence, we get that for $\delta \in (0,2\alpha-1)$,
\[
\bP^{\otimes 2}_n\big( \bA_n^c \big)
\leq \bP^{\otimes 2}_n\Big( \gb_n^{2r} \sum_{a=1,2} \sum_{i=1}^{nt_a} \ind_{\{i\in \rho^{(a)}\}} \geq  c k \gb_n^{-1/4} \Big)
\leq \exp\Big( -c_{\delta} (\gb_n^{-1/4} k)^{\frac{1}{1- (2\alpha-1) +\delta}} \Big) \,,
\]
where the last inequality comes from Lemma~\ref{lem:tailandmoments}; note that the conditioning in $\bP^{\otimes 2}_n$ can be removed by using \cite[Lem.~A.2]{GLT10}.
Since the power verifies $\frac{1}{1- (2\alpha-1) +\delta} >1$, this is clearly bounded by $\exp( -c_{\delta} \gb_n^{-1/4} k )\leq \gb_n^{k}$, at least for $n$ large enough.

For the second term, we use that
\[
\bP_{\blamb} \big( \tilde Q_3 \geq 2\eta k\big)
\leq e^{ -2\eta k a_n} \bE_{\blamb}\Big[ e^{a_n \tilde Q_3} \Big] = e^{ -2\eta k a_n} \exp\Big( \sum_{p\geq 3} \blamb_p (e^{pa_n} -1) \Big)
\]
where for the last identity we recalled the definition $\tilde Q_3 = \sum_{p\geq 3} pQ_p$, with $Q_p\sim \text{Poisson}(\blamb_p)$.
Hence, on the event $\bA_n$, we get that  the sum in the last exponential is bounded by $k$: we obtain
\[
\bE^{\otimes 2}_n\Big[
\bP_{\blamb} \big( \tilde Q_3 \geq 2\eta k\big)^{1+\gep}  \ind_{\bA_n}\Big]^{\frac{1}{1+\gep}}
\leq e^{ -k(2\eta a_n-1) } \leq e^{-k \eta a_n}  = (\gb_n)^{ \frac{\eta k}{4}}\,,
\]
where the last inequality is valid for $n$ large enough, using that $a_n\to\infty$; the last identity follows recalling that $e^{a_n} =(\beta_n)^{1/4}$, by definition of $a_n$.

All together, we have obtained that 
\begin{equation}
\label{vgkCase3}
\bE^{\otimes 2}_n\bigg[
\bP_{\blamb} \Big( 2Q_1+2Q_2+ \tilde Q_3 = 2k, \tilde Q_3 \geq 2\eta k \Big)^{1+\gep} \bigg]^{\frac{1}{1+\gep}} 
\leq \gb_n^k+(\gb_n)^{ \frac{\eta k}{4}} \,.
\end{equation}

\medskip
\noindent
{\it Conclusion.}
We now simply need to collect~\eqref{vgkCase1}-\eqref{vgkCase2}-\eqref{vgkCase3} to conclude the proof of~\eqref{vgk25} and hence of Lemma~\ref{lemma:L2rest}; with the constant $c= \frac{\eta}{4} = \frac18 (\alpha-\frac12)$.
\end{proof}

\subsubsection{Proof of item (ii) of Proposition~\ref{lemma:tronck>1}}

Similarly to~\eqref{eq:intdiscrete:k>1:bis} and~\eqref{L2start1},
we can write
\[
\Big( \frac{\sigma_r \gb_n^r}{n^{\frac12 -\alpha} L(n)}  \Big)^{k} \, \big( \psi_{n}^{\delta} \overset{k}\cdot \ol M_n - \psi_{n} \overset{k}\cdot \ol M_n \big)
 = n^{2-\alpha} L(n) \sum_{\bI\in \cI_k } \big( u_n^{\delta}(\bI) -u_n(\bI) \big) \prod_{\bi\in \bI}  \zeta_{\bi} \,,
\]
where for  $\bI=(\bi_1, \ldots, \bi_k)$ increasing, we have set $u_n(\bI) := \prod_{l=1}^{k+1} u(\bi_l-\bi_{l-1})$ 
and $u_n^\delta(\bI) := \prod_{l=1}^{k+1} u^\delta(\bi_l-\bi_{l-1})$ with $u^\delta( \bi) := u(\bi) \ind_{\{\|\bi\| \geq \delta n\}}$ (recall that by convention $\bi_0 =\bzero$ and $\bi_{k+1} = n \bt$).
Using that $\sigma_r \gb_n^r \geq c n^{\frac12-\alpha} L(n)$, we simply need to bound the $L^2$ norm of the right-hand side, which is equal to
\[
(n^{2-\alpha} L(n))^2
\sum_{\bI,\bJ  \in \cI_k }  \big( u_n^{\delta}(\bI) -u_n(\bI) \big) \big( u_n^{\delta}(\bJ) -u_n(\bJ) \big)   \bbE\Big[  \prod_{\bi\in \bI, \bj \in \bJ}  \zeta_{\bi} \zeta_{\bj} \Big] \,,
\]

Note that if all indices in $\bI$ are such that $\|\bi_l-\bi_{l-1} \| \geq \delta n$, then $u_n^{\delta}(\bI) = u_n(\bI)$, and similarly for $\bJ$; if one has $\|\bi_l-\bi_{l-1} \| < \delta n$ for one $l$, then we have $u_n^{\delta}(\bI)=0$.
Hence, we get that 
\[
\|\psi_{n} \overset{k}\cdot \ol M_n - \psi_{n}^\delta \overset{k}\cdot \ol M_n\|_{L^2}^2 \leq C^k (n^{2-\alpha} L(n))^2
\sum_{\bI,\bJ  \in \cI_{k}^{\delta}  }  u_n(\bI) u_n(\bJ)   \bbE\Big[  \prod_{\bi\in \bI, \bj \in \bJ}  \zeta_{\bi} \zeta_{\bj} \Big] \,,
\]
where we have defined
$\cI_{k}^{\delta}= \{ \bI  \in \cI_k, \, \exists 1\leq l\leq k+1 , \|\bi_l-\bi_{l-1} \| <\delta n  \}$.
Then, as in the proof of item (i) of Proposition~\ref{lemma:tronck>1}, we can decompose over the structure of $\bI\cup\bJ$ (see~\eqref{eq:decompset} and~\eqref{correl:structure}): we obtain that
\[
\|\psi_{n} \overset{k}\cdot \ol M_n - \psi_{n}^\delta \overset{k}\cdot \ol M_n\|_{L^2}^2  \leq C^k \big( \cK_1^{\delta} +\cK_2^{\delta} \big)\,,
\]
where $\cK_1^{\delta}, \cK_2^\delta$ are defined exactly as $\cK_1,\cK_2$, see~\eqref{L2main}-\eqref{L2rest}, with the sum restricted to $\bI,\bJ\in \cI_k^{\delta}$ instead of~$\cI_k$.
Now, we can bound $\cK_2^{\delta} \leq \cK_2$, and directly use Lemma~\ref{lemma:L2rest} to deal with this term. Therefore,
\[
\limsup_{n\to\infty} \|\psi_{n} \overset{k}\cdot \ol M_n - \psi_{n}^\delta \overset{k}\cdot \ol M_n\|_{L^2}^2 \leq C^k \limsup_{n\to\infty} \cK_1^{\delta} \,, 
\]
and it remains to deal with $\cK_1^{\delta}$: as in~\eqref{eq:rewritingreplicas}, we can write
\begin{equation}
\label{cK1delta}
\cK_1^{\delta} \leq  C (\gb_n^{2r})^{k} \bE^{\otimes 2}_n \Big[ \big| \big\{ (\bI, \bJ)\in \cI_k^{\delta} \times \cI_k^{\delta}\,,N_2(\bI,\bJ)=k,  \bI\subset \btau, \bJ\subset \btau' \big\}\big| \Big] \,,
\end{equation}
where we recall that $N_2(\bI,\bJ) =k$ means that  $\bI\cup \bJ$ can be written as a union of $k$ disjoint pairs of aligned indices.

Let us denote $\cA(\btau,\btau') = \{ (\bi,\bj) \in \llbracket \bone,n\bt -\bone\rrbracket^2 , \bi\in \btau,\bj\in \btau', \bi\leftrightarrow \bj \}$ the set of pairs of aligned points in $\btau \cup \btau'$;
note that on the event $\{N_2(\bI,\bJ)=k\}$, indices $(\bi,\bj) \in \bI\times \bJ$ form $k$ distinct pairs in $\cA(\btau,\btau')$.
Since $\bI\in \cI_k^{\delta}$, there must be some index $l\in\{1,\ldots, k+1\}$ such that $\|\bi_l-\bi_{l-1}\| \leq \delta n$: decomposing according to whether $l=1$, $l=k+1$ or $2\leq l \leq k$, we get (using some symmetry)
\[
 \big| \big\{ (\bI, \bJ)\in \cI_k^\gd\times \cI_k^{\gd}\,,N_2(\bI,\bJ)=k,  \bI\subset \btau, \bJ\subset \btau' \big\}\big| 
\leq  2  \cC_2^{\delta}(\btau,\btau') \; \binom{\cC_2(\btau,\btau')}{k-1} + \; \tilde C_2^{\delta} (\btau,\btau') \; \binom{\cC_2(\btau,\btau')}{k-2} \,,
\]
where we recall that $\cC_2(\btau,\btau') = |\cA(\btau,\btau')|$ and we defined 
$\cC_2^{\delta} (\btau,\btau') = |\{ (\bi,\bj) \in \cA(\btau,\btau'), \|\bi\|\leq \delta n\}|$ and
\[
\tilde \cC_{2}^{\delta} (\btau,\btau') = \big| \big\{   (\bi,\bj) , (\bi',\bj') \in \cA(\btau,\btau')  ,  \|\bi'-\bi\| \leq  \delta n \big\} \big| \,.
\]
In other words, $\tilde \cC_2^{\delta}$ counts how many couples of aligned pairs of points in $\btau \cup \btau'$  have indices $\bi$ at distance smaller than $\delta n$.
 Using that $\binom b a \leq \frac1{a!}b^a$ and recalling~\eqref{cK1delta}, we therefore need to control (we only need to control the first term if $k=1$)
\[
\begin{split}
\bE^{\otimes 2}_n\Big[  \gb_n^{2r}\cC_2^{\delta}(\btau,\btau') \big( \gb_n^{2r} \cC_2(\btau,\btau') \big)^{k-1} \Big]
&\leq \bE^{\otimes 2}_n\Big[ \big(\gb_n^{2r} \cC_2^{\delta}(\btau,\btau') \big)^2\Big]^{1/2} \bE^{\otimes 2}_n \Big[ \big(\gb_n^{2r} \cC_2(\btau,\btau') \big)^{2(k-1)} \Big]^{1/2} \,;\\
\bE^{\otimes 2}_n\Big[  \gb_n^{4r}\tilde \cC_2^{\delta}(\btau,\btau') \big( \gb_n^{2r} \cC_2(\btau,\btau') \big)^{k-2} \Big]
&\leq \bE^{\otimes 2}_n\Big[ \big(\gb_n^{4r} \tilde \cC_2^{\delta}(\btau,\btau') \big)^2\Big]^{1/2} \bE^{\otimes 2}_n \Big[ \big(\gb_n^{2r} \cC_2(\btau,\btau') \big)^{2(k-2)} \Big]^{1/2}\,.
\end{split}
\]

In both cases, the second term is bounded by a constant (which depends on $k$), see Remark~\ref{rem:finiteexpmoment}. On the first line, for the first term one easily gets that
$\cC_2^{\delta}(\btau,\btau')\leq |\rho^{(1)} \cap [0,\delta n]| +|\rho^{(2)} \cap [0,\delta n]|$.
It is then straightforward to get that 
\[
\bE^{\otimes 2}_n\Big[  \big(\gb_n^{2r}\cC_2^{\delta}(\btau,\btau') \big)^2\Big]
\leq C \bE^{\otimes 2}\bigg[ \Big( \gb_n^{2r} \sum_{i=1}^{\delta n} \ind_{\{i\in \rho^{(1)}\}} \Big)^2\bigg]
\leq C' \Big( \gb_n^{2r}  L(n)^{-2} (\delta n)^{2\alpha-1}\Big)^2 \leq C \delta^{2(2\alpha-1)} \,,
\]
using first~\cite[Lem.~A.2]{GLT10} to remove the conditioning, then expanding the square and using that $\bP(i\in \rho^{(1)}) =\bP(i\in \tau^{(1)})^2 \sim c L(i)^{-2} i^{2\alpha-2}$ as $i\to \infty$. Details are left to the reader.

It only remains to estimate
\[
\begin{split}
& \bE^{\otimes 2}_n  \Big[ \tilde \cC_2^{\delta}(\btau,\btau') ^2\Big]
 = \bE^{\otimes 2}_n\bigg[\bigg(\sumtwo{ \bzero \prec \bi\prec \bi' \prec n\bt }{ \bj,\bj' \in \llbracket \bzero, n \bt \rrbracket } \ind_{\{\|\bi'-\bi\|\leq \delta n\}}\ind_{\{ (\bi,\bj)\in \cA (\btau,\btau') \}}\ind_{\{ (\bi',\bj')\in \cA (\btau,\btau') \}} \bigg)^2\bigg] \\
&\quad \leq \frac{1}{\bP(n \bt \in \btau)^2}  \sumtwo{ \bzero \prec \bi_1\preceq \bi_2 \preceq \bi_3\preceq \bi_4 \preceq n\bt}{\|\bi_2-\bi_1\|\leq \delta n} 
\sumtwo{ \bzero \prec \bj_1\preceq \bj_2 \preceq \bj_3\preceq \bj_4 \prec n\bt}{ \exists \, \sigma\in \mathfrak{S}_4, \, \bj_l \aligne \bi_{\sigma(l)} \text{ for } 1\leq l\leq4 } \bP\big( \bi_1,\bi_2,\bi_3,\bi_4, n\bt \in \btau\big) \bP\big( \bj_1,\bj_2,\bj_3,\bj_4,n\bt\in \btau'\big) \,.
\end{split}
\]
Note that the indices must be in non-decreasing order, otherwise we cannot have $\bi_1,\bi_2,\bi_3,\bi_4, n\bt \in \btau$. Then, we can use the following uniform bound from \cite[Thm.~4.1, see also Eq.~(4.2)]{B18}: there exists $c>0$ such that
\begin{equation*}
\bP(\bi \in \bt) \, \leq \,c\,L(\|\bi\|)^{-1}\,\|\bi\|^{\ga-2} \;,\qquad \forall\; \bi\in\N^2\;.
\end{equation*}
Note that Lemma~\ref{lem:homogene} is the analogous of this inequality in the case $h_n \not \equiv 0$.
Using this, if all indices are distinct, we bound $\bP( \bi_1,\bi_2,\bi_3,\bi_4, n\bt \in \btau')$ by a constant times $\prod_{l=1}^{5} L(\|\bi_l-\bi_{l-1}\|)^{-1} \|\bi_l-\bi_{l-1}\|^{\alpha-2}$
(with the convention $\bi_0=\bj_0=\bzero$, $\bi_5=\bj_5:=n\bt$), and similarly for $\bP( \bj_1,\bj_2,\bj_3,\bj_4, n\bt \in \btau')$: by a Riemann sum approximation, we have that for any $\sigma\in \mathfrak{S}_4$
\[
\begin{split}
  \sumthree{ \bzero \prec \bi_1\prec \bi_2 \prec \bi_3\prec \bi_4 \prec n\bt}{ \bzero \prec \bj_1\prec \bj_2 \prec \bj_3\prec \bj_4 \prec n\bt}{\bj_l \aligne \bi_{\sigma(l)} \text{ for } 1\leq l\leq4 } 
&  \!\! \ind_{\{{\|\bi_2-\bi_1\|\leq \delta n}\}} \bP\big( \bi_1,\bi_2,\bi_3,\bi_4, n\bt \in \btau'\big) \bP\big( \bj_1,\bj_2,\bj_3,\bj_4,n\bt\in \btau'\big) \leq C \big(  L(n)^{-1} n^{\alpha-2}  \big)^{10}  (n^{3})^4  \\[-\baselineskip]
 & \quad     \times \int_{[\bzero, \bt]^4} \bigg( \int_{\cA_{\bu_1}\times \cdots \times \cA_{\bu_4}}   g(\bv_1,\ldots, \bv_4)  \prod_{l=1}^4 \lambda_{\bu_{\sigma(l)}}(\dd \bv_l) \bigg)    g(\bu_1,\ldots, \bu_4)  \ind_{\{\|\bu_2-\bu_1\|\leq \delta\}}  \prod_{l=1}^4\dd \bu_l \,,
 \end{split}
\]
where $g(\bs_1,\ldots,\bs_4) = \prod_{l=1}^5 \|\bs_l-\bs_{l-1}\|^{\alpha-2}$
with the convention $\bs_0=\bzero$ and $\bs_{5}=\bt$.
Here, we have used the notation of Proposition~\ref{propbis:covariancemeasure:k>1}: $\cA_{\bu}$ denotes the points in~$\bbR_+^2$ aligned with $\bu$ and $\lambda_{\bu}$ is the (one-dimensional) Lebesgue measure on $\cA_{\bu}$. 
In particular, the integral in the right-hand side is bounded by 
\[
\int_{\bbR_+^{16}}  g(\bu_1,\ldots, \bu_4) g(\bv_1,\ldots, \bv_4) \dd \nu_{\cM^{\otimes 8}} (\bu_1,\ldots, \bu_4, \bv_1,\ldots, \bv_4) = \|g\|_{\nu_{\cM}^{(4)}}^2,
\]
which is known to be finite by Proposition~\ref{prop:ub:psikk}. The case where some of the indices $\bi_a$ (or $\bj_b$) are equal in the sum is treated exactly in the same manner: there are simply fewer terms and the sum is smaller (details are left to the reader).

Recalling that $\gb_n^{2r} \sim c L(n)^{2} n^{1-2\alpha}$ and that $\bP(n\bt \in \btau) \sim \gp(\bt) L(n)^{-2} n^{\alpha-2}$, we therefore end up with 
\[
\bE^{\otimes 2}_n  \big[ \big( \gb_n^{4r} \tilde \cC_2^{\delta}(\btau,\btau') \big)^2\big] \leq C I_{\delta} ,
\]
where
\[
 I_{\delta} := \int_{\bbR_+^{16}}  \ind_{\{\|\bu_2-\bu_1\|\leq \delta\}}   g(\bu_1,\ldots, \bu_4) g(\bv_1,\ldots, \bv_4) \dd \nu_{\cM^{\otimes 8}} (\bu_1,\ldots, \bu_4, \bv_1,\ldots, \bv_4) \,.
\]

We therefore end up with $\limsup_{n\to\infty} \cK_1^{\delta} \leq C_{k,\bt} (\delta^{2(2\alpha-1)} +I_{\delta})$.
Since $\lim_{\delta\downarrow 0}I_{\delta} =0$ by dominated convergence, this concludes the proof.
\qed

\section{Homogeneous and degenerate disordered case: proof of Propositions~\ref{prop:scalinghom} and~\ref{prop:degenerate}}\label{sec:homogeneous}
\subsection{Homogeneous case: proof of Proposition~\ref{prop:scalinghom}}

Let us prove the result, which essentially comes from Riemann-sum convergence.
We focus here on the constrained partition function; the free case is identical.
First of all, we expand the partition function as
\begin{equation}
\label{eq:expandhomogeneous}
\begin{split}
Z_{ n\bt, h_n } & = \bE \Big[ \Big(\prod_{\bi \in \llbracket 1, n \bt \rrbracket } (1+ (e^{h_n}-1) \ind_{\{\bi\in \btau\}} ) \Big) \ind_{\{ n\bt \in \btau \}} \Big] \\
& = e^{h_n}  \sum_{k=0}^{ (nt_1) \wedge (nt_2)} (e^{h_n}-1)^k   \sum_{ \bzero=\bi_0 \prec \bi_1 \prec \ldots \prec \bi_k \prec \bi_{k+1} = n\bt }   \prod_{l=1}^{k+1} u(\bi_l-\bi_{l-1}) \,.
 \end{split}
\end{equation}
Since $h_n \sim \hat h  L(n) n^{-\alpha}$ by assumption, we have that for any fixed $\gep>0$, for $n$ large enough,
$(1-\gep) \hat h   \leq \frac{e^{h_n}-1}{L(n) n^{-\alpha}} \leq (1-\gep) \hat h $.
We now define for $ \check h \in \mathbb R$,
\[
\check Z_{n\bt, \check h} = \sum_{k=0}^{\infty} \check h^k  \frac{1}{n^{2k}} \sum_{ \bzero=\bi_0 \prec \bi_1 \prec \ldots \prec \bi_k \prec  \bi_{k+1} = n\bt }   \prod_{l=1}^{k+1} \Big( L(n) n^{2-\alpha}  u(\bi_l-\bi_{l-1}) \Big) \,,
\]
so that  $ \check Z_{n\bt, (1-\gep) \hat h} \leq   L(n) n^{2-\alpha} Z_{ n\bt, h_n }  \leq  \check Z_{n\bt, (1+\gep) \hat h} $ for large enough $n$.
We show below that
\begin{equation}
\label{eq:convergencehomogeneous}
\lim_{n\to\infty}  \check Z_{n\bt, \check h} = \bZ_{\bt, \check h} = \sum_{k=0}^{+\infty}  \check h^k 
\idotsint \limits_{ \bzero \prec \bs_1 \prec \cdots \prec \bs_k \prec \bt }  \psi_{\bt} \big( \bs_1,\ldots, \bs_k  \big)  \dd \bs_1 \cdots \dd \bs_k 
\end{equation}
and that $\check h \mapsto \bZ_{\bt ,\check h}$ is continuous. Together with the above bounds, this will conclude the proof.

For each $k$, the convergence of the $k$-th term in $ \check Z_{n\bt, \check h}$ to the $k$-th term of $\bZ_{\bt, \check h}$ is a simple consequence of Riemann-sum convergence, since $L(n) n^{2-\alpha}  u( \lfloor n \bu \rfloor -\lfloor n\bv \rfloor)$ converges to $\gp(\bi_l -\bi_{l-1})$, see Proposition~\ref{thm:renouv}---the convergence is actually uniform on compacts. One can also use the uniform bound
$L(n) n^{2-\alpha}  u( \bi) \leq  C \| \frac1n \bi  \|_1^{\alpha-2}$,
which comes from~\cite[Thm.~4.1]{B18}, to bound all sums uniformly:  there exists $C>0$ such that for all $k \geq 1$
\[
\begin{split}
\frac{1}{n^{2k}} & \sum_{ \bzero=\bi_0 \prec \bi_1 \prec \ldots \prec \bi_k \prec  \bi_{k+1} = n\bt }   \prod_{l=1}^{k+1} \Big( L(n) n^{2-\alpha}  u(\bi_l-\bi_{l-1}) \Big) \\
& \leq C^k \idotsint \limits_{ \bzero = \bs_0 \prec \bs_1 \prec \cdots \prec \bs_k \prec \bs_{k+1} = \bt }    \prod_{i=1}^{k+1} \| \bs_i - \bs_{i-1}\|^{\alpha-2}  \dd \bs_1 \cdots \dd \bs_k  \\
& \leq (C')^k  \idotsint \limits_{  0=t_0 <t_1 <\cdots < t_k <t_{k+1}= \|\bt\|_1 }    \prod_{i=1}^{k+1} ( t_i - t_{i-1})^{\alpha-1}  \dd t_1 \cdots \dd t_k = \frac{ (C' \Gamma(\alpha))^k}{\Gamma(k\alpha)} \|\bt\|_1^{(k+1)\alpha-1} \,,
\end{split}
\]
where we have used a standard calculation for the last identity (see e.g.\ \cite[Lem.~A.3]{BL21_scaling}).
Therefore, this shows that 
for any $k_0 \geq 1$
\[
\sum_{k\geq k_0} \check h^k  \frac{1}{n^{2k}} \sum_{ \bzero=\bi_0 \prec \bi_1 \prec \ldots \prec \bi_k \prec  \bi_{k+1} = n\bt }   \prod_{l=1}^{k+1} \Big( L(n) n^{2-\alpha}  u(\bi_l-\bi_{l-1}) \Big) \leq  \sum_{k\geq k_0} \frac{ (C' \Gamma(\alpha) \check h)^k}{\Gamma(k\alpha)} \|\bt\|_1^{(k+1)\alpha-1} \,,
\]
which can be made arbitrarily small by taking $k_0$ large, uniformly for $\check h$ in a bounded interval.
This concludes the proof of~\eqref{eq:convergencehomogeneous} and shows that the convergence is uniform on compacts. Hence, this also shows that the limit $ \bZ_{\bt ,\check h}$ is continuous in $\check h$.
\qed

\smallskip
Notice that Proposition~\ref{prop:scalinghom} shows that $\lim_{n\to\infty} n^{2-\alpha} L(n) Z_{ n\bu, n\bv, h_n } = \bZ_{\bu-\bv,\hat h}$, where we have set
\[
Z_{ \ba, \bb, h_n } = \bE\bigg[  \exp\Big( h_n \sum_{\bi \in \llbracket \ba+\bone,\bb \rrbracket} \ind_{\{\bi\in \btau \}} \Big)  \ind_{\{\bb \in \btau \}}\, \bigg| \, \ba \in \btau \bigg] \,.
\]
Let us state a lemma which will be useful in the following: it can be found in~\cite[Lem.~5.2]{Leg21}
\begin{lemma}
\label{lem:homogene}
If $h_n \sim \hat h  L(n) n^{-\alpha}$ for some $\hat h \in\mathbb R$, there exists a constant $C=C_{\hat h,\bt}$ such that for any $n\in\N$, $\bi \in \llbracket \bone, n \bt \rrbracket$,
we have
\[
Z_{\bi, h_n} \leq C L(\|\bi\|_1)^{-1} \|\bi\|_1^{\alpha-2} \,.
\]
As a by-product, this proves that
$\bZ_{\bs,\hat h} \leq C \|\bs\|_1^{\alpha-2}$ for all $ 0\prec \bs \prec \bt$.
\end{lemma}

\subsection{Degenerate disordered case: proof of Proposition~\ref{prop:degenerate}}
\label{sec:secondmomentbound}

Here, we focus on the free partition function.
First of all, let us notice that we can write
\[
Z_{n\bt, h_n}^{\gb_n , \free} 
 = Z_{n\bt, h_n}^{\free} \bE_{h_n} \Big[  \exp\Big( \sum_{\bi \in \llbracket \bone, n \bt \rrbracket} (\gb_n \go_{\bi} - \lambda(\gb_n)) \ind_{\{\bi \in \btau\}} \Big) \Big] \,,
\]
where we have used the short-hand notation $\bP_{h_n} = \bP_{n\bt,h_n}^{\hat \gb =0,\free}$.
We have seen in Proposition~\ref{prop:scalinghom} that $Z_{n\bt, h_n}^{\free}$ converges to $\bZ_{\bt,\hat h}^{\free}$. We therefore simply need to prove that the second term above converges to $1$ in $L^2(\mathbb P)$, which is the purpose of the following lemma.

\begin{lemma}
\label{lem:scalinghom}
Assume that $\alpha \in (0,\frac12)$ or that $\alpha \in (0,1)$ and $\mathbb P \in \mathfrak P_{\infty}$.
Then if $h_n \sim \hat h  L(n) n^{-\alpha}$, for any vanishing sequence $(\gb_n)_{n\geq 1}$ we have
\[
\lim_{n\to\infty} \bE_{h_n} \Big[  \exp\Big( \sum_{i \in \llbracket \bone, n \bt \rrbracket} (\gb_n \go_{\bi} - \lambda(\gb_n)) \ind_{\{\bi \in \btau\}} \Big) \Big] =1 \quad \text{ in } L^2(\mathbb P) \,.
\]
\end{lemma}

\begin{proof}
We focus on the proof in the case $h_n\equiv 0$, that is when $\bP_{h_n} = \bP$.
Let 
\[
Z_{n\bt, \gb_n}^{\go}
= \bE \Big[  \exp\Big( \sum_{\bi \in \llbracket \bone, n \bt \rrbracket} (\gb_n \go_{\bi} - \lambda(\gb_n)) \ind_{\{\bi \in \btau\}} \Big) \Big] \,.
\]
Since $\bbE[Z_{n\bt, \gb_n}^{\go}] =1$, we simply need to show that $\lim_{n\to\infty} \bbE[(Z_{n\bt, \gb_n}^{\go})^2]=1$.

\medskip
\noindent
{\it Case $\alpha\in (0,\frac12)$.} In that case, one can use \cite[Prop.~3.3]{Leg21} (whose proof uses only that $\go_{\bi}$ is correlated via horizontal and vertical lines, but not the specific definition of $\go_{\bi}$): it gives that 
\begin{equation}
\label{upperboundsecondmoment1}
1\leq \bbE[(Z_{n\bt, \gb_n}^{\go})^2] \leq \bE^{\otimes 2} \Big[   e^{  \frac32 (\lambda(2\gb_n)  - 2\lambda(\gb_n))  \big( |\btau^{(1)}\cap \btau'^{(1)}|  + | \btau^{(2)}\cap \btau'^{(2)}| \big) }\Big] \,,
\end{equation}
where $\btau,\btau'$ are two independent bivariate renewals with the same distribution, and $\btau^{(i)},\btau'^{(i)}$ are their projections on the $i$-th coordinate.
Notice that $\btau^{(i)},\btau'^{(i)}$ are two independent recurrent renewal processes, with inter-arrival distribution verifying $\bP(\btau^{(i)}=n) \sim c_{\alpha} L(n) n^{-(1+\alpha)}$ as $n\to\infty$, so we get $\bP(n\in \btau^{(i)}) \sim c'_{\alpha} L(n)^{-1} n^{1-\alpha}$ thanks to~\cite{Doney97}. Hence, $\btau^{(i)}\cap \btau'^{(i)}$ is a renewal process which is terminating if $\alpha\in (0,\frac12)$, 
because $\sum_{n=1}^{\infty}\bP(n \in \btau^{(i)}\cap \btau'^{(i)}) =\sum_{n=1}^{\infty} \bP(n\in \btau^{(i)})^2 <+\infty$.
We therefore have that
$|\btau^{(1)}\cap \btau'^{(1)}|$, $|\btau^{(2)}\cap \btau'^{(2)}|$ are two (correlated)
geometric random variable, so the upper bound in~\eqref{upperboundsecondmoment1} goes to $0$ as $\gb_n \to 0$ (for instance using Cauchy--Schwarz inequality to treat the two geometric random variables separately).

\medskip
\noindent
{\it Case $\alpha\in (0,1)$, $\bbP \in \mathfrak P_{\infty}$.}
In that case, we can compute exactly the second moment of the partition function.
Writing $\zeta_{\bi} := e^{\gb_n\go_\bi -\lambda(\gb_n)}-1$ and expanding the product as in~\eqref{expansion}, we get
\[
Z_{n\bt, \gb_n}^{\go} = 1+\sum_{k=1}^{(nt_1)\wedge (n t_2)} \sum_{ \bzero=\bi_0 \prec \bi_1 \prec \ldots \prec \bi_k \preceq  n\bt } \prod_{l=1}^k \zeta_{\bi_l} \,u(\bi_l-\bi_{l-1}).
\]
By Lemma~\ref{lem:correl}, if $\bbP \in \mathfrak{P}_{\infty}$ then we have $\bbE[\zeta_{\bi} \mid (\zeta_{\bj})_{\bj\neq \bi}] =0$.
We therefore get that, for any $\bi_1 \prec \cdots \prec \bi_{k}$ and $\bi'_1 \prec \cdots \prec \bi'_{k'}$,
\[
\bbE\Big[ \prod_{l=1}^k \zeta_{\bi_l} \prod_{l=1}^{k'} \zeta_{\bi'_l}\Big] =
\begin{cases}
 0& \qquad \text{ if } \bi_l\neq \bi'_l \text{ for some }l \\
 \bbE[\zeta_{\bone}^2]^k & \qquad \text{ if } k=k'\ \text{ and } \bi_l= \bi'_l \text{ for all }l \,,
\end{cases}
\]
where the second line comes from the fact that $(\zeta_{i_l}^2)_{1\leq l \leq k}$ are independent, because $\bi_1 \prec \cdots \prec \bi_{k}$.
Since $\bbE[\zeta_{\bone}^2] = e^{\lambda(2\gb) -2 \lambda(\gb)}-1$, we therefore end up with
\begin{equation}
\label{upperboundsecondmoment2}
\begin{split}
\bbE[ (Z_{n\bt, \gb_n}^{\gb,\go})^2] 
& = 1+ \sum_{k=1}^{(nt_1)\wedge (n t_2)} \sum_{ \bzero=\bi_0 \prec \bi_1 \prec \ldots \prec \bi_k \preceq  n\bt } \prod_{l=1}^k (e^{\lambda(2\gb_n) -2 \lambda(\gb_n)}-1) u(\bi_l-\bi_{l-1})^2 \\
& = \bE^{\otimes 2} \bigg[ \exp\bigg( (\lambda(2\gb_n) -2 \lambda(\gb_n)) \sum_{\bi\in \llbracket \bone, n \bt\rrbracket} \ind_{\{\bi \in \btau\cap \btau'\}}\bigg) \bigg]  \leq\bE^{\otimes 2} \Big[ e^{ (\lambda(2\gb_n) -2 \lambda(\gb_n)) | \btau\cap \btau'|}  \Big] \,,
\end{split}
\end{equation}
where again $\btau,\btau'$ are two independent bivariate renewals with the same distribution.
From \cite[Prop~A.3]{BGK20}, $\btau\cap \btau'$ is terminating when $\alpha <1$, so $|\btau \cap \btau'|$ is a geometric random variable, and the upper bound in~\eqref{upperboundsecondmoment2} goes to $1$ as $\gb_n \to 0$.

\smallskip
The case of a general sequence $(h_n)_{n\geq 1}$ satisfying $h_n \sim \hat h  L(n) n^{-\alpha}$ can easily be adapted, using for instance that $\bP_{h_n}(\bi\in \btau) = Z_{\bi,h_n} Z_{n\bt-\bi ,h_n}^{\free}$, together with Proposition~\ref{prop:scalinghom} (and the help of Lemma~\ref{lem:homogene})---or analogous results for the one-dimensional pinning model in the case $\alpha \in (\frac12,1)$.
\end{proof}

\begin{remark}[Proof of Corollary~\ref{cor:boundcriticalpoint2}]
\label{rem:boundedL2}
Let us stress that the bounds~\eqref{upperboundsecondmoment1}-\eqref{upperboundsecondmoment2} provide uniform bounds on the  second moment $\bbE[ (Z_{n\bt, \gb}^{\gb,\free})^2]$, also for a non-vanishing $\gb>0$.
We therefore get that if $\gb>0$ is fixed small enough, the upper bounds~\eqref{upperboundsecondmoment1}-\eqref{upperboundsecondmoment2} are finite, so that $Z_{n\bt, \gb}^{\gb,\free}$ is bounded in $L^2(\bbP)$.
Applying Proposition~\ref{prop:secondmomentcriticalpoint}, this gives Corollary~\ref{cor:boundcriticalpoint2}.
\end{remark}

\begin{appendix}

\section{Technical results on renewal processes}\label{app:renewal}

We give in this section some technical estimates on the intersection of two independent copies $\btau,\btau'$ of a bivariate renewal satisfying~\eqref{def:tau}.
We start with a lemma that gives estimates on $\btau^{(1)} \cap \btau'^{(1)}$, the intersection of the projection of $\btau,\btau'$.

\smallskip
Let $\tau = (\tau_i)_{i\geq 1}$ be a recurrent \emph{one-dimensional} renewal process on $\bbN$ starting from $\tau_0=0$ and inter-arrival distribution verifying $\bP(\tau_1 >n) \sim \ell(n) n^{-\gamma}$ as $n\to\infty$, for some $\gamma \in (0,1)$ and some slowly varying function~$\ell(\cdot)$.
Define 
\[
U_n := \sum_{i=1}^n \bP(i\in \tau) \,.
\]
Then, $\bP(\tau_1 >n) \sim \ell(n) n^{-\gamma}$ is equivalent to the fact that $U_n \sim c_{\gamma} n^{\gamma} \ell(n)^{-1}$ with $c_{\gamma}=\frac{\sin(\pi \gamma)}{\pi \gamma}$, see \cite[Thm.~8.7.3]{BGT87}.

\begin{remark}
\label{rem:appliclemmatail}
If $\btau,\btau' $ are two independent copies of a bivariate renewal satisfying~\eqref{def:tau},
then $\tau = \btau^{(1)} \cap \btau'^{(1)}$ is a one-dimensional renewal process: if $\alpha\in (\frac12,1)$, then $\tau$ is recurrent and verifies the above tail assumption, with $\gamma=2\alpha-1 \in (0,1)$ and $\ell(n) =c_{\alpha} L(n)^2$ for some explicit constant $c_{\alpha}$.
Indeed, thanks to~\eqref{def:tau}, we have $\bP(\btau^{(1)}=n) \sim c L(n) n^{-(1+\alpha)}$ as $n\to\infty$, so Doney's result~\cite{Doney97} gives that $\bP(n\in \btau^{(1)}) \sim c_{\alpha} n^{\alpha-1} L(n)^{-1}$. Then, if $\alpha\in (\frac12,1)$,
\[
U_n = \sum_{i=1}^n \bP ( i \in \btau^{(1)}\cap \btau'^{(1)}) = \sum_{i=1}^n \bP ( i \in \btau^{(1)})^2 \sim c_{\gamma} n^{2\alpha-1} L(n)^{-2} \,
\]
and one concludes thanks to \cite[Thm.~8.7.3]{BGT87}.
\end{remark}

The following large deviation estimate is standard but we include it here since most of the literature treats more general cases (or with less optimal bounds, as in~\cite[Lem.~A.3]{AB18}).

\begin{lemma}
\label{lem:tailandmoments}
For any $\delta \in (0,\gamma)$ there is a constant $c_{\delta}$ such that for any $t\geq 1$ and any $n\geq 1$
\[
\bP\Big(  \sum_{i=1}^n \ind_{\{i\in \tau\}} \geq t U_n \Big) \leq  \exp\Big( - c_{\delta}\, t^{  \frac{1}{1-\gamma+\delta} } \Big) \,.
\]
As a consequence, for any $\delta \in (0,\gamma)$, there is a constant $C_\delta>0$ such that for any $k\geq 1$
\[
\bE\Big[ \Big( \frac{1}{U_n} \sum_{i=1}^n \ind_{\{i\in \tau\}}\Big)^k \Big] \leq   (C_{\delta})^k \Gamma\big( k(1-\gamma +\delta) \big) \,.
\]
\end{lemma}

\begin{proof}
Denote $t_n := \lceil t  U_n \rceil$, so that, for any $\lambda\in (0,1)$,
\[
\bP\Big(  \sum_{i=1}^n \ind_{\{i\in \tau\}} \geq t U_n \Big) = \bP(\tau_{t_n} \leq n)  \leq e^{\lambda n} \bE[e^{-\lambda \tau_1}]^{t_n} \,.
\]
Then, one can use that there is a constant $c$ such that $\bE[e^{-\lambda \tau_1}] \leq  1- c\ell(1/\lambda) \lambda^{\gamma}$ for all $\lambda\in (0,1)$, by standard properties of the Laplace transform (see e.g.\ \cite[Thm.~1.7.1]{BGT87}).
Hence, using that  $U_n \sim c_{\gamma} n^{\gamma} \ell(n)^{-1}$, we get that 
\[
\bP\Big(  \sum_{i=1}^n \ind_{\{i\in \tau\}} \geq t U_n \Big)  \leq \exp\Big( \lambda n  - c' t n^{\gamma}  \ell(n)^{-1} \lambda^{\gamma} \ell(1/\lambda)  \Big) 
\leq  \exp\Big( \lambda n  - c'_{\delta} t (\lambda n)^{\gamma -\delta}   \Big) 
\,,
\]
where we have used Potter's bound (\cite[Thm.~1.5.6]{BGT87}), to get that for any $\delta \in (0,\gamma)$ there is a constant $c_{\delta}>0$ such that $\frac{\ell(1/\lambda)}{\ell(n)} \geq c_{\delta} (\lambda n)^{-\delta}$ for any $\lambda \geq 1/n$.
Optimizing over $\lambda$, we choose $\lambda = c_{\delta}'' t^{1/(1-\gamma+\delta)}/n$ (which is greater than $1/n$, at least for $t$ large). This completes the upper bound.

For the second term, using the first part of the result, we bound
\[
\begin{split}
\bE\Big[ \Big( \frac{1}{U_n} \sum_{i=1}^n \ind_{\{i\in \tau\}}\Big)^k \Big] = \int_0^{\infty} \bP\Big(  \sum_{i=1}^n \ind_{\{i\in \tau\}} \geq t^{1/k} U_n \Big) \dd t & \leq 1+ \int_0^{\infty}  e^{- c_{\delta} t^{\frac{1}{k(1-\gamma+\delta)} }}\dd t \\
& = 1+ (c_{\delta})^{k(1-\gamma+\delta)} k(1-\gamma-\delta)  \Gamma(k(1-\gamma+\delta))
\end{split}
\]
where we simply used a change of variable $t= (u/c_{\delta})^{k(1-\gamma +\delta)}$ for the last identity.
This concludes the proof, using also that $z\Gamma(z) =\Gamma(z+1)$.
\end{proof}

\begin{lemma}
\label{lem:interesctcond}
Let $\btau,\btau'$ be two independent copies of a bivariate renewal satisfying~\eqref{def:tau} with $\alpha\in (0,1)$.
For any $\bt \succ \bzero$, there exist constants $c,c_{\bt}>0$ such that  for any $k\geq 1$
\[
\bP^{\otimes 2} \Big( \big|\btau\cap \btau' \cap \llbracket \bzero , n \bt \rrbracket \big| > k  \,\Big|\, n\bt \in \btau\cap \btau'\Big) \leq c_{\bt}  \bP^{\otimes 2} \Big( \big| \btau\cap \btau' \cap \llbracket \bzero , n \bt \rrbracket \big| > k \Big) \leq c_{\bt} e^{-c k} \,.
\]
\end{lemma}

\begin{proof}
Let $\bT := \{ \bi \in \llbracket \bzero , n \bt \rrbracket, \|\bi\| \leq \|n\bt -\bi\|  \}$
be the set of points that are closer to $\bzero$ than $n\bt$. 
Since conditionally on $n\bt \in \btau$, the time-reversed process  $\tilde \btau$ in $\llbracket \bzero , n \bt \rrbracket \setminus \bT $ starting from
$n\bt$ has the same law in  as $\btau$ in $\bT$, we get by sub-additivity that
\[
\bP^{\otimes 2} \big( |\btau\cap \btau' \cap \llbracket \bzero , n \bt \rrbracket | > k  \mid n\bt \in \btau\cap \btau'\big)
\leq 2 \bP^{\otimes 2} \big( |\btau\cap \btau' \cap \bT | > \tfrac k2  \mid n\bt \in \btau\cap \btau'\big) \,.
\]
Therefore, it suffices to compute an upper bound for the r.h.s.\ above. 
Let $\bX := \sup\{ \bi\in \btau , \bi \in \bT\}$ and $\bX':=\sup\{\bi \in \btau' , \bi \in \bT\}$ be the up-right most point (\textit{i.e.}\ the sup is taken for the order $\preceq$) of $\btau$, resp.\ $\btau'$ in $\bT$. Then \cite[Lem.~A.1]{Leg21} (which is proven in the symmetric case $\bt =\bone$ but remains valid for any $\bt\succ \bzero$) proves that there exists a constant $C_\bt$ such that, for all $\bi \in \bT$,
\[
\bP(\bX= \bi \mid n\bt \in \btau ) \leq  \frac{C_{\bt}}{\bP(n\bt \in \btau)}   L(\|\bn\|)^{-1} \|\bn\|^{-(2-\alpha)} \bP(\bX=\bi)
\leq C'_{\bt}  \bP(\bX=\bi) \,,
\]
where the last inequality comes from Proposition~\ref{thm:renouv} (taken from~\cite{Will68}).
Using this, we obtain
\[
\begin{split}
\bP^{\otimes 2} &\big( |\btau\cap \btau' \cap \bT | > \tfrac k2  \,\big|\, n\bt \in \btau\cap \btau'\big)  \\
&  = \sum_{\bi,\bj \in \bT} \bP^{\otimes 2} \big( |\btau\cap \btau' \cap \bT | > \tfrac k2  \,\big|\,  \bX=\bi, \bX'=\bj\big) 
 \bP(\bX =\bi \mid n\bt \in \btau)  \bP(\bX' =\bj \mid n\bt \in \btau') \\
 & \leq  (C'_{\bt})^2 \sum_{\bi,\bj \in \bT} \bP^{\otimes 2} \big( |\btau\cap \btau' \cap \bT | > \tfrac k2  \mid  \bX=\bi, \bX'=\bj\big) 
 \bP(\bX =\bi )  \bP(\bX' =\bj ) \\
 & \leq (C'_{\bt})^2\bP^{\otimes 2} \big( |\btau\cap \btau' \cap \bT | > \tfrac k2 \big) \,.
\end{split}
\]
This proves the first inequality in the lemma.
Since it is proven in~\cite[Prop.~A.3]{BGK20} that the bivariate renewal $\btau\cap \btau'$ is terminating for $\alpha\in (0,1)$, 
we get that $|\btau\cap \btau'|$ is a geometric random variable, which concludes the proof.
\end{proof}

\section{Integration with a covariance measure: general theory}\label{app:intstoch}
This appendix provides more details on covariance measures and how they define stochastic integrals in Section~\ref{sec:integration}. 
Recall briefly the definitions of $\cS_d$, $L^2_\nu$ and $\la \cdot,\cdot\ra_\nu$ from Section~\ref{sec:integration}, where $\nu$ is the covariance measure of some field $X$. 
In this section we establish some of their properties, and use them to define the integral against $X$. 
We divide this appendix into three parts: 
\begin{itemize}
\item First, we claim that any $L^2(\bbP)$ random field on $\cS_d$ can be uniquely extended to bounded sets of $\Bor(\R^d)$ (in particular this defines the integrals of indicator functions of bounded Borel sets against $X$);
\item Second, we claim in Proposition~\ref{prop:propcovmeasure} that $L^2_\nu$ is (almost) an inner product vector space; from this, the definition of the integral against $X$ can be extended to $L^2_\nu$, yielding Theorem~\ref{thm:stocint};
\item Finally, we provide sufficient conditions for a generic correlation function on $\cS_{k}$, $k\in\N$ to be extended to a full $\gs$-additive measure on $\Bor(\R^{k})$. In particular this implies the well-posedness of $\nu_\cM$ in Proposition~\ref{prop:covariancemeasure}.
\end{itemize}

\subsection{Extension to indicator functions  of bounded Borel sets}
Recall the definition of the collection of rectangles of $\bbR^d$:
\[
\cS_d\;:=\; \big\{[\bu,\bv) \subset \bbR^d\,;\, \bu\preceq\bv\big\}\cup\{\emptyset\}\,,
\]
Notice that $\cS_d$ is a semi-ring: it is non-empty, stable by finite intersection and for $A,B\in\cS_d$, $A\setminus B$ is a finite union of disjoint elements of $\cS_d$. Also, we have $\gs(\cS_d) = \Bor(\bbR^d)$.

\begin{definition}
We call \emph{(additive) $L^2(\bbP)$-random field} on $\cS_d$ any family $X$ of random variables $X(A)\in L^2(\bbP)$, $A\in\cS_d$, such that for $A,B\in \cS_d$ with $A\cup B\in\cS_d$, $A\cap B=\emptyset$, one has $X(A\cup B)=X(A)+X(B)$ $\bbP$-a.s. 
With a slight abuse of terminology we also call it a random field on $\R^d$. 
\end{definition}

\begin{remark}
Any random function generates an additive random field via its increments, recall~\eqref{eq:stocint:defvariation}; but some fields are not constructed by pointwise-defined functions, as for instance some white noises. Most of the upcoming statements hold for generic random fields, thus we do not distinguish notation between increments of functions and fields.
\end{remark}

%
%

Let $\cR_d$ be the ring of sets generated by $\cS_d$: it is given by all finite unions of rectangles in $\cS_d$ (recall that a ring of sets is a non-empty class of sets, stable by finite union and difference). Since any set $\cR_d$ can be written as a disjoint union of elements of $\cS_d$, then any additive field $X$ on $\cS_d$ can naturally be extended to $\cR_d$ in a $\bbP$-a.s. unique way.

\begin{proposition}
Let $X:\cS_d\to L^2(\bbP)$ be an additive $L^2(\bbP)$-random field, which admits a $\sigma$-finite, non-negative covariance measure $\nu$ on $\Bor(\bbR^{d}\times\bbR^{d})\simeq \Bor(\bbR^{2d})$. Then $X$ admits a unique ``regular'' extension to bounded sets of $\Bor(\R^d)$.
\end{proposition}
Here, ``regular'' means that for a given bounded $B\in\Bor(\R^d)$ and a sequence $A_n\in\cR_d$, $n\geq1$ that converges to $B$ in the sense that
\[ \limsup_{n\to\infty} B\triangle A_n := \bigcap_{n\geq1}\bigcup_{k\geq n} (B\triangle A_k) = \emptyset, \]
(where $\triangle$ denotes the symmetric difference), then one has $\lim_{n\to\infty}X(A_n)=X(B)$ in $L^2(\bbP)$. Also, ``unique'' means that for any two such extensions $\tilde X$, $\hat X$ then for any bounded $B\in\Bor(\R^d)$ one has $\tilde X(B)=\hat X(B)$ $\bbP$-a.s. 

\begin{proof}The core of the proof relies on the following two observations.

\noindent\emph{---Fact 1.}  
Notice that any $C\in\cR_d$ can be written as a disjoint union of rectangles $\cup_{i=1}^kC_i$, so we can define
\begin{equation}\label{eq:thmstocint:extring}
\ind_C\diamond X \;:=\; \sum_{i=1}^k(\ind_{C_i}\diamond X) \;=\; \sum_{i=1}^k X(C_i)\;=:\; X(C)\;,
\end{equation}
which does not depend on the chosen decomposition (recall \eqref{eq:stocint:rectdecomp}). Moreover, the isometry relation~\eqref{eq:thmstocint:covarbis} clearly applies to $\ind_A, \ind_{B}$ with $A,B\in\cR_d$ by bilinearity, that is
\[
\bbE[(\ind_A\diamond X)(\ind_B\diamond X)]\;=\;\la \ind_A,\ind_B\ra_\nu\,,\qquad A,B\in\cR_d\,,
\]
(recall that the $\la\cdot,\cdot\ra_\nu$ is defined in~\eqref{eq:deflaranu}). Furthermore, the definition of $(\cdot)\diamond X$ can also be extended linearly to expressions $\sum_{i=1}^k a_i \ind_{A_i}$ with $A_i\in\cR_d$, $a_i\in\R$, $1\leq i\leq k$, and by bilinearity they still satisfy the isometry relation.

\noindent\emph{---Fact 2.} We have the following lemma, which is proven afterwards.
\begin{lemma}\label{lem:stocint:difsym}
Let $(E,\cE,\mu)$ be a measured space such that $\mu(E)<\infty$ and let $\cR\subset\cP(E)$ be a non-empty ring of sets such that $\sigma(\cR)=\cE$. Then for any $B\in\cE$ and $n\in\N$, there exists $A_n\in\cR$ such that
\begin{equation}\label{eq:lemstocint:difsym}
\mu(B\triangle A_n) \leq 2^{-n}\,,
\end{equation}
where $\triangle$ denotes the symmetric difference.
\end{lemma}

With these two observations at hand the proof is direct. Let $B$ be a bounded set of $\Bor(\bbR^d)$: in particular, setting $\Lambda_m := [-m,m)^d$,
we have that $B\in\Bor( \Lambda_m)$ for some $m>0$.
Let $(A_n)_{n\geq1}$ be a sequence of elements of $\cR_d\cap \Bor( \Lambda_m)$ satisfying \eqref{eq:lemstocint:difsym} for the finite measure $\mu_m(A):=\nu(A\times \Lambda_m)$ on $\Bor(\Lambda_m)$. Then $X(A_n)=\ind_{A_n}\diamond X\in L^2(\bbP)$, $n\geq1$ is well defined, and for $p,q\geq1$,
\begin{align*}
\bbE\big[(X(A_p)-X(A_q))^2\big] \;&=\; \|\ind_{A_p}\diamond X - \ind_{A_q}\diamond X\|^2_{L^2} \;=\; \|\ind_{A_p}- \ind_{A_q}\|^2_{\nu} \\
&=\; \int_{\Lambda_m^2} (\ind_{A_p}-\ind_{A_q})(\bu)\times(\ind_{A_p}-\ind_{A_q})(\bv) \,\dd\nu(\bu,\bv)\\
&\leq\; \int_{\Lambda_m^2} \ind_{A_p\triangle A_q}(\bu) \,\dd\nu(\bu,\bv) \;=\; \mu_m(A_p\triangle A_q)\;,
\end{align*}
where we used $|\ind_{A_p}-\ind_{A_q}|=\ind_{A_p\triangle A_q} \leq 1$. Since $A_p\triangle A_q\subset (A_p\triangle B) \cup (B\triangle A_q)$, it follows that $\mu(A_p\triangle A_q)\leq 2^{-p}+2^{-q}$, thus $(X(A_n))_{n\geq1}$ is a Cauchy sequence in $L^2(\bbP)$. By completeness, it therefore has a limit that we denote $X(B)$ (or $\ind_B\diamond X$), which does not depend on the chosen sequence $(A_n)_{n\geq1}$ that verifies $\lim_{n\to\infty}\mu(B\triangle A_n)=0$ (this is obtained with the same computation as above).
\end{proof}

\begin{proof}[Proof of Lemma~\ref{lem:stocint:difsym}]
Let us define for $n\in\N$,
\[\cA_n\;:=\; \big\{B\in \cE\,;\, \exists A\in\cR,\, \mu(B\triangle A)\leq 2^{-n} \big\}\,,\]
and $\cA=\cap_{n\geq1}\cA_n$. It is clear that $\cR\subset\cA_n$ for all $n\in\N$, so $\cR\subset\cA$. Let us prove that $\cA$ is a Dynkin system, which will conclude the proof by Dynkin's $\pi$-$\gl$ theorem.

First, we clearly have $\cR\subset\cA$, so $\cA$ is non-empty. Let $B_1,B_2\in\cA$ such that $B_1\subset B_2$, and $n\in\N$. By assumption there exist $A_1,A_2\in\cR$ such that $\mu(B_a\triangle A_a)\leq 2^{-n-1}$, $a\in\{1,2\}$. Since $\cR$ is a ring, we have $A_2\setminus A_1\in\cR$, and we also notice
\[(B_2\setminus B_1) \,\triangle\, (A_2\setminus A_1) \;\subset\; (B_2 \triangle A_2)\cup (B_1 \triangle A_1)\,.\]
Hence $\mu((B_2\setminus B_1) \triangle (A_2\setminus A_1))\leq 2^{-n}$ and $B_2\setminus B_1\in\cA_n$ for all $n\in\N$, thus $\cA$ is stable by difference.

Let $B_k\in\cA$ such that $B_k\subset B_{k+1}$, $k\geq1$, and let $B=\cup_{k\geq1}B_k$. Let $n\in\N$. 
Notice that, since $\mu(B)<\infty$, there exists $k_0\in\N$ such that $\mu(B\setminus B_{k_0})\leq 2^{-n-1}$. Moreover $B_{k_0}\in\cA$, so there exists $A\in\cR$ such that $\mu(B_{k_0}\triangle A)\leq 2^{-n-1}$. Thus,
\[\mu(B\triangle A) \;\leq\; \mu(B\setminus B_{k_0}) + \mu(B_{k_0}\triangle A) \;\leq\; 2^{-n}\,, \]
which finishes the proof.
\end{proof}

Therefore, for $X:\cS_d\to L^2(\bbP)$ an additive $L^2(\bbP)$-random field and $E\in\Bor(\R^d)$, the integral,
\[\int \ind_E \,\dd X\;=\; \ind_E \diamond X\;:=\; X(E)\;,\]
is well-defined. Moreover, the ``regularity'' of the extension also implies that for any $E,F\in\Bor(\R^d)$, their indicator functions satisfy the isometry relation~\eqref{eq:thmstocint:covarbis}, by dominated convergence and bilinearity.

\subsection{Properties of the vector space $L^2_\nu$} We have the following
\begin{proposition}\label{prop:propcovmeasure}
Let $X:\cS_d\to L^2(\bbP)$ be a random field which admits a $\sigma$-finite, non-negative covariance measure $\nu$ on $\Bor(\bbR^d \times \bbR^d)$.\begin{enumerate}[label=(\roman*)]
\item The set $L^2_\nu$ is a vector space. Moreover $\la g,h\ra_\nu$ is well-posed for $g,h\in L^2_\nu$.
\item The application $\la \cdot,\cdot\ra_\nu$ is bilinear, symmetric and semi-definite positive. In particular $\| g\|_\nu$ is well-posed for $g\in L^2_\nu$.
\item \emph{(Cauchy-Schwarz)} For $g,h\in L^2_\nu$, one has $\la g,h\ra_\nu \leq \| g\|_\nu\| h\|_\nu$.
\item \emph{(Triangular inequality)} For $g,h\in L^2_\nu$, one has $ \| g+h\|_\nu \leq \| g\|_\nu + \| h\|_\nu$.
\end{enumerate}
\end{proposition}

\begin{remark}\label{rem:notscalarprod}
Let us stress that $\la \cdot,\cdot\ra_\nu$ is not, in general, a scalar product on $L^2_\nu$ (in particular $\| \cdot\|_\nu$ is not a norm). Recall the expression of $\nu_\cM$ from \eqref{eq:stocint:exprint2}
and let $h:\bbR_+\to\R$ be defined by
\[h=\ind_{[0,1/2)\times[0,1/2)} - \ind_{[1/2,1)\times[0,1/2)} - \ind_{[0,1/2)\times[1/2,1)} + \ind_{[1/2,1)\times[1/2,1)}\,,\]
we have $\|\pt|h|\pt\|_{\nu_\cM} =\|\ind_{[\bzero,\bone)}\|_{\nu_\cM}= 2<\infty$ so $h\in L^2_{\nu_\cM}$; however $\|h\|_{\nu_\cM}=0$, and more generally $\la g,h\ra_{\nu_\cM}=0$ for all $g\in L^2_{\nu_\cM}$.

On the other hand, with the Gaussian white noise $W$ on $\bbR^2$, for any functions $g,h:\bbR^2\to\R$ we have   $\la g,h\ra_{\nu_W}^2 \;=\; \int_{\bbR^2} g(\bu)h(\bu) \pt \dd\gl_2(\bu)$
(see Remark~\ref{rem:whitenoise}). This proves that $\|g\|_{\nu_W}=0$ if and only if $g=0$ $\gl_2$-a.e.
\end{remark}

\begin{proof}
Consider the application
\[(g,h) \;\mapsto\; \la g,h\ra_\nu \;=\;\int_{\bbR^d} g(\bu)h(\bv)\pt\dd\nu(\bu,\bv)\;,\]
which is well-defined (possibly infinite) on non-negative, measurable functions. We claim that it is bilinear, symmetric and positive semi-definite: 
indeed, we have proven in the previous section that the isometry \eqref{eq:thmstocint:covarbis} is satisfied for indicator functions of bounded Borel sets, 
which implies those properties; and they can be extended to non-negative measurable functions on $\R^d$ by monotone convergence.

In order to prove that $L^2_\nu$ is a vector space (notice that it is not straightforward from the definition), 
we first have to prove a Cauchy-Schwarz inequality for non-negative functions: for $g,h$ non-negative measurable functions on $\bbR^d$, we claim that:
\begin{equation}\label{eq:CSstocint}
\la g,h\ra_\nu\;\leq\; \|g\|_\nu \|h\|_\nu\,, \qquad \forall \; g,h \geq 0\;.
\end{equation}
To show this, let us define $\cG$ (resp.\ $\cG_+$) the set of finite linear combinations of indicator functions $g=\sum_{i=1}^n c_i\ind_{A_i}$ of bounded Borel sets $A_i\in\Bor(\bbR^d)$, $c_i\in\R$, $1\leq i\leq n$ (resp.\ with $c_i\geq0$). Notice that, using the bilinearity of $\la \cdot,\cdot\ra_\nu$, the application $g\mapsto g\diamond X$ can be extended to $\cG$ and satisfies the isometry property \eqref{eq:thmstocint:covarbis} for all $g,h\in\cG$. With those observations in mind let $g, h$ be measurable non-negative functions, which we can write as $g=\sum_{i\geq1} c_i \ind_{A_i}$ and $h=\sum_{j\geq1} d_j \ind_{B_j}$ for some $c_i,d_j\geq0$, $A_i, B_j\in\Bor(\bbR^d)$ bounded, $i,j\geq1$. For $n\in\N$, define $g_n=\sum_{i=1}^n c_i \ind_{A_i}$ and $h_n=\sum_{j=1}^n d_j \ind_{B_j}$, so that $g_n,h_n\in\cG_+$. Therefore, using \eqref{eq:thmstocint:covarbis} on $\cG$, Cauchy--Schwarz inequality on $L^2(\bbP)$, then~\eqref{eq:thmstocint:covarbis} again, we obtain
\[ \la g_n,h_n\ra_\nu \,=\, \bbE\big[(g_n\diamond X)(h_n\diamond X)\big] \,\leq\, \bbE\big[(g_n\diamond X)^2\big]^{1/2} \bbE\big[(h_n\diamond X)^2\big]^{1/2}\,=\, \|g_n\|_\nu \|h_n\|_\nu \,. \]
We conclude the proof of \eqref{eq:CSstocint} by monotone convergence, letting $n\to\infty$.

Therefore, for $g,h\in L^2_\nu$, the inequality \eqref{eq:CSstocint} implies that $\la |g|,|h|\ra_\nu<\infty$ and thus $g+h\in L^2_\nu$, which proves that $L^2_\nu$ is a vector space. Moreover, the observation that $\la \cdot,\cdot\ra_\nu$ is bilinear, symmetric and positive semi-definite can also be extended on $L^2_\nu$ using \eqref{eq:thmstocint:covarbis} on $\cG$ and dominated convergence. This fully proves Proposition~\ref{prop:propcovmeasure}-($i$)-($ii$).

Regarding the Cauchy-Schwarz inequality, we may extend it from non-negative functions (recall \eqref{eq:CSstocint}) to~$L^2_\nu$:  
indeed, write for $g,h\in L^2_\nu$ and $\gl\in\R$,
\[0 \;\leq\; \| g+\gl h \|_\nu^2 \;=\; \|g\|_\nu^2 + 2\gl\la g,h\ra_\nu + \gl^2 \|h\|_\nu^2\;, \]
and since $\la g,h\ra_\nu\in\R$ is well-posed, we deduce that the quadratic polynomial in $\gl$ above has a non-positive discriminant, which concludes the proof of $(iii)$.

Finally, we deduce that a triangle inequality $(iv)$ holds for $\|\cdot\|_{\nu}$: for $g,h\in L^2_\nu$, write
\begin{equation}\label{eq:triangleineq:nu}
\|g+h\|_{\nu}^2 \;=\; \|g\|_\nu^2 + 2\la g,h\ra_\nu +  \|h\|_\nu^2 \;\leq\; \|g\|_\nu^2 + 2\|g\|_\nu\|h\|_\nu +  \|h\|_\nu^2 \;=\; (\|g\|_{\nu}+\|h\|_{\nu})^2, 
\end{equation}
which proves the inequality.
\end{proof}

With Proposition~\ref{prop:propcovmeasure} at hand, Theorem~\ref{thm:stocint} follows from a completeness argument. This finishes the construction of the stochastic integral against fields that admit a covariance measure.

\begin{proof}[Proof of Theorem~\ref{thm:stocint}]
Recall that $\cG$ is the set of finite linear combinations of indicator functions of bounded Borel sets, that $g\mapsto g\diamond X$ is well-posed on~$\cG$ and satisfies \eqref{eq:thmstocint:covarbis}. Finally, notice that $\cG$ is dense in $L^2_\nu$. Therefore, our goal is to extend the integral from~$\cG$ to $L^2_\nu$ by completeness; however this is not straightforward since $(L^2_\nu, \|\cdot\|_\nu)$ is not actually a normed space (because $\|\cdot\|_\nu$ is not a genuine norm, recall Remark~\ref{rem:notscalarprod}). We circumvent this difficulty with a quotient of vector spaces. 

Define (with an abuse of notation)
\begin{equation}\label{eq:defKerX}
\mathrm{Ker}(\nu) \;:=\; \{g\in L^2_\nu\,;\, \|g\|_\nu=0\}=\; \{g\in L^2_\nu\,;\, \forall \, h\in L^2_\nu,\, \la g,h\ra_\nu=0\} \;.
\end{equation}
The equality in \eqref{eq:defKerX} is a direct consequence of the Cauchy-Schwarz inequality: for $g\in L^2_\nu$ such that $\|g\|_\nu=0$ and $h\in L^2_\nu$, one has $\la g,h \ra_\nu \leq \|g\|_\nu\|h\|_\nu = 0$ and $-\la g,h \ra_\nu \leq \|g\|_\nu\pt\|-h\|_\nu = 0$, so $\la g,h \ra_\nu=0$. In particular, $\mathrm{Ker}(\nu)$ is a linear subspace of $L^2_\nu$.


For $g\in L^2_\nu$, let us denote $\ol g$ its equivalence class in $L^2_\nu/\mathrm{Ker}(\nu)$. It is clear that for $\ol g$, $\ol h\in L^2_\nu/\mathrm{Ker}(\nu)$ and any representatives $g_1$, $g_2\in \ol g$ and $h_1$, $h_2 \in \ol h$, then we have $\la g_1,h_1\ra_\nu=\la g_2,h_2\ra_\nu$. Therefore $\la\cdot,\cdot\ra_\nu$ can be defined on $L^2_\nu/\mathrm{Ker}(\nu)$, on which it is a scalar product; in particular $\|\cdot\|_\nu$ is a well-defined norm on $L^2_\nu/\mathrm{Ker}(\nu)$.

Recall that the integral is well-posed on $\cG$ and satisfies \eqref{eq:thmstocint:covarbis}. 
For any $g\in\cG$, we may define $\ol g \diamond X := g\diamond X$, which is well-defined almost everywhere on $(\gO,\cF,\bbP)$: indeed, for any two representatives $g_1$, $g_2\in\ol g$,
\[\|g_1\diamond X-g_2\diamond X\|^2_{L^2} \;=\; \|g_1-g_2\|^2_\nu \;=\; 0\;,\]
so $g_1\diamond X=g_2\diamond X$ almost surely. Therefore, the application $\ol g\mapsto \ol g\diamond X$ is an isometry from the normed space $(\cG/\mathrm{Ker}(\nu),\|\cdot\|_\nu)$ to $L^2(\bbP)$, hence it can be extended to the completion of $(\cG/\mathrm{Ker}(\nu),\|\cdot\|_\nu)$ which is $(L^2_\nu/\mathrm{Ker}(\nu),\|\cdot\|_\nu)$. Finally for $g\in L^2_\nu$, define $g\diamond X:= \ol g\diamond X$, which satisfies \eqref{eq:thmstocint:covarbis} and thus concludes the proof of Theorem~\ref{thm:stocint}.
\end{proof}

\subsection{A sufficient condition for the well-posedness of the covariance measure}
Let us stress that, if $\nu$ defines a non-negative, $\sigma$-additive function on $\cS^d$, then it does not always admit an extension into a measure on $\Bor(\bbR^d)$. In the following proposition, we provide a simple sufficient condition for the existence of such an extension.

On the one hand, if the field $X$ admits some negative correlations, then the construction has to be adapted to signed measures (which should not prove too difficult conceptually). On the other hand, it is easy to construct fields with non-$\sigma$-additive correlations.

\begin{example}\label{rem:discontinuousfield}
Consider the case $d=2$ and the deterministic field $X(u_1,u_2)=\ind_{\{u_1+u_2\geq0\}}$, $(u_1,u_2)\in\R^2$. For any point of the line $\{(x,-x),x\in\R\}$, say $\bzero$ for simplicity, one can construct sequences of sets such as $A_n=[-2/n,1/n]^2$ and $B_n=[-1/n,2/n]^2$ which satisfy
\[\limsup_{n\to\infty} A_n = \limsup_{n\to\infty} B_n = \{\bzero\},\quad \text{and}\quad X(A_n)=2=-X(B_n)\,,\;\forall\, n\geq1\,.\]
Hence it is clear that neither $X$ or its covariance function can be extended to bounded Borel sets consistently. 
\end{example}

In dimension 1 this issue around discontinuities can be circumvented by allowing atoms in the covariance measure, but this isn't enough in higher dimensions (in the above example, the whole diagonal $\{(x,-x),x\in\R\}$ is singular). An interesting question would therefore be to find sufficient and/or necessary conditions on a random function or field $X$ for its covariance function to be $\sigma$-additive and effectively define a measure on $\Bor(\bbR^d)$.
The following result is a step in this direction.
It will be applied to functions defined on $\cS_d\times\cS_d\simeq \cS_{2d}$, but for the sake of generality we prove it for functions on $\cS_k$, $k\geq1$.

\begin{proposition}
\label{prop:dominationmeasure}
Let $\nu : \cS_{k}\to\R$ satisfy the following:
\begin{enumerate}[label=(\roman*)]
\item $\nu$ is non-negative on $\cS_k$;
\item $\nu$ is additive on $\cS_k$;
\item $\nu$ is $\gs$-finite on $\cS_k$: that is $\nu([-m,m)^{k})<\infty$ for all $m>0$;
\item there exists a measure $\mu$ on $\Bor(\bbR^k)$, $\sigma$-finite on $\cS_k$, such that for $A\in\cS_k$, $\nu(A)\leq \mu(A)$.
\end{enumerate}
Then $\nu$ can be extended to a $\sigma$-finite measure on $\Bor(\bbR^k)$, which is unique. In particular, if $\nu(A)=\mu(A)$ for all $A\in\cS_{k}$, then $\nu(A)=\mu(A)$ for all $A\in\Bor(\R^k)$.
\end{proposition}


Let us point out that for $X$ an additive $L^2(\bbP)$-random field on $\cS_d$ and $\nu(A\times B):=\bbE[X(A)X(B)]$, $A,B\in\cS_d$, then \emph{(ii)}, \emph{(iii)} automatically hold, and \emph{(i)} holds for fields $X$ with non-negative correlations. Finally, we prove in~\eqref{eq:stocint:exprint1}-\eqref{eq:stocint:nuleb} that $\nu_\cM$ satisfies assumption~\emph{(iv)}, so Proposition~\ref{prop:dominationmeasure} applies to $\nu_\cM$ and directly implies Proposition~\ref{prop:covariancemeasure}.

\begin{proof}[Proof of Proposition~\ref{prop:dominationmeasure}]
By assumption $\nu$ is additive on $\cS_k$ which is a semi-ring of sets and $\gs(\cS_k)=\Bor(\R^k)$, so in order to extend it into a measure with Carath\'eodory's theorem it only remains to show that it is $\gs$-additive on $\cS_k$ (and the extension will be unique because of assumption \emph{(iii)}). We do so under the additional assumption $\mu(\cS_k)<\infty$, then the general result follows by defining $\nu_m:=\nu(\cdot\cap[-m,m)^k)$, $\mu_m:=\mu(\cdot\cap[-m,m)^k)$ and letting $m\to\infty$.

First, let us prove that $\nu$ is non-decreasing on $\cS_k$: let $A_1,A_2\in\cS_k$ such that $A_1\subset A_2$. Since $\cS_k$ is a semi-ring, $A_2\setminus A_1=\cup_{i=1}^pB_i$ for some $p\in\N$ and disjoint $B_i\in\cS_k$, $1\leq i\leq p$. Using 
assumptions~\emph{(i)} and~\emph{(ii)}, we have
\[\nu(A_2)\;=\;\nu(A_1) + \sum_{i=1}^p\nu(B_i) \;\geq\;\nu(A_1)\;,\]
which proves the statement. Moreover, this can straightforwardly be extended to $A_1, A_2\in\cR(\cS_k)$ which is the ring generated by $\cS_k$ (\emph{i.e.}\ all finite unions of rectangles).

Now let $A_i\in\cS_k$, $i\in\N$, such that $A:=\bigcup_{i\geq1}A_i\in\cS_k$. We may assume that the $(A_i)_{i\geq1}$ are disjoint without loss of generality. For $n\in\N$, we have
\[
\nu\Big(\bigcup_{i\geq1} A_i \Big) \;\geq\; \nu\Big( \bigcup_{i=1}^{n} A_i \Big)\;=\; \sum_{i=1}^n \nu(A_i)\;,
\]
and taking the limit as $n\to\infty$, we obtain $\nu(\bigcup_{i\geq1} A_i) \geq \sum_{i\geq1} \nu(A_i)$. Let us now show that $\nu(\bigcup_{i\geq n}A_i)$ vanishes as $n\to\infty$, which will conclude the proof. For $n\in\N$, recall that $\bigcup_{i\geq1} A_i\in\cS_k$ and $A_i\in\cS_k$, $1\leq i\leq n-1$. Since $\cS_k$ is a semi-ring of sets, $\bigcup_{i\geq n}A_i$ can be written as a finite union $\bigcup_{j=1}^p B_j$ for some  disjoint $B_j\in\cS_k$, $1\leq j\leq p$. Thereby,
\[0\;\leq\;\nu\Big(\bigcup_{i\geq n} A_i \Big) \;=\; \sum_{j=1}^p\nu(B_j)\;\leq\; \sum_{j=1}^p\mu(B_j) \;=\; \mu\Big(\bigcup_{i\geq n} A_i \Big)\,,\]
where we used assumption~\emph{(iv)}. Since we assumed that $\mu$ is a finite measure,
\[
\lim_{n\to\infty} \mu\Big(\bigcup_{i\geq n} A_i \Big) \;=\; \mu\Big(\bigcap_{n\geq 1}\bigcup_{i\geq n} A_i \Big)\;=\; \mu(\emptyset)\;=\;0\;,
\]
which concludes the proof.
\end{proof}

\section{An example of distribution in \texorpdfstring{$\Pfk_4,\Pfk_8$}{P4,P8}}
\label{app:example}

Let us provide an example of distributions $\bbP$ and interaction function $V(x,y)$ tailored to obtain cases where $\Pfk_4,\Pfk_8$ are not empty. We let $V(x,y) =xf(y) +yf(x)$, where $f$ is determined below.
We choose $X:=\hat \go_i$ (and $Y:=\bar \go_i$) to be uniformly distributed in the set $E=\{\pm\sqrt{a},\pm\sqrt{b},\pm\sqrt{2-a},\pm\sqrt{2-b}\}$, where $0<a<b<1$ are two parameters we can play with.
Now, we choose a function $f:E\to E$ by setting 
\[
\begin{split}
f(\pm \sqrt{a}) &=\pm \sqrt{2-a}  \quad \text{ (the sign is the same)}, \\
f(\pm \sqrt{2-a}) &=\mp \sqrt{a} \quad \text{ (the sign is reversed)},\\ 
f(\pm \sqrt{b})& =\pm \sqrt{2-b} \quad \text{ (the sign is the same)}, \\
f(\pm \sqrt{2-b})& =\mp \sqrt{b} \quad \text{ (the sign is reversed)}.
\end{split}
\]
Observe then that for all $x\in E$ we have $x^2+f(x)^2=2$
and that $f(f(x))=-x$.
From this we get the following facts on $\go_{\bi} := V(X,Y) =X f(Y) +Yf(X)$:

\smallskip
\textbullet\ For any $k\geq 0$ we have $\bbE[ V(X,Y)^{2k+1} \,|\, X] =0$. This is due to the fact that all the terms $Y^{2k+1-j} f(Y)^{j}$, $0\leq j\leq 2k+1$ appearing in the binomial expansion of $V(X,Y)^k$ have mean zero. This is easily shown by induction on~$j$, using that that $\bbE[Y^{2k+1}]=0$, $\bbE[ Y^{2k} f(Y)]=0$ by a direct calculation and then reducing $j$ by using that $f(x)^2= 2-x^2$ when $j \geq 2$. 
This shows that $\bbP\notin\Pfk_r$ for any odd $r$.

\smallskip
\textbullet\ We have that $Y$ and $f(Y)$
have the same distribution and $\bbE[Y] =0$, $\bbE[Y^2]=1$. 
By a direct calculation, one finds that
\[
\bbE[V(X,Y)^2 \,|\, X] = X^2 + f(X)^2 =2 \,,
\]
where we also have used that  $\bbE[ Y f(Y)]=0$ (by a direct calculation or making use of the fact that $Y$ and $f(Y)$ have the same law and $f(f(Y))=-Y$).
We therefore get that $\bbP\notin \Pfk_2$.

\smallskip
\textbullet\
Now, by a direct calculation, we get that
\begin{equation*}
\bbE[V(X,Y)^4 \,|\, X] 
= \bbE[Y^4] \big( X^4 +f(X)^4 \big) + 6 X^2 f(X)^2 \bbE[Y^2 f(Y)^2] 
+4Xf(X) (f(X)^2 -X^2) \bbE[Y^3 f(Y)] \,,
\end{equation*}
using also that $\bbE[Y f(Y)^3] = - \bbE[Y^3 f(Y)]$,
since $Y$ and $f(Y)$ have the same law and $f(f(Y))=-Y$.
Note that $X^4+f(X)^4 = 4- 2X^2f(X)^2$ and $\bbE[Y^4]  = 2- \bbE[Y^2 f(Y)^2]$,
so that setting $\mu_0 =\bbE[Y^2 f(Y)^2]$ and $\mu_1 = \bbE[Y^3 f(Y)]$ we can write
\[
\bbE[V(X,Y)^4 \,|\, X] 
= 4(2-\mu_0)+ 4(2\mu_0-1) X^2f(X)^2 +4 \mu_1 Xf(X) (f(X)^2 -X^2)  \,.
\]
Then, one can notice that $X^2f(X)^2$ can only take two values,
$u:=a(2-a)$ and $v:=b(2-b)$, with equal probabilities (note that $0<u<v<1$).
In particular, we have 
\[
\mu_0:= \bbE[Y^2 f(Y)^2] = \frac12 (u+v) \,.
\]
Similarly, one can check that $Xf(X)(f(X)^2-X^2)$ takes only two values (with equal probabilities), $2\sqrt{a(2-a)} (1-a) = 2\sqrt{u(1-u)}$ and $2\sqrt{b(2-b)} (1-b) = 2\sqrt{v(1-v)}$.
Recalling that $\bbE[Y f(Y)^3] = - \bbE[Y^3 f(Y)]$ this gives 
that 
\[
\mu_1= \bbE[Y^3f(Y)] = -\frac12 (\sqrt{u(1-u)} +\sqrt{v(1-v)} ) \,.
\]
Overall, we find that $\bbE[V(X,Y)^4 \,|\, X]$
takes two values with equal probabilities:
\[
\begin{split}
A&= 4(2-\mu_0)+ 4(2\mu_0-1) u + 8 \mu_1 \sqrt{u(1-u)}  \,,\\
B&= 4(2-\mu_0)+ 4(2\mu_0-1) v + 8 \mu_1 \sqrt{v(1-v)}\,.
\end{split}
\]
It then remains to determine whether $A\neq B$.
We have
\[
\begin{split}
A-B& = 4 (2\mu_0-1) (u-v) + 8 \mu_1 (\sqrt{u(1-u)} - \sqrt{v(1-v)} ) \\
& =4(u+v-1) (u-v) - 4 (u(1-u) - v(1-v)) = 8(u-v)(u+v-1)
\,,
\end{split}
\]
where we have used the values of $\mu_0$ and $\mu_1$ above. 
If $u+v\neq 1$ then $A\neq B$,
showing that $\bbE[V(X,Y)^4 \,|\, X] $ is not a.s.\ constant and thus $\bbP\in \Pfk_4$.

\smallskip
\textbullet\ In the case where we have $u+v=1$ then the above shows that $\bbE[V(X,Y)^4 \,|\, X] $ is a.s.\ constant, so $\bbP\notin \Pfk_4$.
One can then carry on and determine whether $\bbP\in \Pfk_r$ for some $r\geq 6$.
We do not develop here the calculations, but one can actually check that in the case $u+v=1$ then
$\bbE[V(X,Y)^6 \,|\, X]$ is again a.s.\ constant: we have that $\bbP\notin\Pfk_6$.
On the other hand, 
$\bbE[V(X,Y)^8 \,|\, X]$ can be checked to be a.s.\ non-constant
(except for one value of $u$) 
showing that $\bbP\in \Pfk_8$ in that case.

\end{appendix}

\subsection*{Acknowledgements}
The authors are grateful for insightful comments from anonymous referees, which helped improve the presentation of the paper. 

The authors QB and AL acknowledge the support of the grant ANR-22-CE40-0012 (project Local). AL was also partially supported by the grant ANR-20-CE92-0010 (project Remeco).

%

\end{document}